\documentclass{article}
\usepackage[utf8]{inputenc}
\usepackage{epsfig,amsmath,amssymb,graphics,verbatim,amsfonts,subfigure,psfrag,amsthm}
\usepackage{amsthm,bm,bbm,enumerate, graphicx, mathtools,color}
\usepackage[letterpaper,margin=0.85in]{geometry}
\usepackage[english]{babel}
\usepackage{xcolor,comment}
\usepackage{float}
\usepackage{tikz}
\usepackage{caption}
\usepackage{subfig}
\usepackage[capitalise]{cleveref}

\usepackage{fancyhdr}

\newtheorem{theorem}{Theorem}[section]
\newtheorem{lem}[theorem]{Lemma}
\newtheorem{prop}[theorem]{Proposition}
\newtheorem{cor}[theorem]{Corollary}
\newtheorem{defn}[theorem]{Definition}

\newtheorem{rmk}[theorem]{Remark}

\newtheorem{assn}[theorem]{Assumption}

\newcommand{\pr}[1]{\mathbb{P}\!\left(#1\right)}
\newcommand{\E}[1]{\mathbb{E}\!\left[#1\right]}
\newcommand{\estart}[2]{\mathbb{E}_{#2}\!\left[#1\right]}
\newcommand{\prstart}[2]{\mathbb{P}_{#2}\!\left(#1\right)}
\newcommand{\prcond}[3]{\mathbb{P}_{#3}\!\left(#1\;\middle\vert\;#2\right)}

\newcommand{\prb}[1]{\mathbf{P}\!\left(#1\right)}
\newcommand{\Eb}[1]{\mathbf{E}\!\left[#1\right]}

\newcommand{\prcondb}[3]{\mathbf{P}_{#3}\!\left(#1\;\middle\vert\;#2\right)}

\newcommand{\p}[1]{{P}\!\left(#1\right)}
\newcommand{\pcond}[3]{{P}_{#3}\!\left(#1\;\middle\vert\;#2\right)}

\newcommand{\bPb}{\mathbf{P}}
\newcommand{\Pb}{\mathbb{P}}
\newcommand{\cPb}{\mathcal{P}}

\newcommand{\prc}[1]{\mathcal{P}\!\left(#1\right)}

\newcommand{\prcondc}[3]{\mathcal{P}_{#3}\!\left(#1\;\middle\vert\;#2\right)}

\definecolor{darkgreen}{rgb}{0.0, 0.6, 0.0}
\renewcommand{\epsilon}{\varepsilon}
\newcommand{\Ti}{T_{\infty}}

\newcommand{\Height}{\textsf{Height}}
\newcommand{\Binom}{\textsf{Binom}}

\newcommand{\Z}{\mathbb{Z}}
\newcommand{\R}{\mathbb{R}}

\newcommand{\F}{\mathcal{F}}
\newcommand{\gambadi}{\Gamma^{(1, \text{bad})}_{x,t}}
\newcommand{\gambadii}{\Gamma^{(2, \text{bad})}_{x,t}}
\parskip \medskipamount
\title{Parabolic Anderson model on critical Galton-Watson trees in a Pareto environment}
\author{Eleanor Archer \thanks{Department of Mathematical Sciences, Tel Aviv University, Tel Aviv 69978, Israel. Email: {eleanora@mail.tau.ac.il}.} \hspace{1cm} Anne Pein \thanks{Department of Meteorology, University of Reading, United Kingdom. Email: a.pein@reading.ac.uk}}
\date{\today}

\begin{document}

\maketitle

\begin{abstract}
The parabolic Anderson model is the heat equation with some extra spatial randomness. In this paper we consider the parabolic Anderson model with i.i.d. Pareto potential on a critical Galton-Watson tree conditioned to survive. We prove that the solution at time $t$ is concentrated at a single site with high probability and at two sites almost surely as $t \to \infty$. Moreover, we identify asymptotics for the localisation sites and the total mass, and show that the solution $u(t,v)$ at a vertex $v$ can be well-approximated by a certain functional of $v$. The main difference with earlier results on $\Z^d$ is that we have to incorporate the effect of variable vertex degrees within the tree, and make the role of the degrees precise.
\end{abstract}

%\tableofcontents

%%%%%%%%%%%%%%%%%%%%%%%%%%%%%%%%%%%%%%%%%
%%%%% Section: Introduction
%%%%%%%%%%%%%%%%%%%%%%%%%%%%%%%%%%%%%%%%%
\section{Introduction}
In this paper we study the parabolic Anderson model (PAM) on a critical Galton-Watson tree conditioned to survive. Given an infinite tree $\Ti$, with root $O$, the PAM on $\Ti$ is the heat equation on $\Ti$ with a random potential:
\begin{equation}\label{eqn:PAM}
\begin{aligned}
    \partial_t u(t,v)&=\Delta u(t,v)+\xi(v)u(t,v),  &(t,v)\in(0,\infty)\times V(\Ti),\\
    u(0,v)&=\mathbbm{1}\{v = O\}, &v\in V(\Ti),
\end{aligned}
\end{equation}
where $V(\Ti)$ denotes the set of vertices of $\Ti$.
Here, $\Delta$ is the discrete Laplacian on $\Ti$, defined by
\[
(\Delta f)(v)=\sum_{y\sim v}[f(y)-f(v)], \ \ \ v\in V(T_\infty), \ f:V(T_\infty)\rightarrow \mathbb R.
\]
One can also work under other initial conditions, such as $u(0,v)=1$ for all $v \in V(\Ti)$, but we are interested here in studying spread of particles from a single initial seed. We work under the \textit{quenched} law of $\Ti$, whereby we sample the tree $\Ti$ first, and then independently sample $(\xi(v):v\in V(\Ti))$ as a collection of independent identically distributed random variables, static in time.

The random operator $H:=\Delta+\xi$ is known as the Anderson Hamiltonian and was originally analysed by Anderson in 1958 in the context of random Schroedinger equations that model  electrons moving through a crystal with impurities \cite{Anderson1958Absence}. The parabolic analogue, the PAM, is now a classical model for mean-field spread of particles in an inhomogeneous random environment, with applications including population genetics, reaction kinetics and magnetism, to name but a few. See the book \cite{KoenigBook16} or the surveys \cite{Anderson1958Absence, Moerters2009survey, GaertnerKoenig2005survey} for more background.

The solution to the PAM is well-known to exhibit an \textit{intermittency} effect, whereby almost all the solution is asymptotically concentrated in a small number of disjoint regions, known as \textit{islands}. This effect has been observed by physicists for some time, but since the early 1990s there has been a concerted mathematical effort to prove this quantitatively, and on $\Z^d$ this effect is now well-understood. The size and number of islands depend on the tail decay of the potential, which can give rise to several regimes on the behaviour of the islands.

In this paper we will be exclusively focused on the case where the potential is Pareto distributed, i.e. with polynomially decaying tails, for which the intermittency effect is strongly pronounced. In this case, it was shown in \cite{KLMStwocities09} that the solution to the PAM on $\Z^d$ almost surely localises at two sites as time runs to infinity. In fact, the authors show that the solution localises at a single site with high probability, and identify the localisation site as the maximiser of a certain functional; this localisation site evolves in time, so that there will be a sequence of transition times when the solution ``jumps'' from one site to another, meaning that both sites contain non-negligible mass, and in particular that two-site localisation is the best result one can hope for in the almost sure setting.

The solution to the PAM on $\Z^d$ in this regime was also studied by Ortgiese and Roberts, who showed in \cite{OrtgieseRobertsIntermittency} that the solution can be well approximated by a mean-field ``lilypad model'' in which particles spread out from sites of high potential in a sequence of growing lilypads. Their main motivation was to study the behaviour of a branching random walk, for which the PAM describes the expected number of particles, but they show that the solution to the PAM at any given site corresponds roughly to the number of particles reaching that site from the centre of a nearby lilypad.

In this paper we prove similar results for the solution to the PAM on a critical Galton-Watson tree conditioned to survive. This extends the current literature since the underlying graph is now random, with fractal properties, and has non-uniform volume growth. One consequence of this is that the random walk appearing in the Feynman-Kac solution to \eqref{eqn:PAM} on $\Ti$ is known to be subdiffusive \cite{KestenIICtree, CroydonKumagai08}, in contrast to the $\Z^d$ case. Moreover, in contrast to most of the existing literature, we consider the PAM on a graph with unbounded degrees and we make the role of the degrees explicit.

The PAM has also been studied on supercritical Galton-Watson trees of unbounded degree in the recent article \cite{HollanderWang}. Here the authors consider the (very different) regime where the potential has double-exponential tails, and show that results of \cite{HollanderKoenigSantosPAMOnTrees} on the PAM profile in the bounded degree case can be extended to the case where the offspring distribution decays super-double-exponentially. (The PAM was also considered on sequences of complete graphs in \cite{FleischmannMolchanov1990exact} and hypercubes in \cite{AvenaGunHesse2020parabolic}, which both have unbounded degrees in the limit, but in these cases the Laplace operator was multiplied by an appropriate prefactor to cancel out the effect, which is not the case in the present paper). In our case, due to the heavier tails on the potential $\xi$, the localisation effect of the PAM is more pronounced and we will prove our results under weaker assumptions on the tail of the offspring distribution.

Before stating our main results, we give the setup. We suppose that $\Ti$ is an infinite Galton-Watson tree defined on a probability space $(\Omega, \F, \bPb)$ and let $d$ denote the fractal dimension of $\Ti$ (we will state $d$ explicitly in Corollary \ref{cor:volume bounds}; in particular $d=2$ a.s. when the offspring distribution of $\Ti$ has finite variance). Conditionally on $\Ti$, we fix some $\alpha > d$ (this restriction is both necessary and sufficient for a finite solution to \eqref{eqn:PAM} to exist, almost surely), and let $\left(\xi(v):v \in V(\Ti)\right)$ denote a sequence of i.i.d. random variables with law $\cPb$, satisfying
\[
\prc{\xi(v) \geq x} = x^{-\alpha}
\]
for all $x \geq 1$. This is known as the Pareto distribution with parameter $\alpha$. If $v \in V(\Ti)$, we call $\xi(v)$ the \textit{potential} at $v$. We let $P$ denote the law of the product measure $\bPb \times \cPb$.

Given vertices $u,v \in V(\Ti)$, we let $|u-v|$ denote the graph distance between them in $\Ti$, and $|v| = |v-O|$.

Given $\Ti$ and $\left(\xi(v): v\in V(\Ti)\right)$, we let $u(t,v)$ denote the solution to \eqref{eqn:PAM}, and let $U(t) = \sum_{v \in V(\Ti)} u(t,v)$ denote the total mass in the system at time $t$. Following the notation of \cite{KLMStwocities09}, we also set $q := \frac{d}{\alpha - d}$, and define the two functions
\[
r(t) := \left(\frac{t}{\log t} \right)^{q+1}, \hspace{2cm} a(t) := \left(\frac{t}{\log t} \right)^{q}.
\]
We will see later that, roughly speaking, the bulk of the solution at time $t$ is supported on a ball of radius $r(t)$, and the highest potential seen on this ball will be of order $a(t)$.

We first present some high probability bounds on the solution, starting with the total mass. We prove these under the following assumption.

\begin{assn}\label{assn:whp}
Assume the offspring distribution is critical and there exists $\beta \in (1,2]$ such that the offspring distribution is in the domain of normal attraction of a $\beta$-stable law.
\end{assn}

Recall that a critical offspring distribution $\mu$ is in the normal domain of attraction of a $\beta$-stable law if there exists a constant $c$ such that 
\begin{equation*}\label{eqn:domatt}
\frac{S_n-n}{cn^{\frac{1}{\beta}}}\overset{(d)}{\rightarrow} X 
\end{equation*}
as $n \to \infty$, where $S_n=\sum_{i=1}^nZ_i$ with $Z_i$ i.i.d. with law $\mu$, and $X$ is a $\beta$-stable random variable, i.e. such that $\mathbb E[\exp(-\lambda X)]=\exp(-\lambda^\beta)$ for all $\lambda>0$. In the finite variance case we have $\beta = 2$. One could also incorporate slowly-varying functions, but for simplicity we have not done this.

\begin{theorem}\label{thm:total mass asymp intro} Under Assumption \ref{assn:whp},
\[
\lim_{\lambda \to \infty} \liminf_{t \to \infty} P \left(\log U(t) \in [\lambda^{-1} ta(t), \lambda t a(t)]\right) = 1.
\]
\end{theorem}

This result follows by similar considerations to those in \cite{KLMStwocities09}.

In light of the behaviour of the Pareto-distributed PAM on $\Z^d$, the main focus of this paper is to prove a localisation result; in particular the following. In order to prove this we have to make slightly stronger assumptions on the offspring distribution of $\Ti$.

\begin{assn}\label{assn:as}
Assume the offspring distribution of $\Ti$ is critical and decays super-polynomially, i.e. for every $K< \infty$ it satisfies $\prb{\# \{\text{offspring}\} > x} = o(x^{-K})$ as $x \to \infty$.
\end{assn}

\begin{rmk}
\begin{enumerate}
\item In particular, Assumption \ref{assn:as} implies that the distribution has finite variance and is thus in the domain of attraction of a $\beta=2$ stable law (via the classical central limit theorem).
\item Assumption \ref{assn:as} is required as it allows us to control volumes of balls in $\Ti$ more tightly. This is useful in two places in particular. Firstly we use it to prove localisation on a smaller set $\Gamma_t$ (Lemma \ref{lem:loceigenfunction}) for which we require that the localisation site $\hat{Z}_t^{(1)}$ is also the maximiser of $\xi (v) - \deg v$ on $\Gamma_t$, which in turn requires that $\Gamma_t$ is not too large. Secondly we use it when summing over paths that avoid $\hat{Z}_t^{(1)}$ to show that their contribution to the total mass is small (Propositions \ref{prop:U2 zero} and \ref{prop:U4 zero}), for which we require that the number of such paths within a ball is not too large.
\item The assumption on the offspring distribution in Assumption \ref{assn:as} is not too restrictive since it is satisfied by most offspring distributions for which $\Ti$ is used as a local proxy for structures in statistical physics (for example the \textsf{Poi}(1) offspring distribution which gives a good approximation to critical percolation clusters \cite{Aldous1997brownian} and uniform spanning trees \cite{NachmiasPeres2020local} in high dimensions).
\item Assumption \ref{assn:as} is somewhat reminiscent of \cite[Assumption 1.3]{HollanderWang}, in that we require super-polynomial tail decay to deal straightforwardly with polynomial tails for the potential, and in \cite{HollanderWang} the authors require super-double-exponential tail decay in their main result to deal straightforwardly with double-exponential tails for the potential. In fact, we could slightly weaken our assumption to the requirement that the tails are $O(x^{-K_{\alpha, \beta}})$ for some $K_{\alpha, \beta} < \infty$, but it was not possible to bound $K_{\alpha, \beta}$ uniformly in $\alpha$ and $\beta$.
\end{enumerate}
\end{rmk}

\begin{theorem}\label{thm:main a.s. localisation intro}
Let Assumption \ref{assn:as} hold. There exist processes $(\hat{Z}^{(1)}_t)_{t\geq 0}$ and $(\hat Z^{(2)}_t)_{t\geq 0}$ taking values in $V(T_\infty)$ such that, $P$-almost surely as $t \to \infty$,
\begin{equation}
    \frac{u(t, \hat{Z}^{(1)}_t) + u(t, \hat{Z}^{(2)}_t)}{U(t)}\to 1.
\end{equation}
\end{theorem}

As remarked above, two-site localisation is the best result we can hope for in the almost sure setting, due to the existence of transition times between the best sites. A consequence of the proof will be the fact that we actually have one-site localisation with high probability, as outlined in the following theorem.

\begin{theorem}\label{thm:main whp localisation intro} Let Assumption \ref{assn:whp} hold. There exists a process $(\hat{Z}^{(1)}_t)_{t\geq 0}$ taking values in $V(\Ti)$ such that
\begin{equation}
    \frac{u(t,\hat{Z}^{(1)}_t)}{U(t)}\to 1 \text{ in $P$-probability as }t\to \infty.
\end{equation}
Under Assumption \ref{assn:as}, $\hat{Z}_t^{(1)}$ is the same as in Theorem \ref{thm:main a.s. localisation intro}.
\end{theorem}

The main tool in our analysis is the Feynman-Kac formula, which allows us to represent the solution to \eqref{eqn:PAM} as
\begin{align}\label{eqn:FK formula}
    u(t,v) = \mathbb E_0\left[\exp\left\{\int_0^t\xi(X_s)~\text{d} s\right\}\mathbbm 1\{X_t=v\}\right]
\end{align}
for all $v \in V(\Ti)$, where $\left(X_s\right)_{s \geq 0}$ is a continuous time variable speed random walk on $\Ti$ with law $\mathbb P$ and we denote the corresponding expectation as $\mathbb E$. It is clear from the formula that random walk trajectories spending time at sites of higher potential will contribute more to the total mass, so to find the best trajectory we have to balance the high potential enjoyed at good sites against the cost of forcing the random walk to visit and then stay at that site.

Our strategy is broadly the same as that used to prove similar results for the PAM on $\Z^d$ in \cite{KLMStwocities09}, but the non-uniformity of the underlying graph presents extra challenges which take some effort to overcome. In particular, due to variable degree on $\Ti$ the random walk estimates are less precise, and the effect of the degree cannot be neglected when comparing the solution at different sites, since this quantifies the cost of forcing a random walk to remain at a given site. An important part of the localisation proof will be to prove an analogous localisation result for the principal eigenfunction of the Anderson Hamiltonian operator, and here in particular we have to work a lot harder to obtain the desired result. In addition, fluctuations in the tree $\Ti$ will lead to extra fluctuations in the solution of the PAM and the location of the concentration site, whereas on $\Z^d$ the behaviour of these quantities is determined solely by the realisation of the random potential.

In order to balance the potential at a site and the cost of reaching and then staying at it, we define the functional (cf \cite[Equation (1.9)]{KLMStwocities09}):
\begin{align}\label{eqn:psit def}
\begin{split}
\psi_t(v) &= \sup_{\rho \in [0,1]} \left\{ (1-\rho)(\xi(v)-\deg(v))-\frac{|v|}{t} \log \left(\frac{|v|}{\rho et}\right) \right\}  \\
&\geq \left[\xi(v)-\deg(v)-\frac{|v|}{t} \log( \xi(v)-\deg(v)) \right]\mathbbm{1}\{t(\xi(v)-\deg(v)) \geq |v|\}.
\end{split}
\end{align}
This can be understood as follows: $e^{t(1-\rho)\xi(v)}$ is the contribution to the exponential term in the Feynman-Kac formula \eqref{eqn:FK formula} if the random walk remains at site $v$ for a time duration $(1-\rho)t$, then $e^{t(1-\rho)\deg(v)}$ is the cost of forcing the random walk to remain at site $v$ for a time of $(1-\rho)t$ after reaching it, and we will see in Proposition \ref{prop:hitting time prob bounds} that $e^{-|v| \log \left(\frac{|v|}{\rho t}\right)}$ is roughly the probability that the random walk hits $v$ by time $\rho t$. At the best sites, it turns out that $u(t,v)$ is well-approximated by $e^{t\psi_t(v)}$.

The condition $t(\xi(v)-\deg(v)) \geq |v|$ ensures that the optimal value of $\rho = \frac{|v|}{t(\xi(v)-\deg(v))}$ in \eqref{eqn:psit def} is indeed in $[0,1]$. A similar functional for the $\Z^d$ case is defined in \cite[Equation (1.9)]{KLMStwocities09}, though they do not need to incorporate the degree term there. We define the processes $\hat Z_t^{(1)}$ and $\hat Z_t^{(2)}$ by
 \begin{equation}\label{eqn:Zt def intro}
     \hat Z_t^{(1)}:=\arg \max_{v \in V(T_\infty)} \psi_t(v), \hspace{1cm} \hat Z_t^{(2)}:=\arg \max_{v \in V(T_\infty) \setminus \{\hat Z^{(1)}_t\}} \psi_t(v).
 \end{equation}

We will see in Section \ref{sec:concentration site} that $\hat Z_t^{(1)}$ and $\hat Z_t^{(2)}$ are well-defined and in particular enjoy the following asymptotic.

\begin{theorem}\label{thm:Zt asymp intro} Let Assumption \ref{assn:whp} hold.
For each $i=1,2$, we have that
\[
\lim_{\lambda \to \infty} \liminf_{t \to \infty} P \left(\hat Z_t^{(i)} \in [\lambda^{-1} r(t), \lambda r(t)]\right) = 1.
\]
\end{theorem}

 Our final result approximates the solution and follows by similar arguments to \cite[Theorem 1.4]{OrtgieseRobertsIntermittency}, though extra care is needed to deal with the unbounded degrees in our case. Analogously to \cite{OrtgieseRobertsIntermittency}, to state the result, for $v \in V(\Ti)$ we define 
\begin{equation}\label{eqn:lambda intro def}
\lambda(t,v) := \sup_{y \in V(\Ti)} \left\{ t(\xi(y) - \deg y) - |y| \log \left( \frac{|y|}{te}\right) -  |v-y| \log \left(\frac{|v-y|}{te}\right) \right\},
\end{equation}
and $\lambda_+(t,v) = \lambda(t,v) \vee 0$. Note the similarity with the terms in $\psi_t$ in \eqref{eqn:psit def}.

\begin{theorem}\label{thm:log conv result intro}
Let Assumption \ref{assn:whp} hold. With high P-probability,
\[
\left( ta(t) \right)^{-1} \sup_{v \in V(\Ti)} |\log_+ (u(t,v)) - \lambda_+(t,v)| \to 0
\]
\end{theorem}
The intuition here is that the vertex $y$ achieving the supremum in \eqref{eqn:lambda intro def} represents the ``best'' vertex that has the right balance between being close to $v$, and having high potential. The main contribution to the solution at $v$ comes from particles that spend most of the time interval $[0,t]$ at site $y$ where they reproduce at rate $\xi(y)$, and then dart to $v$ just before time $t$. The total number of particles obtained in this way is roughly equal to $e^{t \xi(y)}$, and the cost of darting quickly from $O$ to $y$ and then quickly from $y$ to $v$ is roughly equal to the exponential of the subtracted terms in \eqref{eqn:lambda intro def}.

We note that in the case of the PAM on $\Z^d$, the authors in \cite{KLMStwocities09} obtain stronger results than those in Theorems \ref{thm:total mass asymp intro} and \ref{thm:Zt asymp intro} by identifying the scaling limit of the relevant quantities. This is achieved by showing that the scaling limit of the random potential field $\left(\xi(z):z \in \Z^d\right)$ can be represented by a Poisson point process in $\R^d$ and then interpreting the other scaling limits in terms of this point process. Since Gromov-Hausdorff-Prokhorov scaling limits of critical Galton-Watson trees are well understood, we anticipate that a similar approach could be employed in our setting, and in this case the limiting Poisson point process would have a similar density but with respect to the canonical mass measure on the continuum random sin-trees of Aldous \cite{AldousCRTII} (in the finite variance case) and Duquesne \cite{DuqSinTree} (in the $\beta$-stable case). Due to space considerations, we have not attempted to prove this result in this article, but we believe this makes our proofs more robust since we have not relied on on any information about the limiting random variables and their densities, which is not the case in \cite{KLMStwocities09}.

Although the results of this paper are restricted to Galton-Watson trees and in certain proofs we take advantage of the tree structure, we anticipate that the PAM should exhibit similar behaviour on other critical random graphs, such as the Uniform Infinite Planar Triangulation, critical Erd\"os-R\'enyi graphs, and the critical configuration model. Moreover, critical Galton-Watson trees are conjectured and in many cases known to be good toy models for a range of stochastic structures such as high-dimensional critical percolation clusters and uniform spanning trees, and some classes of random planar maps, so we also expect that the PAM on these more complicated models will display similar behaviour.

This paper is organised as follows. In Section \ref{sec:criticalGWT} we visit the definition of a critical Galton-Watson tree conditioned to survive and we derive several estimates regarding its volume growth. Furthermore, we also give bounds for the random walk defined on $T_\infty$, which will be needed later on. In Section \ref{sec:Potential} we analyse extremal values of the potential on certain subsets of $T_\infty$. In Section \ref{sctn:PAM on Ti} we prove the asymptotic behaviour stated in Theorem \ref{thm:total mass asymp intro} and \ref{thm:Zt asymp intro}, and the approximation for the solution as given in Theorem \ref{thm:log conv result intro}. Subsequently, we are concerned with proving the localisation result. In particular, in Section \ref{sec:concentration site} we derive several estimates for the first maximisers of the functional defined in \eqref{eqn:psit def} and the gap between the largest values of this functional. Section \ref{sec:spectralresults} contains the spectral analysis that we mentioned above, that is, we prove that the support of the principal eigenfunction of the Anderson Hamiltonian restricted to certain time-dependent subsets is strongly concentrated as time goes to infinity. Finally, these auxiliary results are used to prove the almost sure two-site localisation of the PAM on $T_\infty$ (Theorem \ref{thm:main a.s. localisation intro}) in Section \ref{sec:localisation} and the with high probability one-site localisation (Theorem \ref{thm:main whp localisation intro}) in Section \ref{sec:singlesitelocalisation}.

\textbf{Acknowledgements.} We would like to thank Nadia Sidorova for suggesting to look at PAM on trees. EA would also like to thank Matt Roberts and Marcel Ortgiese for helpful discussions. The research of EA was supported by an LMS Early Career Fellowship and ERC starting grant 676970 RANDGEOM. AP was supported by the German Academic Exchange Service and a Lichtenberg-Professorship (Christian Kuehn).

\textbf{Notation.}
To facilitate reading we summarise notation in the following table.
\begin{table}[H]
\begin{tabular}{|l|l|}
\hline
Notation & Definition \\
\hline
$\alpha$& Pareto parameter for the potential $\xi$\\
$\beta$& parameter for the GW offspring distribution, in $(1,2]$\\
$d$ & $\frac{\beta}{\beta-1}$\\
$\mathbb{T}_{\infty}$& set of rooted trees with a single (one-ended) infinite path to infinity\\
$O=s_0$ & root of $T_\infty$\\
$s_0,s_1,s_2,...$& ordered vertices of the backbone of $T_\infty$ \\
$u\prec v$& $u$ is an ancestor of $v$\\
$[[u,v]]$ & (unique) direct path from $u$ to $v$, inclusive of endpoints\\
$|u-v|$, $|v|$& $|u-v|$ is the graph distance between $u$ and $v$, $|v|:=|O-v|$\\
$B(v,r)$, $B_r$& $B(v,r)$ closed ball in $\Ti$ of radius $r$ around $v$, $B_r:=B(O,r)$\\
$B_r^{T_v}$&$B_r^{T_v} = B(v, r) \cap T_v$ for subtree $T_v$ of $\Ti$ \\
$A_r$& connected component containing $O$ obtained after removing the vertex $s_r$ from $T_\infty$ \\
$V(\mathcal T)$& set of vertices of a tree/subtree $\mathcal T$\\
$\#S$& number of vertices in $S$\\
$\Height(\mathcal T)$&$\Height(\mathcal T):=\max_{v\in V(\mathcal T)}|\tilde O-v|$ for a finite tree $\mathcal T$ with root $\tilde O$\\
$\tilde T$&$\{v\in V(T_\infty):\deg(v)\leq 4\}$\\
$\tilde B_r$& $\{v\in B_r:\deg(v)\leq 4\}$\\
$D_r$& $\sup_{v\in B_r}\deg(v)$ \\
$\mathcal{L}$& $\{u \in \Ti: \deg u = 1\}$ \\
$p(u)$& parent of $u$ in $\Ti$ \\
$\ell (u)$& leftmost child of $u$ that is also a leaf \\
$\eta_r(n,v)$& number of distinct paths of length $n$ starting at $O$, passing through $v$, fully contained in $B_r$ \\
$\xi$& the potential \\
$\xi^{(i)}_r$& the $i^{th}$ highest value of $\xi$ on $B_r$ \\
$\tilde{\xi}^{(i)}_r$& the $i^{th}$ highest value of $\xi$ on $\tilde{B}_r$ \\
$q$& $\frac{d}{\alpha - d}$ \\
$a(t)$& $\left(\frac{t}{\log t} \right)^{q}$ \\
$r(t)$& $\left(\frac{t}{\log t} \right)^{q+1}$ \\
$z$& $\frac{d\alpha}{\alpha - d}$. \\
$p$&$q+2$ \\
$R(t)$& $r(t) (\log t)^{p}$ \\
$C$&$\frac{pd}{\alpha}+z+1$\\
$F_t$& $\{v\in V(B_{R(t)}):\xi(v) \geq R(t)^{\frac{d}{\alpha}} (\log t)^{-C}\}$ \\
$E_t$& $\{v \in V(B_{R(t)}) \setminus F_t: \exists z \in F_t \text{ such that } \xi(z) - \deg (z) \leq \xi(v) - \deg (v)\}$ \\
$\psi_t(v)$ & $\left[ (\xi(v)-\deg(v))-\frac{|v|}{t} \log \left(\xi(v) - \deg (v)\right)\right]\mathbbm 1\{t(\xi(v)-\deg(v)) \geq |v|\}$ \\
$\lambda(t,v,y)$& $t(\xi(y) - \deg y) - |y| [\log |y| - \log t - 1] -  |v-y| [\log |v-y| - \log t - 1]$ \\
$\lambda(t,v)$& $\sup_{y \in \Ti} \lambda(t,v,y)$ \\
$Y(t,v)$& $\arg \max_{y \in B_{r(t)(\log t)^{m}}} \left\{ \lambda(t,v,y) \right\}$, where $m$ is a constant in $(q, \infty)$ \\
$\tilde{Y}(t,v)$& $\arg \max_{y \in B_{r(t)(\log t)^{m}}} \left\{ \lambda(t,v,y) + t \deg y \right\}$ \\
$\hat Z_t^{(1)}$ & $\arg \max_{v \in V(T_\infty)} \psi_t(v)$\\
$\hat Z_t^{(2)}$&$\arg \max_{v \in V(T_\infty) \setminus \{\hat Z^{(1)}_t\}} \psi_t(v)$\\
$Z_{B_r}$ & $\arg_{z \in B_r} \max \{\xi(z) \}$ \\
$\tilde{Z}_{B_r}$ &$\arg_{z \in B_r} \max \{\xi(z) - \deg(z) \}$ \\
$g_{B_r}$ &$\xi(Z_{B_r}) - \max_{z \in B_r, z \neq Z_{B_r}} \{\xi(z) \}$ \\
$\tilde{g}_{B_r}$ &$\xi(\tilde Z_{B_r}) - \deg(\tilde Z_{B_r}) - \max_{z \in B_r, z \neq \tilde Z_{B_r}} \{\xi(z) - \deg(z) \}$\\
$j_t$&$t^{\frac{d+\epsilon}{\beta}+1}r(t)$\\
\hline
\end{tabular}
\end{table}

\begin{table}[H]
\begin{tabular}{|l|l|}
\hline
Notation & Definition \\
\hline
$\Gamma^{(i)}_t$&$ \left\{ v \in \Ti: d(v, \hat{Z}_t^{(i)}) + \min \{ |v|, |\hat{Z}_t^{(i)}|\} \leq  \left( 1 + (\log t)^{-z}\right)|\hat{Z}_t^{(i)}|\} \right\}$ \\
$\Lambda_t$&$\Gamma_t^{(1)}\cup \Gamma_t^{(2)}$\\
$\Omega_t$&$\{\hat Z_t^{(1)},\hat Z_t^{(2)}\}$\\
$\Gamma^{(i)}_{x,t}$& the set of all paths from $x$ to $ \hat{Z}^{(i)}_t$\\
$\gamma^{(i)}_{x,t}$ &
the direct path from $x$ to $\hat{Z}_t^{(i)}$ \\
$\pi(X_{[0,t]})$& the path that consists of all the steps taken by the random walk $(X_s)_{s\geq 0}$  between times $0$ and $t$.  \\
$\overline{\psi}_t(v)$& $\xi(v)-\frac{|v|}{t} \log \left( \frac{|v|}{t}\right)$ \\
\hline
\end{tabular}
\end{table}

%%%%%%%%%%%%%%%%%%%%%%%%%%%%%%%%%%%%%%%%%
%%%%% Section: Tree model
%%%%%%%%%%%%%%%%%%%%%%%%%%%%%%%%%%%%%%%%%
\section{Critical Galton-Watson trees conditioned to survive}
\label{sec:criticalGWT}

We assume that the reader is familiar with the definition of a Galton-Watson tree with offspring distribution $\mu$, and simply recall that such a Galton-Watson tree can be formed from a single individual who reproduces with offspring number distributed according to $\mu$, and that individuals in subsequent generations continue to reproduce independently with offspring distribution $\mu$, ad infinitum (or until the process dies out). We refer to \cite{AbDelGWIntro} for further background and definitions.

In this paper, we will restrict to the case of \textit{critical} GW trees where $\sum_{k=0}^{\infty} k \mu(k) = 1$. We let $T$ denote such a GW tree, and $O$ its root. For a given vertex $v \in T$, its offspring number $k_v$ is equal to the number of children under the natural genealogical structure (i.e. $k_v = \deg v -1$). Additionally, we will denote by $T_v$ the subtree containing $v$ and all its descendants under this genealogical structure, rooted at $v$.

In order to consider the long-time behaviour of the PAM on a Galton-Watson tree, it is natural to restrict to the case of infinite trees. A natural candidate for such a tree would be a supercritical Galton-Watson tree; however, in the case of a Pareto potential, the exponential volume growth in a supercritical Galton-Watson tree causes the solution of the PAM to blow up, almost surely; hence the restriction to critical Galton-Watson trees. Although a critical Galton-Watson tree is known to have zero probability of survival, Kesten \cite{KestenIICtree} gave a construction of a critical Galton-Watson tree conditioned to survive, and showed that it arises as the local limit of finite Galton-Watson trees conditioned to have large height \cite[Lemma 1.14]{KestenIICtree}. We give his construction below.

\begin{defn}\label{def:infinite crit tree}
Let $\mu=(\mu_n)_{n \geq 0}$ be a critical offspring distribution. Define its size-biased version $\mu^*$ by
$$\mu^*_n=n\mu_n$$
for all $n \geq 0$. The Kesten's tree $T_\infty$ associated to the probability distribution $\mu$ is a two-type Galton-Watson tree with the following properties: 
\begin{itemize}
\item Each individual is of one of two types: normal or special. 
\item The root of $T_\infty$ is special. 
\item A normal individual produces normal individuals according to $\mu$. 
\item A special individual produces individuals according to the size-biased distribution $\mu^*$. One of them is chosen uniformly at random. This individual is of type special, the rest of the produced individuals are of type normal. 
\end{itemize}
We let $\bPb$ denote the law of $\Ti$.
\end{defn}
We let $\mathbb{T}_{\infty}$ denote the set of rooted trees with a single (one-ended) infinite path to infinity. Note that, $\bPb$-almost surely, $\Ti \in \mathbb{T}_{\infty}$, since the special vertices form a unique one-ended infinite backbone of $\Ti$. We let $s_0, s_1, s_2, \ldots$ denote the ordered vertices of the backbone, so that $s_0$ is the root, and $s_n$ is at distance $n$ from the root. Moreover, each of the subtrees emanating from normal children of the backbone vertices are independent unconditioned Galton-Watson trees with offspring distribution $\mu$.

\medskip
%%%%%%%%%%%%%%%%%%%%%%%%%%%%%%%%%%%%%%%%%
%%% Subsection: Volume growth
\subsection{Properties of $\Ti$}

In this section we give some properties of $\Ti$ that we will use later in the paper, mainly regarding its volume growth. Since $\Ti$ ($\bPb$-almost surely) consists of a collection of finite Galton-Watson trees grafted to its infinite backbone, it will also be useful to state some bounds for an unconditioned Galton-Watson tree with the same offspring distribution $\mu$, which we denote by $T$.

\begin{lem}[{\cite[Lemma 1.11]{Kort}} and {\cite[Theorem 2]{SlackInfVar}}]\label{lem:finite GW prob bounds} Under Assumption \ref{assn:whp}, there exist constants $c,c' < \infty$ such that, as $n \rightarrow \infty$,
\begin{align*}
\prb{\# T \geq n} &\sim cn^{\frac{-1}{\beta}} \\
\prb{\Height (T) \geq n} &\sim c'n^{\frac{-1}{\beta-1}}.
\end{align*}
\end{lem}

\begin{prop}\label{prop:AR vol growth and exit time}
Take $r, \lambda > 1$. Then, for any $\epsilon > 0$, there exist constants $c,C \in (0, \infty)$ such that
\begin{enumerate}[(i)]
    \item \cite[Proposition 2.5]{CroydonKumagai08}. Under Assumption \ref{assn:whp}, $\prb{\#B_r \geq \lambda r^{\frac{\beta}{\beta -1}}} \leq C\lambda^{-(\beta -1 -\epsilon)}$.
    \item Under Assumption \ref{assn:as}, for any $K< \infty$ there exists $C< \infty$ such that $\prb{\#B_r \geq \lambda r^{2}} \leq C\lambda^{-K}$.
    \item \cite[Section 8.1]{ArcherRWDecGWTrees}. Under Assumption \ref{assn:whp}, $\prb{\#B_r \leq \lambda^{-1} r^{\frac{\beta}{\beta -1}}} \leq Ce^{-c \lambda^{\frac{\beta-1}{\beta}}}$.
\end{enumerate}
\end{prop}

\begin{proof} We just prove $(ii)$. Take some $\epsilon >0$ and set $b=a+\epsilon=\frac{1}{2}(1-4\epsilon)$. We bound $B_r$ by considering all subtrees attached to the backbone within distance $r$ of the root. We say that a subtree is ``tall'' if it has height exceeding $r \lambda^b$. We start by identifying all tall subtrees attached to the backbone. If a subtree is tall, we can similarly decompose this subtree into subsubtrees by decomposing along its (leftmost) path of maximal height, which we call its spine, and again identify all tall subsubtrees that are attached somewhere within distance $r$ of the root of the subtree. We can then keep repeating this process on tall subsubtrees in an inductive way until we eventually end with something as in Figure \ref{fig:tall subtrees1}.

\begin{figure}[h]
\centering
\begin{minipage}[t]{0.45\linewidth}
\centering
{\begin{tikzpicture}[]
\draw (0,0) -- (6,0);
\draw (0,0) node[align=left,   below]{$O$};
\draw[dashed](6,0) -- (7,0);
\draw [<->] (0,-2)--(6,-2);
\draw (3.5,-2) node[align=center,below] {$r$};
\draw[color=magenta] (1,0)--(2,1);
\draw[color=magenta] (2.5,0)--(4,-1.5);
\draw[color=cyan] (3,-0.5)--(3,-1.5);
\draw[color=magenta] (4.5,0)--(6,1.5);
\draw[color=cyan] (5,0.5)--(5,1.5);
\draw[color=orange] (5,1)--(4.5,1.3);
\draw[color=magenta](5.5,0)--(6,-1);
\draw[color=cyan] (5.8,-0.6)--(5,-1.5);
\end{tikzpicture}}
\caption{Spines of tall subtrees}
\label{fig:tall subtrees1}
\end{minipage}
\begin{minipage}[t]{0.45\linewidth}
\centering
{\begin{tikzpicture}[]
\draw (0,0) -- (4,0);
\draw (0,0) node[align=left,below, color=white]{$O$};
\draw (3.5,-2) node[align=center,below, color=white]{$r$};
\draw (2.5,0)--(4,-1.5);
\draw (1.5,0)--(3,1.5);
\draw[color=red] (0.2,0)--(0.4,0.4);
\draw[color=red] (1,0)--(1.2,0.6);
\draw[color=darkgreen] (2,0)--(2.2,0.6);
\draw[color=darkgreen] (3,0)--(3.5,0.4);
\draw[color=darkgreen] (3.7,0)--(4,-0.8);
\draw[color=blue] (1.7,0)--(2.5,-0.4);
\draw[color=blue] (0.7,0)--(2,-0.8);
\end{tikzpicture}}
\caption{One tall subtree $T$}
\label{fig:tall subtrees2}
\end{minipage}
\end{figure}
This collection of tall subtrees can be indexed by a different branching process whereby the offspring of a given tall subtree correspond to the tall subtrees attached to its spine (so in Figure \ref{fig:tall subtrees1}, different colours represent subsequent generations). We let the total number of tall subtrees (including the original tree $\Ti$) be $N$. Then $N$ is an $O(1)$ random variable and moreover we can control $\#B_r$ by controlling the masses of all the non-tall subtrees in $B_r$.

Given a tall subtree $T$, we define its $r$-spine to be the collection of vertices in the spine of $T$ that also lie within distance $r$ of the root of $T$, and let $Sp_T(r)$ denote the set offspring of vertices in the $r$-spine of $T$. By Definition \ref{def:infinite crit tree} and \cite[Lemma 2.1]{GeigerKersting1998galton} the offspring tails on the $r$-spine are at most size-biased, so under Assumption \ref{assn:as} they still decay super-polynomially. Consequently, for any $z>0$ we have that $\prb{\#Sp_T(r)\geq r\lambda^a} \leq c_z \lambda^{-az}$. In particular, it follows from a union bound that for the first $\lambda^{\epsilon}$ tall subtrees present in Figure \ref{fig:tall subtrees1} (say ordered in a breadth-first order), the probability that one of them has $\#Sp_T(r)\geq r\lambda^a$ is at most $c_z \lambda^{\epsilon-az}$. Additionally, if $\#Sp_T(r)< r\lambda^a$, then the number of tall subtrees attached to the $r$-spine of $T$ is at most a \textsf{Binomial}($r\lambda^a, cr^{-1}\lambda^{-b}$) random variable by Lemma \ref{lem:finite GW prob bounds}. Consequently, it follows from the main theorem of \cite{Dwass1969} applied to the indexing branching process that
\begin{align}\label{eqn:vol UB 1}
\begin{split}
\prb{N \geq \lambda^{\epsilon}} &\leq \prb{ \text{one of the first }\lambda^{\epsilon} \text{ tall subtrees has }\#Sp_T(r)\geq r\lambda^a} \\
&\quad + \prb{N \geq \lambda^{\epsilon}, \text{none of the first }\lambda^{\epsilon} \text{ tall subtrees has }\#Sp_T(r)\geq r\lambda^a}  \\
&\leq  c_z \lambda^{\epsilon-az} + \prb{GW_{\textsf{Binomial}(r\lambda^a, cr^{-1}\lambda^{-b})} \geq \lambda^{\epsilon}} \\
&\leq  c_z \lambda^{\epsilon-az} + e^{-c\lambda^{\epsilon}}.
\end{split}
\end{align}
(e.g. see \cite[Proposition 5.1]{archer2021brownian} for a similar proof).
Also let $M$ denote the number of subtrees attached to the $r$-spine of some tall subtree, with volume exceeding $r^2 \lambda^{2b+2\epsilon}$ but with height less than $r \lambda^b$. On the event $N<\lambda^{\epsilon}$ and $\#Sp_T(r)< r\lambda^a$ for all tall subtrees, we have from \cite[page 5]{KortSubexp} that $M$ is stochastically dominated by a \textsf{Binomial}($r\lambda^{a+\epsilon}, cr^{-1} \lambda^{-b-\epsilon}e^{-c\lambda^{\epsilon}}$) random variable, so that by Markov's inequality,
\begin{equation}\label{eqn:vol UB 2}
\prb{M > 0} \leq Ce^{-c\lambda^{\epsilon}}.
\end{equation}
We now look at a single tall subtree $T$. Its contribution to $\#B_r$ is equal to the sums of the volumes of the masses of small subtrees attached to its $r$-spine, plus $r$ (which is the volume of its $r$-spine). Each of these small subtrees can be sorted into groups whereby two small subtrees are in the same group precisely when there are no tall subtrees between them in the contour exploration of $T$. These groups are indicated by colours in Figure \ref{fig:tall subtrees2}. The number of groups is equal to $n +1$, where $n$ is the number of tall subtrees attached to $T$. Therefore, if we sum over all tall subtrees $T$, the total number of such groups is at most $2N$. Label them as $X_1, X_2, \ldots X_{2N}$. Moreover, by Lemma \ref{lem:finite GW prob bounds} the total volume of one single group falls into the framework of \cite[Lemma A.1]{ArcherRWDecGWTrees} with $\beta = \frac{1}{2}$ there. We deduce that, on the event $N < \lambda^{\epsilon}$,
\begin{equation}\label{eqn:vol UB 3}
\prb{\sum_{i \leq 2N} \# X_i \geq r^2 \lambda^{2b+4\epsilon}} \leq \prb{\exists i \leq 2N: \# X_i \geq r^2 \lambda^{2b+3\epsilon}} \leq 2Ne^{-\lambda^{\epsilon}} \leq \lambda^{\epsilon}e^{-\lambda^{\epsilon}}.
\end{equation}
Therefore, since $\#B_r \leq \sum_{i \leq 2N} \# X_i + Nr$, we have from a union bound on all the events described in \eqref{eqn:vol UB 1}, \eqref{eqn:vol UB 2} and \eqref{eqn:vol UB 3} above that
\[
\prb{\#B_r \geq 2r^2 \lambda} = \prb{\#B_r \geq 2r^2 \lambda^{2b+4\epsilon}} \leq 3c_z \lambda^{\epsilon-(1-5\epsilon)z}.
\]
Therefore, choosing $z> K+1$ and $\epsilon < \frac{1}{6(z \vee 1)}$ gives the stated result.
\end{proof}
By applying Borel-Cantelli along the sequence $r_n = 2^n$, taking $\lambda_n = \delta (\log r_n)^{\frac{1+\epsilon}{\beta - 1}}$ (and similarly for the other bounds), using monotonicity of volumes and then $\delta \downarrow 0$, we deduce the following result.

\begin{cor}\label{cor:volume bounds}
Set $d := \frac{\beta}{\beta - 1}$. Under Assumption \ref{assn:whp}, $\mathbf P$-almost surely for any $\epsilon >0$ it holds that:
\begin{align*}
\limsup_{n \rightarrow \infty} \frac{\#B_r}{r^{d} (\log r)^{\frac{1+\epsilon}{\beta - 1}}} = 0, \hspace{1cm} \liminf_{ r \rightarrow \infty} \frac{\#B_r}{r^{d} (\log \log r)^{-(d+\epsilon)}} = \infty.
\end{align*}
Under Assumption \ref{assn:as}, $\mathbf P$-almost surely for any $\epsilon >0$ it holds that:
\begin{align*}
\limsup_{n \rightarrow \infty} \frac{\#B_r}{r^{2} (\log r)^{\epsilon}} = 0, \hspace{1cm} \liminf_{ r \rightarrow \infty} \frac{\#B_r}{r^{2} (\log \log r)^{-(2+\epsilon)}} = \infty.
\end{align*}
\end{cor}

We will also need the following result on the largest degree in a ball.

\begin{lem}\label{lem:supdeg} 
\begin{enumerate}[(i)]
\item For any $\epsilon >0$, under Assumption \ref{assn:whp} it holds eventually $\mathbf P$-almost surely that
$$D_r := \sup_{v\in B_r}\deg(v)\leq r^{\frac{d(1+\varepsilon)}{\beta}} = r^{\frac{(1+\varepsilon)}{\beta - 1}}.$$
\item Under Assumption \ref{assn:as} we have for any $\epsilon > 0$ that eventually $\mathbf P$-almost surely 
$$D_r \leq r^{\varepsilon}.$$
\end{enumerate}
\end{lem}
\begin{proof}
\begin{enumerate}[(i)]
\item Note that there are at most $r$ backbone vertices in $B_r$ with the size-biased offspring tail bound, and all other vertices have the original offspring tail bounds, independently of each other (by Definition \ref{def:infinite crit tree}). By Corollary \ref{cor:volume bounds}, $\bPb$-almost surely for all sufficiently large $r$ we can compute 
\begin{align*}
    \mathbf P\left( D_r > 2^{\frac{-d(1+\varepsilon)}{\beta}} r^{\frac{d(1+\varepsilon)}{\beta}}\right)&\leq r^d (\log r)^{\frac{1+\varepsilon}{\beta - 1}} \sup_{v \in B_r \setminus \{s_0, \ldots, s_r\}} \mathbf P(\deg v>2^{\frac{-d(1+\varepsilon)}{\beta}}r^{\frac{d(1+\varepsilon)}{\beta}}) \\
    &\qquad + r \sup_{v \in\{s_0, \ldots, s_r\}} \mathbf P(\deg v>2^{\frac{-d(1+\varepsilon)}{\beta}}r^{\frac{d(1+\varepsilon)}{\beta}}) \\
    &\leq cr^{d}(\log r)^{\frac{1+\varepsilon}{\beta-1}}r^{-d-\varepsilon}+crr^{\frac{-d(1+\varepsilon)(\beta - 1)}{\beta}}\\
    &=r^{-\varepsilon/2}.
\end{align*}
This probability is summable along the sequence $r_n=2^n$ and thus by Borel-Cantelli along this sequence it holds eventually $\mathbf P$-almost surely that $\sup_{v\in B_{r_n}}\deg(v)\leq 2^{\frac{-d(1+\varepsilon)}{\beta}} r_n^{\frac{d(1+\varepsilon)}{\beta}}$. Then note for $r\in[r_n,r_{n+1}]$ we have by monotonicity that $\sup_{v\in B_{r}}\deg(v)\leq 2^{\frac{-d(1+\varepsilon)}{\beta}} r_{n+1}^{\frac{d(1+\varepsilon)}{\beta}}\leq  r^{\frac{d(1+\varepsilon)n}{\beta}}$. 
\item Fix $\epsilon \in (0,1)$. Since the offspring distribution tails decay super-polynomially under Assumption \ref{assn:as}, the same holds for the size-biased offspring tails, so we certainly have that $\sup_{v\in B_{r}} \mathbf P(\deg v>x) = o(x^{-(2/\varepsilon+1)})$ as $x \to \infty$. The proof then proceeds exactly as in part $(i)$.
\end{enumerate}
\end{proof}

Given $v \in B_r, n \geq |v|$, we let $N_r(n,v)$ denote the number of distinct paths starting at $O$ and passing through $v$ that are fully contained in $B_r$ and have length exactly $n$, and $N_r(n,v) = \log \eta_r(n,v)$. The previous lemma allows us to bound $\eta_r(n,v)$.

\begin{lem}\label{lem:no of paths}
Under Assumption \ref{assn:whp}, for all $\epsilon > 0$, eventually almost surely, for all $v \in B_{r}$ we have that
\[
\eta_r (n,v) \leq (n-|v|)\left(\frac{d+\epsilon}{2\beta} (\log r) + \log \left( \frac{n}{n-|v|} \right)\right).
\]
Under Assumption \ref{assn:as}, for all $\epsilon > 0$, eventually almost surely, for all $v \in B_{r}$ we have that
\[
\eta_r (n,v) \leq (n-|v|)\left(\epsilon (\log r) + \log \left( \frac{n}{n-|v|} \right)\right).
\]

 \end{lem}
 \begin{proof}
By counting deviations from the direct path from $O$ to $v$, and letting $D_r = \sup_{v \in B_r} \deg v$ we have that
\[
N_r(n,v) \leq \binom{n}{n-|v|} D_r^{\frac{n-|v|}{2}}.
\]
The factor of $\frac{1}{2}$ appears since if we cross one edge not on the direct path we must return to the direct path by the same edge, so this reduces our number of choices by a factor of $2$. Applying Stirling's formula and the bound on $D_r$ from Lemma \ref{lem:supdeg}$(i)$ or \ref{lem:supdeg}$(ii)$ then gives the results. For example, under Assumption \ref{assn:as} we get
\begin{align*}
    \eta_{r}(n,v)\leq \log\left(\binom{n}{n-|v|} (r^\varepsilon)^{\frac{n-|v|}{2}}\right) \leq \log\left(\left(\frac{ne}{n-|v|}\right)^{n-|v|} (r^\varepsilon)^{\frac{n-|v|}{2}}\right)
    %&\leq (n-|v|)\left(\frac{\varepsilon}{2}\log r+\log\left(\frac{ne}{n-|v|}\right)\right) \\
    &\leq (n-|v|)\left(\varepsilon\log r+\log\left(\frac{n}{n-|v|}\right)\right).
\end{align*}
\end{proof}

\begin{rmk}\label{rmk:whp assn problem}
As $n \to \infty$, it is not possible to do much better than the result in Lemma \ref{lem:no of paths} under Assumption \ref{assn:whp}, since the maximal degree of a vertex in $B_r$ is really of order $r^{\frac{d}{\beta}}$. Therefore one can really count paths that dart directly to this vertex of high degree, then alternate between this vertex and its neighbours for as long as possible, and finally dart directly to $v$ at the last possible moment. Counting these paths gives a lower bound asymptotic to $\frac{d+\epsilon}{2\beta} (n-|v|)(\log r)$ (and we will mainly use the result with $\frac{n}{n-|v|} \ll r$).
\end{rmk}

\begin{lem}\label{lem:deg4set}
Let $\tilde B_r=\{v\in B_r:\deg v\leq 4\}$. Then, under Assumption \ref{assn:whp},
$$\mathbf P(\#\tilde B_r\leq r^d\lambda^{-1})\leq Ce^{-c\lambda^{\frac{\beta-1}{\beta}}}+(r\lambda^{-1})^{-(\beta-1-\epsilon)}.$$
Consequently, $\bPb$-almost surely,
$$
\liminf_{ r \rightarrow \infty} \frac{\#\tilde{B}_r}{r^{d} (\log \log r)^{-(d+\varepsilon)}} > 0.
$$
\end{lem}
\begin{proof}
In any finite tree at least half of the vertices have at most degree $4$. Taking care of the boundary vertices in $B_r$ this yields 
$$\#\tilde B_r\geq \frac{1}{2}\#B_r-Z_r^*,$$
where $Z_r^*$ is the generation size at level $r$. Thus invoking Proposition \ref{prop:AR vol growth and exit time}$(iii)$ and \cite[Proposition 2.2]{CroydonKumagai08}
\begin{align*}
\mathbf P(\#\tilde B_r\leq r^d\lambda^{-1})&\leq \mathbf P(\frac{1}{2}\#B_r-Z_r^*\leq r^d\lambda^{-1}) \leq \mathbf P(\#B_r\leq 4r^d\lambda^{-1})+\mathbf P(Z_r^*\geq r^d\lambda^{-1})\\
&\leq C e^{-c\lambda^{\frac{\beta-1}{\beta}}}+(r\lambda^{-1})^{-(\beta-1-\epsilon)}.
\end{align*}
The second claim follows by Borel-Cantelli exactly as in Corollary \ref{cor:volume bounds}.
\end{proof}

We will also use the following result in the random walk estimates.

\begin{lem}\label{lem:Br log sum bound}
Let $\Gamma_r$ be the set of all direct paths between two vertices in $B_r$. Under Assumption \ref{assn:whp} there exist deterministic constants $B, \tilde B < \infty$ such that $\bPb$-almost surely,
\[
\sup_{\gamma \in \Gamma_r}\left\{\sum_{v \in \gamma}\log(\deg v) - \tilde B |\gamma| - B\log r\right\}\leq 0
\]
for all sufficiently large $r$.
\end{lem}
\begin{proof}

Let $\epsilon \in (0, \beta - 1)$, define 
  $c:=\beta-1-\epsilon>0$ and choose $B>\frac{1}{c}\left(2\frac{\beta+\epsilon}{\beta-1}+\epsilon\right)$. Note that $\Eb{(\deg v)^{c}}< \infty$ for all $v \in B_r$. We set $A:=\log\mathbf E[(\deg v)^c]<\infty$. Also note that the degrees of distinct vertices are independent of each other by Definition \ref{def:infinite crit tree}.
     Given $\gamma \in \Gamma_r$, we also define $\lambda_{\gamma,r}:=\frac{B\log r}{|\gamma|}+\frac{A}{c}$ and $v_r := r^{\frac{\beta+\epsilon}{\beta - 1}}$. With this choice we calculate using a Chernoff bound that
      \begin{align*}
\mathbf P\left(\sum_{v\in \gamma}\log(\deg v)\geq \lambda_{\gamma,r} |\gamma|\right) \leq \mathbf E\left[\exp\left(c\sum_{v\in \gamma}\log(\deg v)\right)\right]\exp(-c\lambda_{\gamma,r} |\gamma|)
&\leq \exp\left(A|\gamma|-c\lambda_{\gamma,r} |\gamma|\right) =r^{-cB}.
 \end{align*}
Note also that $|\Gamma_r|\leq B_r^2$ since we are only bounding direct paths. Applying a union bound over all $\gamma \in \Gamma_r$ and Proposition \ref{prop:AR vol growth and exit time}$(i)$ therefore yields (for all sufficiently large $r$) that
 \begin{align*}
    \mathbf P\left(\sup_{\gamma \in \Gamma_r}\left\{\sum_{v\in \gamma}\log(\deg v)- \lambda_\gamma |\gamma|\right\}\geq 0\right) &  \leq \mathbf P\left( B_r \geq v_r\right) + v_r^2 \mathbf P\left(\sup_{\gamma\in \Gamma_r}\left\{\sum_{v\in \gamma}\log(\deg v)- \lambda_{\gamma,r} |\gamma|\right\}\geq 0\right) \\
    &\leq r^{\frac{\epsilon}{2}} + v_r^2 r^{-cB} = 2r^{\frac{\epsilon}{2}}. 
 \end{align*}
The result therefore follows by applying Borel-Cantelli along the subsequence $r_n = 2^n$, and then applying monotonicity if $r \in [2^n, 2^{n+1}]$, since $\log r = o(r)$.
\end{proof}

%%%%%%%%%%%%%%%%%%%%%%%%%%%%%%%%%%%%%%%%
%%%% Subsection: RW on tree

\subsection{Random walk on $\Ti$}
Given a particular realisation of $\Ti$, we let $X=((X_t)_{t\geq 0},\mathbb P_x^{\Ti},x\in V(\Ti))$ continuous time variable speed random walk on $\Ti$ started at $x \in V(\Ti)$. This means that for $v, w \in V(\Ti)$, $X$ jumps from $v$ to $w$ at rate $\mathbbm{1}\{v \sim w\}$, so the holding times at vertex $v$ are exponentially distributed random variables with parameter $\deg v$. In particular, the mean holding time at vertex $v$ is $\frac{1}{\deg v}$.

In the case where $X$ is started from the root $O$ of $\Ti$, we just write $\Pb$ for the law of $X$.

Let $A\subset V(\Ti)$. We denote the first exit time from $A$ as $\tau_A$, i.e. $\tau_A=\inf\{t\geq 0:X_t\notin A\}$. Given a vertex $v \in V(\Ti)$, we also let $H_v = \inf\{ s \geq 0: X_s = v\}$ denote the first hitting time of $v$.

The next proposition is actually a result about random walks on deterministic trees.

\medskip

\begin{prop}\label{prop:hitting time prob bounds}
Conditional on $\Ti$, for each $v\in V(\Ti)$,
\begin{enumerate}[(i)]
 \item If $|v| \geq 3$, then $$ \mathbb P(H_v\leq t)\leq \exp\{-|v|([\log |v|-\log t-1]\vee 0)\}.$$
 \item For all $v \in V(\Ti),t > 0$,
\begin{align*}
 \mathbb P(H_v\leq t) \geq \left( \prod_{u \prec v} \frac{1}{\deg(u)} \right) e^{-|v| \left( [\log |v| - \log t - 1] \vee 0 \right) + O(t)}.
 \end{align*}
\end{enumerate}
\end{prop}
\begin{proof}
\begin{enumerate}[(i)]
    \item Set $r=|v|$. Let $O=v_0 \prec v_1 \prec \ldots \prec v_r=v$ denote the ordered ancestors of $v$. To reach $v$ from $v_0$, the random walk must first pass through each of the points $v_1, \ldots, v_{r-1}$. When it reaches the point $v_i$, the time to jump to $v_{i+1}$ will stochastically dominate an $\textsf{Exp}(1)$ random variable (since it can only reach $v_{i+1}$ through one edge which rings at rate $1$). The time delay for the random walk to reach $v$ will thus be greater than 
    $\sum_{u\prec v} E_u,$
    where $(E_u)_{u\prec v}$ is a sequence of i.i.d. $\textsf{Exp}(1)$ random variables. Hence, using Stirling's formula we obtain

\begin{align*}
    \pr{H_v \leq t} \leq \pr{\sum_{u\prec v} E_u \leq t} = \pr{ \textsf{Poi} ( t) \geq r} = e^{- t} \left( \frac{t^r}{r !} + \sum_{i=r+1}^{\infty} \frac{ t^i}{i !} \right) 
    %&= e^{- t}  \frac{ t^r}{r !} \left( 1 + \sum_{i=r+1}^{\infty} \frac{t^{i-r} r!}{i !} \right) \\ 
    &\leq e^{- t}  \frac{ t^r}{r !} \left( 1 + \sum_{i=1}^{\infty} \frac{t^{i}}{i !} \right)\\
    &\leq \frac{t^r e^{-t}e}{\sqrt{2\pi}r^{r+1/2}e^{-r}}  \\
    &\leq \exp \{-r[\log r - \log t - 1] \}.
\end{align*}

\item We lower bound this by the probability of going directly to $v$ in time $t$. Accordingly, again let $O=v_0 \prec v_1 \prec \ldots \prec v_r=v$ denote the ordered ancestors of $v$, but this time let $(E_i)_{i=0}^{r-1}$ be independent random variables, where $E_i\sim \textsf{Exp}(\text{deg}(v_i))$ for each $i=0,...,r-1$. Also, recall that $e^{-x} + e^{-x^{-1}} \leq \frac{1}{1+x}+\frac{1}{1+x^{-1}}=1$ for all $x>0$. 

Then, if $r(\log 2) >3t$, since all vertices except perhaps $v_0$ have degree at least $2$, we have (by Stirling's formula) that for all sufficiently large $r$
\begin{align}\label{eqn:RW hitting time LB estimate}
\begin{split}
    \pr{\sum_{i=0}^{r-1} E_i \leq t} &\geq \pr{ E_0 \leq 1} \pr{\sum_{i=1}^{r-1} E_i \leq t-1} \\
    &\geq (1-e^{-1})\pr{\textsf{Poi}(2(t-1)) \geq r-1} \geq (1-e^{-1})\frac{e^{-2(t-1)}(2t-2)^{(r-1)}}{(r-1)!} \\
    %&\geq \frac{e^{-(2t)}e^{(r+1)}(2t)^{(r+1)}}{{(r+1)}^{{(r+1)}+\frac{1}{2}}e} \\
    &\geq \exp\{-r\left(\log r - \log t -1\right) - O\left(\log (r \vee t) + \frac{r}{t}\right) +{(r+1)}\log 2- (2t+1)  \}.
    %&\geq \exp\{-r\left(\log r - \log t -1\right) \}.
\end{split}
\end{align}
The case $r(\log 2) >3t>50$ follows from the penultimate line above (the constant $50$ is not optimal).
\item The case $r(\log 2) \leq 3t$ similarly follows from \eqref{eqn:RW hitting time LB estimate}. 
\end{enumerate}
\end{proof}

\color{black}

\begin{cor}\label{cor:exitball}
Under Assumption \ref{assn:whp}, $\mathbf P$-almost surely, we have for all sufficiently large $r$ that
$$\mathbb P(\tau_{B_{r-1}}\leq t)\leq   \exp\left\{-\frac{r}{5}\log\left(\frac{r}{et}\right) \right \} $$
for all $t > 0$.
\end{cor}
\begin{proof}
Let us define for $r>0$ the set 
\begin{equation}\label{eqn:Cr}
C_r:=\left\{v\in V(\Ti):|v|=\frac{r}{2}, \exists u \in V(\Ti): |u|=r,  v\prec u\right\},
\end{equation}
that is, the set of vertices in generation $\frac{r}{2}$ that have a descendant in generation $r$. We will first derive an upper bound for the volume of this set. 
Let $Z^*_n$ denote the size of the $n^{th}$ generation of $\Ti$ and let $v_i$, $i=1,...,Z^*_n$ denote the vertices of the $n^{th}$ generation. Furthermore, let $T_{v_i}$ denote the subtree emanating from $v_i$. Then 
$$\#C_r=\sum_{w\in \{v_1,...,v_{Z^*_{r/2}}\}} \mathbf 1\{\Height(T_{w})\geq r/2\}=1+\sum_{w\in \{v_1,...,v_{Z^*_{r/2}}\}\setminus \{s_{r/2}\}} \mathbf 1\{\Height(T_{w})\geq r/2\}.$$
Recall from Lemma \ref{lem:finite GW prob bounds} that there exists $c<\infty$ such that for all $v_i\neq s_{r/2}$, $\mathbf P(\Height(T_{v_i})\geq x)\sim c x^{\frac{-1}{\beta-1}}$. Consequently $\#C_r -1$ is stochastically dominated by a $\textsf{Binomial}(Z^*_{r/2}-1,cr^{-\frac{1}{\beta-1}})$ random variable.
Hence, for $c_2=\frac{1+\epsilon}{\beta-1-\epsilon}$ and $c_1=c_2+\epsilon$,
\begin{equation}\label{eqn:probcutvol}
    \mathbf P(\#C_r\geq (\log r)^{c_1})\leq \mathbf P (Z^*_{r/2}\geq r^{\frac{1}{\beta-1}}(\log r)^{c_2})+\mathbf P(X\geq (\log r)^{c_1}-1),
\end{equation}
where $X\sim \Binom(r^{\frac{1}{\beta-1}}(\log r)^{c_2}-1,cr^{-\frac{1}{\beta-1}})$.
We have by \cite[Proposition 2.2]{CroydonKumagai08} that for every $\epsilon>0$ 
\begin{equation}\label{eqn:someniceequation1}\mathbf P(Z^*_{r/2}\geq r^{\frac{1}{\beta-1}}(\log r)^{c_2})\leq c_3 (\log r)^{-c_2(\beta-1-\epsilon)}=c_3(\log r)^{-(1+\epsilon)},\end{equation}
for some constant $c_3>0$. Furthermore, since $c_1 > c_2$, by a Chernoff bound we have 
\begin{align} \label{eqn:someniceequation2}
    \mathbf P(X\geq (\log r)^{c_1}-1) \leq \frac{\mathbf E[e^{ X}]}{e^{((\log r)^{c_1}-1)}}
    &\leq \frac{\exp((r^{\frac{1}{\beta-1}}(\log r)^{c_2}-1)cr^{-\frac{1}{\beta-1}}(e-1))}{\exp(((\log r)^{c_1}-1))} \nonumber\\
   % &\leq \frac{\exp(c(\log r)^{c_2}(e-1))}{\exp((\log r)^{c_1}-1)}\\
    &\leq \exp(c(\log r)^{c_2}e-(\log r)^{c_1}+1).
\end{align}
Combining \eqref{eqn:someniceequation1} and \eqref{eqn:someniceequation2} with \eqref{eqn:probcutvol}, we deduce that for any $\epsilon > 0$, there exists $c_4 < \infty$ such that
\begin{equation*}\label{eqn:Cr tail bound}
    \mathbf P(\#C_r\geq (\log r)^{\frac{1+\epsilon}{\beta - 1}})\leq c_4(\log r)^{-(1+\epsilon)}.
\end{equation*}

For $r_n=2^n$ this probability is summable over $n$ and thus by Borel-Cantelli $\mathbf P$-almost surely for $n$ large enough $\#C_{r_n}\leq (\log {r_n})^{c_1}$. 

Now, since the random walk needs to visit a vertex in $C_r$ before it can exit the ball $B_{r-1}$, we have by a union bound and Proposition \ref{prop:hitting time prob bounds}$(i)$ that
 \begin{align*}
    \mathbb P(\tau_{B_{r-1}}\leq t)&\leq \mathbb P(\exists v\in C_r:H_v\leq t)\leq \sum_{v\in C_r}\mathbb P(H_v\leq t)\leq \#C_r \exp\left\{-\frac{r}{2}\log\left(\frac{r}{2et}\right)\right \}.
\end{align*}
Therefore, $\mathbf P$-almost surely along the subsequence $r_n$, for all $n$ sufficiently large it holds that
\begin{align*}
    \mathbb P(\tau_{B_{r_n-1}}\leq t)
    &\leq (\log {r_n})^{c_1} \exp\left\{-\frac{r_n}{2}\log\left(\frac{r_n}{2et}\right)\right \} =\exp\left\{-\frac{r_n}{2}\log\left(\frac{r_n}{2et}\right)+c_1\log(\log(r_n))\right \}.
\end{align*} 
Consequently, for $r\in[r_n,r_{n+1}]$ sufficiently large:
\begin{align*}
    \mathbb P(\tau_{B_{r-1}}\leq t)&\leq    \mathbb P(\tau_{B_{r_n}}\leq t)
\leq \exp\left\{-\frac{r_n}{2}\log\left(\frac{r_n}{2et}\right)+c_1\log(\log(r_n))\right \}
    %&\leq  \exp\left\{-\frac{r}{4}\log\left(\frac{r}{4et}\right)+c_1\log(\log(r))\right \}\\
    \leq  \exp\left\{-\frac{r}{5}\log\left(\frac{r}{et}\right)\right \}.
\end{align*}
\end{proof}

\begin{lem}\label{lem:RW ball exit time restricted radii}
Take any $\tilde f,\tilde p>0$. Under Assumption \ref{assn:whp}, we have with high $\bPb$-probability as $t\to \infty$, that for all $r \in [r(t)(\log t)^{-\tilde f}, r(t) (\log t)^{\tilde p}]$,
\begin{align*}
\pr{\tau_{B_{r-1}} \leq t} &\leq  \exp\left\{-r \log \left( \frac{r}{et} \right) + o \left( r(t)\right) \right\}.
\end{align*}
\end{lem}
\begin{proof}
We follow a refined version of the previous proof, this time with several steps.

\textbf{STEP 1.}
Choose some $\varepsilon>0$ and $0<\delta<1$. Let $m>\tilde p+1$, $k>\frac{m+\delta}{\beta-1-\varepsilon}+2\varepsilon$ and set $K:=k+\frac{m}{\beta-1}$. Given $r>0$, this time let $Z_r^*$ denote the number of vertices in generation $\left(1-\frac{1}{(\log r)^m}\right)r$, and let $C_r$ denote the set of vertices in this generation that have a descendant at level $r$. Then, similarly to the calculations in \eqref{eqn:probcutvol}-\eqref{eqn:someniceequation2} we compute
\begin{align}\label{eqn:hitting time calc}
\begin{split}
\prb{\#C_r \geq (\log r)^{K}} &\leq \prb{Z_r^* \geq r^{\frac{1}{\beta - 1}} (\log r)^{k-\epsilon}} + \prb{\textsf{Binomial}\left(r^{\frac{1}{\beta - 1}}(\log r)^{k-\epsilon}, c'r^{\frac{-1}{\beta-1}} (\log r)^{\frac{m}{\beta-1}}\right) \geq (\log r)^K} \\
&\leq c(\log r)^{-(k-\epsilon)(\beta - 1-\epsilon)} + \exp\{-e(\log r)^{\epsilon})\} \\
&\leq c(\log r)^{-(k-\epsilon)(\beta - 1-\epsilon)}.
\end{split}
\end{align}
Since the random walk needs to visit a vertex in $C_r$ before it can exit the ball $B_{r-1}$, we can lower bound the exit time of the set $B_{r-1}$ by the hitting time of the set $C_r$. For any $r>0$, on the event $\{\#C_r \leq (\log r)^K\}$, we can thus compute using Proposition \ref{prop:hitting time prob bounds}$(i)$ that
\begin{align}\label{eqn:RW exit time from Cr}
\begin{split}
\pr{\tau_{B_{r-1}} \leq t}
%&\leq \mathbb P(\exists v\in C_r:H_v\leq t)\\
\leq \sum_{v\in C_r}\mathbb P(H_v\leq t) 
 &\leq \#C_r\exp\left\{-\left(1-\frac{1}{(\log  r)^m}\right)r \left[ \log \left( \frac{r}{et}\right) + \log \left(1-\frac{1}{(\log r)^m}\right) \right]\right\} \\
%&\leq (\log r)^K\exp\left\{-\left(1-\frac{1}{(\log  r)^m}\right)r \left[ \log \left( \frac{r}{et}\right) -\frac{2}{(\log r)^m} \right]\right\} \\
&\leq \exp\left\{K \log \log r-\left(1-\frac{3}{(\log  r)^m}\right)r \log \left( \frac{r}{et}\right)\right\},
\end{split}
\end{align}

\textbf{STEP 2.} We now define $\tilde r(t):=r(t)(\log t)^{-\tilde f}$, $\tilde R(t):=r(t)(\log t)^{\tilde p}$ and for $0\leq n \leq N_t := \lceil \log (\tilde{R}(t))^{m+\delta} \rceil$ set $$r_n := \left(1 - \frac{1}{(\log \tilde{R}(t))^m} \right)^n \tilde{R}(t).$$ Then $r_0=\tilde R(t)$ and for $t$ large enough we have 
\begin{align*}
r_{N_t}\leq \exp\left(-\frac{\lceil \log (\tilde{R}(t))^{m+\delta} \rceil}{\log (\tilde R(t))^m}\right) \tilde R(t) &\leq \exp \left\{\frac{-1}{2}(\log (\tilde R(t)))^\delta \right\} \tilde R(t) \leq \tilde r(t).
\end{align*}
Therefore, using \eqref{eqn:hitting time calc}, \eqref{eqn:RW exit time from Cr} and a union bound, we have that
\begin{align*}
&\prb{\exists n \leq N_t: \pr{\tau_{B_{r_n-1}} \leq t} \geq \exp\left\{K(\log \log r_n)-\left(1-\frac{3}{(\log  r_n)^m}\right)r_n \log \left( \frac{r_n}{et}\right)\right\}} \\
&\qquad \leq \prb{\exists n \leq N_t: \#C_{r_n} \geq (\log r_n)^K} \\
&\qquad \leq \sum_{n \leq N_t}c(\log r_n)^{-(k-\epsilon)(\beta - 1-\varepsilon)}
\end{align*}
Applying the definition of $r_n$ and the asymptotic $\log (1-x) \geq -2x$ as for all sufficiently small $x$, this is upper bounded by
\begin{align*}
\sum_{n \leq N_t}c\left[\log(\tilde R(t))+n\log \left(1-\frac{1}{(\log \tilde R(t))^m}\right)\right]^{-(k-\epsilon)(\beta-1-\varepsilon)} 
&\leq \sum_{n \leq N_t}c\left[\log(\tilde R(t))-n\frac{2}{(\log \tilde R(t))^m}\right]^{-(k-\epsilon)(\beta-1-\varepsilon)}\\
&\leq c N_t \left[\log \tilde{R}(t) \left( 1- \frac{4}{(\log \tilde{R}(t))^{1-\delta}} \right)\right]^{-(k-\epsilon)(\beta - 1-\varepsilon)} \\
&\leq C \left[ \log (\tilde{R}(t)) \right]^{m+\delta-(k-\epsilon)(\beta - 1-\varepsilon)}.
\end{align*}
Since $k > \frac{m+\delta}{\beta-1-\varepsilon}+2\varepsilon$  this probability goes to zero as $t \to \infty$, i.e. with high $\mathbf P$-probability
\begin{align}\label{eqn:donowh}
\pr{\tau_{B_{r_n-1}} \leq t} &\leq \exp\left\{K(\log \log r_n)-\left(1-\frac{3}{(\log  r_n)^m}\right)r_n \log \left( \frac{r_n}{et}\right)\right\} \ \ \ \text{for all  }n\leq N_t.
\end{align} 
\textbf{STEP 3.} Moreover, on the event in \eqref{eqn:donowh}, again using that $\log (1-x) \geq -2x$ for small $x$ we have for all $r \in [\tilde{r}(t), \tilde{R}(t)]$ with $r \in [r_{n+1}, r_n]$ that
\begin{align*}%\label{eqn:exphitting}
\pr{\tau_{B_{r-1}} < t} \leq \pr{\tau_{B_{r_{n+1}-1}} < t} &\leq \exp\left\{K(\log \log r_{n+1})-\left(1-\frac{3}{(\log r_{n+1})^m}\right)r_{n+1} \log \left( \frac{r_{n+1}}{et}\right)\right\}\nonumber \\
&\leq\exp\left\{ K\log \log r-r_n\left(1-\frac{1}{(\log \tilde R(t))^m}\right)\log \left(\frac{r_{n+1}}{et}\right)+\frac{3r\log \left(\frac{r}{et}\right)}{(\log r_{n+1})^m}\right\}\nonumber\\
%&\leq \exp\left\{K\log \log r-r\log \left(\frac{r}{et}\left(1-\frac{1}{(\log \tilde R(t))^m}\right)\right)+\frac{r\log\left(\frac{r}{t}\right)}{(\log \tilde R(t))^m}+\frac{3r\log \left(\frac{r}{et}\right)}{( \log r_{n+1})^m}\right\}\nonumber\\
&\leq \exp\left\{ -r\log \left( \frac{r}{et}\right) + K(\log \log r)+ \frac{2r}{(\log \tilde{R}(t))^m}+ \frac{r \log \left( \frac{r}{t}\right)}{(\log \tilde{R}(t))^m}   + \frac{6r\log \left( \frac{r}{et}\right)}{(\log r)^m}  \right\}.
\end{align*}
 Then, since $m>\tilde p+1$, $r \in [\tilde{r}(t), \tilde{R}(t)]$ and $r_{n+1} \in [\frac{1}{2}\tilde{r}(t), \tilde{R}(t)]$ we have that $\log r_{n+1} \geq \frac{1}{2} \log r(t) \geq \frac{1}{2} \log t \ \forall n$ provided that $t$ is sufficiently large, so it follows that
\begin{align*}
K(\log \log r)+ \frac{2r}{(\log \tilde{R}(t))^m}+ \frac{r \log \left( \frac{r}{t}\right)}{(\log \tilde{R}(t))^m}   + \frac{6r \log \left( \frac{r}{et}\right)}{(\log r_{n+1})^m}  = O\left( \frac{r(t) (\log t)^{\tilde p+1}}{(\log t)^m} \right) \leq o(r(t)),
\end{align*}
which establishes the result. 
\end{proof}

\subsection{Time reversal for Feynman-Kac formula on $\Ti$}

Here we record the following well-known straightforward proposition which will be useful for manipulating the Feynman-Kac formula of \eqref{eqn:FK formula}.

\begin{prop}[Time reversal]\label{prop:FK time reversal}
For almost every realisation of $\Ti$, it holds for all $v \in \Ti$ that
\begin{align*}
    \mathbb E_0\left[\exp\left\{\int_0^t\xi(X_s)~\text{d} s\right\}\mathbbm 1\{X_t=v\}\right] = \mathbb E_v\left[\exp\left\{\int_0^t\xi(X_s)~\text{d} s\right\}\mathbbm 1\{X_t=0\}\right].
\end{align*}
\end{prop}
\begin{proof}
The proof is just a computation and applies to the Feynman-Kac formula on any locally finite graph. Let $\pi(X_{[0,t]})$ denote the path that consists of all the steps taken by the random walk $(X_s)_{s\geq 0}$  between times $0$ and $t$. Note that 
\begin{align}\label{eqn:path decomp}
\mathbb E_0\left[\exp\left\{\int_0^t\xi(X_s)~\text{d} s\right\}\mathbbm 1\{X_t=v\}\right] = \sum_{\gamma: 0 \to v} f(\gamma, t) 
\end{align}
where, if $\gamma$ denotes the path $0=v_0, \ldots, v_n = v$, then (in what follows we cancelled $\pr{\pi(X_{[0,t]}) = \gamma } = \prod_{i=0}^{n-1} \frac{1}{\deg v_i}$ with the terms $\deg v_i$ in the density of each $(s_i)_{i=0}^{n-1}$):
\begin{align*}
&f(\gamma, t) = \mathbb E\left[\exp\left\{\int_0^t\xi(X_s)~\text{d} s\right\}\mathbbm 1\{\pi(X_{[0,t]}) = \gamma \}\right] \\
&= \underbrace{\int_0^{\infty} \ldots \int_0^{\infty}}_{\text{from } 1 \text{ to } n-1}  e^{(\xi(v_i) - \deg v_i)s_i} \left[\int_0^{t-\sum_{i=1}^{n-1}s_i} e^{(\xi(v_0) - \deg v_0)s_0} \mathbbm{1}\left\{s_0>0, \sum_{i=0}^{n-1}s_i<t\right\}  ds_0 \cdot e^{\xi(v_n) - \deg v_n)(t-\sum_{i=0}^{n-1}s_i)} \right]\prod_{i=1}^{n-1} ds_i.
\end{align*}
We apply the substitution $s_n = t-\sum_{i=0}^{n-1}s_i = t-\sum_{i=1}^{n-1}s_i - s_0$ to the integral over $s_0$ to deduce that the expression in the square brackets is equal to
\begin{align*}
\int_0^{t-\sum_{i=1}^{n}s_i} e^{(\xi(v_n) - \deg v_n)s_n} \mathbbm{1}\left\{s_n>0, \sum_{i=1}^{n}s_i<t\right\}  ds_n \cdot e^{(\xi(v_0) - \deg v_0)(t-\sum_{i=1}^{n}s_i)}.
\end{align*}
In other words, we can perfectly exchange the roles played by $(v_n, s_n)$ and $(v_0, s_0)$ to deduce that, if $\overset{\leftarrow}{\gamma}$ is the reversal of the path $\gamma$, then $f(\gamma, t) = f(\overset{\leftarrow}{\gamma}, t)$, from which the result follows by \eqref{eqn:path decomp}.
\end{proof}

%%%%%%%%%%%%%%%%%%%%%%%%%%%%%%%%%%%%%%%%%%%%%%%%
%%%%%%% Section: Potential Estimates
%%%%%%%%%%%%%%%%%%%%%%%%%%%%%%%%%%%%%%%%%%%%%%%%
\section{Potential estimates}
\label{sec:Potential}

\subsection{Extremal values of the potential}

The following lemma bounds the maximal potential on a ball.

\begin{lem}\label{cor:whppotbo}
For any $\epsilon > 0$, there exists $C< \infty$ such that for all $r>1$ and $\lambda >1$,
\begin{align*}
% P\left(\sup_{v\in A_r}\xi(v)\geq r^{d/\alpha}\lambda \right) &\leq 2\lambda^{-\frac{\alpha}{1+\beta d}} \\
P\left(\sup_{v\in B_r}\xi(v)\geq r^{d/\alpha}\lambda \right) &\leq C\lambda^{-\frac{\alpha - \epsilon}{d}} \text{ under Assumption \ref{assn:whp}}, \\
% P\left(\sup_{v\in A_r}\xi(v)\geq r^{d/\alpha}\lambda \right) &\leq 2\lambda^{-\frac{\alpha}{1+\beta d}} \\
P\left(\sup_{v\in B_r}\xi(v)\geq r^{2/\alpha}\lambda \right) &\leq C\lambda^{-(\alpha - \epsilon)} \text{ under Assumption \ref{assn:as}}.
\end{align*}
\end{lem}

\begin{proof}
Set $p = \frac{\alpha}{\beta}$. Using Proposition \ref{prop:AR vol growth and exit time}$(i)$, under Assumption \ref{assn:whp} we compute
\begin{align*}
&P\left(\exists v\in B_r: \xi(v)>r^{\frac{1}{\alpha}\frac{\beta}{\beta-1}}\lambda \right) \leq r^d\lambda^{p} P\left(\xi(O)>r^{\frac{d}{\alpha}}\lambda\right)  +  P\left(\#B_r \geq r^{d}\lambda^{p}\right) <\lambda^{p-\alpha}+ c\lambda^{-p(\beta-1-\epsilon)} \leq C\lambda^{-\frac{\alpha - \epsilon'}{ d}}.
\end{align*}
Under Assumption \ref{assn:as}, we instead take $p=\frac{\epsilon}{2}$ and choose $K>2\alpha \epsilon^{-1}$ in Proposition \ref{prop:AR vol growth and exit time}$(ii)$ to get
\begin{align*}
&P\left(\exists v\in B_r: \xi(v)>r^{\frac{2}{\alpha}}\lambda \right) \leq r^2\lambda^{p} P\left(\xi(O)>r^{\frac{2}{\alpha}}\lambda\right)  +  P\left(\#B_r \geq r^{2}\lambda^{p}\right) <\lambda^{p-\alpha}+ c\lambda^{-pK} \leq C\lambda^{-(\alpha - \epsilon)}.
\end{align*}
\end{proof}

\begin{lem}\label{lem:xi in Ar}
$P$-almost surely, we have for any $\varepsilon >0$ that
\begin{align*}
\lim_{r \to \infty} \frac{\sup_{v\in B_r}\xi(v)}{r^{\frac{d}{\alpha}}(\log r)^{\frac{d + \epsilon}{\alpha}}}=0 \text{ under Assumption \ref{assn:whp}}, \hspace{1cm} \lim_{r \to \infty} \frac{\sup_{v\in B_r}\xi(v)}{r^{\frac{2}{\alpha}}(\log r)^{\frac{1 + \epsilon}{\alpha}}}=0 \text{ under Assumption \ref{assn:as}}.
\end{align*}
\end{lem}
\begin{proof} We give the proof under Assumption \ref{assn:whp}; the proof under Assumption \ref{assn:as} is identical. Let $\epsilon>0$ be arbitrary and choose $\epsilon' \in (0, \epsilon)$. Using Lemma \ref{cor:whppotbo} we deduce
\begin{align*}
P\left(\exists v\in B_r: \xi(v)> r^{\frac{d}{\alpha}}(\log r)^{\frac{d+\epsilon'}{\alpha}}\right) = (\log r)^{-1-\frac{\varepsilon'}{2d}},
\end{align*}
with $\epsilon'=\frac{\epsilon}{1+\beta^2/(\beta-1)}>0$. 
This probability is summable along $r_n=2^n$, and thus by Borel-Cantelli and monotonicity if $r \in [r_n, r_{n+1}]$, we deduce that $P$-almost surely 
\begin{equation*}
\lim_{n \to \infty} \frac{\sup_{v\in B_{r_n}}\xi(v)}{r_n^{\frac{d}{\alpha}}(\log r_n)^{\frac{d+\epsilon'}{\alpha}}}\leq 1, \hspace{1cm} \lim_{r \to \infty} \frac{\sup_{v\in B_{r}}\xi(v)}{r^{\frac{d}{\alpha}}(\log r)^{\frac{d+\epsilon'}{\alpha}}}\leq 2^{\frac{d+1}{\alpha}}.
\end{equation*}
Since $\epsilon' < \epsilon$ this gives the result.
\end{proof}

Define $\tilde B_r = \{v\in B_r:\deg(v)\leq 4\}$. Given $i > 0$, let $\xi^{(i)}_r$ denote the $i^{th}$ highest value of $\xi$ on $B_r$ and $\tilde{\xi}^{(i)}_r$ denote the $i^{th}$ highest value of $\xi$ on $\tilde{B}_r$.

\begin{lem}\label{cor:sup xi LB}%\label{prop:potential probs}
$P$-almost surely under Assumption \ref{assn:whp}, for $i=1,2,3$ and any $\epsilon > 0$,
\begin{align*}
\liminf_{r \to \infty} \left( \frac{{\xi}^{(i)}_r}{r^{\frac{d}{\alpha}} (\log \log r)^{\frac{-(2\beta - 1 + \epsilon)}{\alpha (\beta - 1)}}} \right) = \infty, \hspace{1cm} \liminf_{r \to \infty} \left( \frac{\tilde{\xi}^{(i)}_r}{r^{\frac{d}{\alpha}} (\log \log r)^{\frac{-(2\beta - 1+ \epsilon)}{\alpha (\beta - 1)}}} \right) = \infty.
\end{align*}
\end{lem}
\begin{proof}
It is sufficient to prove the result for $\tilde{\xi}^{(3)}$ only. By Lemma \ref{lem:deg4set}, $\#\tilde{B}_r \geq r^{d} (\log \log r)^{\frac{-(\beta + \epsilon)}{\beta - 1}}$ for all sufficiently large $r$, $\bPb$-almost surely. When this happens, we can stochastically dominate by the case $\#\tilde{B}_r = r^{d} (\log \log r)^{\frac{-(\beta + \epsilon)}{\beta - 1}}$ to deduce that
\begin{align}\label{eqn:pot LB calc}
\begin{split}
\p{\tilde{\xi}^{(3)}_r \leq r^{\frac{d}{\alpha}} \lambda^{-1}} &\leq \binom{r^{d} (\log \log r)^{\frac{-(\beta + \epsilon)}{\beta - 1}}}{3} \left( 1 - r^{-d} \lambda^{\alpha}\right)^{r^{d} (\log \log r)^{\frac{-(\beta + \epsilon)}{\beta - 1}}-3} (r^{-d} \lambda^{\alpha})^3 \leq Ce^{-c(\log \log r)^{\frac{-(\beta + \epsilon)}{\beta - 1}}\lambda^{\alpha}}.
\end{split}
\end{align}
Setting $\lambda_r = 2c^{-1}(\log \log r)^{\frac{2\beta -1 + \epsilon}{\alpha(\beta - 1)}}$ gives the result using Borel-Cantelli along the sequence $r_n = 2^n$ and monotonicity as in \cref{lem:xi in Ar}.
\end{proof}

\subsection{Sites of high potential}\label{sctn:sites of high potential}
It will also be useful to control some properties of the sites with atypically high potential.

Set $p=q+2, z=\frac{d\alpha}{\alpha - d}$ and $C=\frac{pd}{\alpha}+z+1$. Let $R(t) := r(t) (\log t)^{p}$. We define the set 
\begin{align}\label{eqn:Ft def}
F_t:=\{v\in B_{R(t)}:\xi(v) \geq R(t)^{\frac{d}{\alpha}} (\log t)^{-C} = r(t)(\log t)^{-(z+1)}\}.
\end{align}
 Also let $$E_t := \{v \in B_{R(t)} \setminus F_t: \exists z \in F_t \text{ such that } \xi(z) - \deg (z) \leq \xi(v) - \deg (v)\}.$$
 We now prove several estimates regarding these sets.

\begin{prop}\label{prop:max degree Ft Et}
Under Assumption \ref{assn:whp} it holds that:
\begin{enumerate}[(i)]
\item There exists $D<\infty$ such that  $\#F_t\leq (\log t)^D$ eventually almost surely.
    \item There exists $B <\infty$ such that $\sup_{v \in E_t \cup F_t} \deg v \leq (\log t)^{B}$ eventually almost surely.
\end{enumerate}
\end{prop}
\begin{proof}
\begin{enumerate}[(i)]
\item Define a related set
$$\tilde{F}_t :=\left \{v \in V(B_{2R(t)}): \xi (v) \geq \frac{1}{2}R(t)^{\frac{d}{\alpha}} (\log t)^{-C} \right\}.$$

We have 
$$\# \tilde{F}_t=\sum_{v\in V(B_{2R(t)})}\mathbbm 1\left\{\xi(v)\geq \frac{1}{2}R(t)^{d/\alpha}(\log t)^{-C}\right\}.$$
Set 
\begin{align*}
p_t:= \mathcal P\left(\xi(O)\geq \frac{1}{2}R(t)^{d/\alpha}(\log t)^{-C}\right) &\leq 2^{\alpha}R(t)^{-d}(\log t)^{C\alpha}.
\end{align*} 
Since $\# V(B_{2R(t)})\leq R(t)^d (\log R(t))^{\frac{1+\varepsilon}{\beta-1}}$ eventually almost surely (by Corollary \ref{cor:volume bounds}) and 
we (almost surely) have that
$\#\tilde{F}_t \overset{s.d.}{\preceq} \textsf{Binomial}(R(t)^d (\log R(t))^{\frac{1+\varepsilon}{\beta-1}},p_t)$ for all sufficiently large $t$. In thise case, using a Chernoff bound we compute  \begin{align*}
P(\#\tilde{F}_t\geq  (\log t)^D)&\leq \frac{E[\exp(\theta \#\tilde F_t)]}{\exp(\theta  (\log t)^D)}=\frac{\exp(R(t)^d (\log R(t))^{\frac{1+\varepsilon}{\beta-1}}2^{\alpha}R(t)^{-d}(\log t)^{\alpha C}(e^\theta-1))}{\exp(\theta (\log t)^D)},
\end{align*} 
and for $D>\alpha C+\frac{1+\varepsilon}{\beta-1}$ the right hand side is for $t$ sufficiently large upper bounded by $\exp(-C'(\log t)^D)$. Now for $t_n:=2^n$ we have
\begin{equation*}
    \sum_{n=0}^\infty P(\#\tilde{F}_{t_n}\geq  (\log t_n)^D)<\infty,
\end{equation*}
and thus by Borel-Cantelli, $\#\tilde{F}_{t_n}\leq  (\log t_n)^D$ eventually almost surely. Finally, for $t\in [t_n,t_{n+1}]$, note that $F_t \subset \tilde{F}_{t_{n+1}}$ for all sufficiently large $t$, which implies the result. 
    \item Since the potential is independent of $\Ti$, it follows that, given $\Ti$, each of the points in $F_t$ is uniform on $B_{R(t)}$. Now note that, if $T$ is a finite tree and $U$ a uniform vertex of $T$, we have that
\[
\sum_{v \in V(T)} \deg v = 2(\#V(T)-1)
\]
so that $\mathbf E[\deg U] \leq 2$. Therefore, letting $\deg_t (v) = \sum_{u \in V(B_{R(t)})} \mathbbm{1}\{u \sim v\}$ and $U_t$ a uniform vertex of $B_{R(t)}$, we have from Markov's inequality that
\[
\prb{ \deg_t (U_t) \geq x} \leq 2x^{-1}.
\]
Note that $\deg v = \deg_t v$ unless $v \in \partial B_{R(t)}$. Therefore, if $Z^*_r$ denotes the size of generation $r$, we can choose $\epsilon > 0$ small enough that we have from \cite[Proposition 2.2]{CroydonKumagai08} and Proposition \ref{prop:AR vol growth and exit time}$(iii)$ that
\begin{align}\label{eqn:uniform degree tail bound}
\prb{\deg (U_t) \geq x} &\leq \prb{\deg_t (U_t) \geq x} + \prb{Z^*_{R(t)} \geq R(t)^{\frac{1}{\beta -1} + \epsilon}} + \prb{V(B_{R(t)}) \leq R(t)^{d-\epsilon}} \nonumber \\
&\qquad + \prcondb{U_t \in Z^*_{R(t)}}{Z^*_{R(t)} < R(t)^{\frac{1}{\beta -1} + \epsilon}, V(B_{R(t)}) > R(t)^{d-\epsilon}}{} \nonumber
\\
%&\leq \prb{\deg_t (U_t) \geq x} + \prb{Z^*_{R(t)} \geq R(t)^{\frac{1}{\beta -1} + \epsilon}} + \prb{V(B_{R(t)}) \leq R(t)^{d-\epsilon}} + R(t)^{-\left(d - \frac{1}{\beta - 1} - 2\epsilon\right)}\nonumber \\
&\leq 2x^{-1} + C(R(t))^{-\epsilon(\beta - 1- \epsilon)}.
\end{align}
Now, using (i) and \eqref{eqn:uniform degree tail bound}, we can compute using a union bound that almost surely, for all sufficiently large $t$:
\begin{align*}
    P\left(\sup_{v \in F_t} \deg v \geq (\log t)^{B'}\right)\leq (\log t)^D \left(2(\log t)^{-B'}+C(R(t))^{-\varepsilon(\beta-1-\varepsilon)}\right)
\end{align*}
eventually almost surely. For $B>D+1+\varepsilon$, we can choose $B' \in (D+1+\varepsilon, B)$ and the right hand side is summable along $t_n=2^n$; that is, by Borel-Cantelli eventually almost surely $\sup_{v \in F_{t_n}} \deg v \leq (\log t_n)^{B'}$ and for $t\in[t_n,t_{n+1}]$ 
$$\sup_{v \in F_t} \deg v \leq \sup_{v \in F_{t_{n+1}}} \deg v \leq (\log t_{n+1})^{B}\leq C (\log t)^{B'} \leq (\log t)^{B}.$$ 
If $v \in E_t$, there exists  $z \in F_t$ such  that $\xi(v) - \deg v \geq \xi(z) - \deg z$. But also $\xi (z) \geq \xi (v)$, so that
\[
\deg v \leq \xi(v) - \xi(z)  + \deg z \leq \deg z,
\]
which proves the claim for $v \in E_t$ as well.
\end{enumerate}
\end{proof}

We will also need to consider the following slightly bigger set. For fixed  $0<\delta<\frac{d}{3\alpha}\frac{q}{q+1}$, let us also define the following set 
\begin{equation}\label{eqn:Gt def}
G_t= G_t(\delta) := \{v \in V(B_{R(t)}): \xi (v) \geq R(t)^{\frac{d}{\alpha}-\delta} \}.
\end{equation}

\begin{lem}\label{lem:useful things for Ft and Gt}
Let Assumption \ref{assn:whp} hold. 
Then, eventually almost surely, 
\begin{enumerate}[(i)]
\item $\#G_t (\delta) \leq t^{2\alpha(q+1)\delta}$.
    \item $G_t$ is disconnected, i.e. if $u \in G_t$ then for all $v\in V(B_{R(t)})$ with $v\sim u$ it holds $v\in G_t^c$. (Consequently the same is true for $F_t$).
\end{enumerate}
\end{lem}
\begin{proof}
\begin{enumerate}[(i)]
\item Fix $\delta > 0$ and define a related set
\begin{equation}\label{eqn:hat Gt def}
\tilde{G}_t := \{v \in V(B_{R(t)}): \xi (v) \geq \frac{1}{2}R(t)^{\frac{d}{\alpha}-\delta} \}.
\end{equation}
As in Proposition \ref{prop:max degree Ft Et}$(i)$, we have that (almost surely)  $\#\tilde{G}_t \overset{s.d.}{\preceq} \textsf{Binomial}(R(t)^d (\log R(t))^{\frac{1+\varepsilon}{\beta-1}},p_t)$ for all sufficiently large $t$, where here 
\begin{align*}p_t:= \mathcal P\left(\xi(O)\geq \frac{1}{2}R(t)^{\frac{d}{\alpha}-\delta}\right) &\leq 2^\alpha R(t)^{-d+\alpha \delta}.\end{align*}
Using a Chernoff bound we compute  \begin{align*}
P(\#\tilde{G}_t\geq  t^{2\alpha(q+1) \delta})&\leq \frac{E[\exp(\theta \#\tilde G_t)]}{\exp(\theta t^{2\alpha(q+1) \delta})}=\frac{\exp(R(t)^d (\log R(t))^{\frac{1+\varepsilon}{\beta-1}}R(t)^{-d+\alpha \delta}(e^\theta-1))}{\exp(\theta t^{2\alpha (q+1)\delta})},
\end{align*} 
and the right hand side is for $t$ sufficiently large upper bounded by $\exp(-Ct^{2\alpha(q+1) \delta})$. Now for $t_n:=n$ we have
\begin{equation*}
    \sum_{n=0}^\infty P(\#\tilde{G}_{t_n}\geq  t_n^{2\alpha(q+1) \delta})<\infty,
\end{equation*}
and thus by Borel-Cantelli $\#\tilde{G}_{t_n}\leq  t_n^{2\alpha(q+1) \delta}$ eventually almost surely. Finally, for $t\in [t_n,t_{n+1}]$, note that $G_t \subset \tilde{G}_{\lceil t \rceil}$ for all sufficiently large $t$, which implies the result. 
\item Fix some $v \in V(B_{R(t)})$, suppose that $u \in V(\Ti)$ satisfies $v \sim u$. Since the potentials at different vertices are independent, we have that
\[
\mathcal P\left(\xi (v) \wedge \xi (u) \geq \frac{1}{2^q} R(t)^{\frac{d}{\alpha}-\delta}\right) \leq c (R(t))^{-2\left( d-\alpha\delta \right)}.
\]
The number of neighbouring vertices in $B_{R(t)}$ is given by the number of edges in the ball, which is upper bounded by the number of vertices, i.e. by $\#V(B_{2R(t)})$. 
Therefore, since $\#V(B_{2R(t)}) \leq R(t)^d (\log t)^{\frac{1+\epsilon}{\beta - 1}}$ eventually almost surely by Corollary \ref{cor:volume bounds}, we get from a union bound that for all such $t$, whenever $\delta< \frac{d}{2\alpha} d$, we can choose $\epsilon, \epsilon' > 0$ small enough that
\begin{align*}
\p{\exists v,u\in V(B_{R(t)}): v \sim u, [\xi (v) \wedge \xi (u)] \geq \frac{1}{2^q} R(t)^{\frac{d}{\alpha}-\delta}} 
&\leq R(t)^d (\log t)^{\frac{1+\epsilon}{\beta - 1}}  c (R(t))^{-2\left( d-\alpha \delta\right)} \\&\leq c (r(t))^{2\alpha\delta + \epsilon - d} \leq t^{-\epsilon '}.
\end{align*}
Taking $t_n = 2^n$, and $\tilde{G}_t$ as in $(i)$, we obtain since $d>1$ from Borel-Cantelli that $$P\left(\exists (v,u) \in V(B_{2R(t_n)}): u\sim v, u\in \tilde G_{t_n}, v\in \tilde G_{t_n} \text{ i.o.}\right)=0.$$ We can extend to all $t$ with $G_t$ in place of $\tilde{G}_t$ since $G_t \subset \tilde{G}_{2^{\lceil\log_2 t\rceil}}$ for all sufficiently large $t$, which establishes the result.
\end{enumerate}
\end{proof}

\subsection{Gap between highest values}\label{sctn:gap potential}
To prove Theorems \ref{thm:main a.s. localisation intro} and \ref{thm:main whp localisation intro} we will also need to control the gap between the highest values of $\xi$ on a ball of radius $r$. We therefore define 
$$
Z_{B_r} := \arg_{z \in B_r} \max \{\xi(z) \}, \hspace{1cm} g_{B_r} := \xi(Z_{B_r}) - \max_{z \in B_r, z \neq Z_{B_r}} \{\xi(z) \},$$
and 
$$\tilde{Z}_{B_r} := \arg_{z \in B_r} \max \{\xi(z) - \deg(z) \},  \hspace{1cm}
\tilde{g}_{B_r} := \xi(\tilde Z_{B_r}) - \deg(\tilde Z_{B_r}) - \max_{z \in B_r, z \neq \tilde Z_{B_r}} \{\xi(z) - \deg(z) \}.$$

Note that, since the law of $\xi$ is non-atomic, these are well-defined, almost surely.

\begin{lem}\label{lem:gap to infinity}
Under Assumption \ref{assn:whp}, for any $\epsilon > 0$ there $P$-almost surely exists $c>0$ such that, for all $r \geq 1$,
\[
g_{B_r} \geq r^{\frac{d}{\alpha}} (\log r)^{\frac{-(1+\epsilon)}{\alpha}}.
\]
Moreover, $g_{B_r} \geq r^{\frac{d}{\alpha}} (\log r)^{-\epsilon}$ with high $P$-probability as $r \to \infty$.
\end{lem}
\begin{proof}
We will start by proving the following: let $(X_i)_{i=1}^n$ be i.i.d. Pareto random variables with parameter $\alpha$, and let $m_n = \arg \max_{1 \leq i \leq n} X_i$, $g_n = \sup_{1 \leq i \leq n}X_i - \sup_{1 \leq i \leq n, i \neq m_n} X_i$. Then 
\begin{equation}\label{eqn:gap prob}
\mathcal P (g_n \leq y) \leq (n y^{-\alpha}+1) e^{-y^{-\alpha}(n-1)}.
\end{equation}

Recall that the distribution and density functions for the Pareto distribution are given by $F(x) = 1 - x^{-\alpha}$ and $f(x) = \alpha x^{-(\alpha +1)}$. The maximum of $X_1, \ldots, X_n$ has density $nf(x)F(x)^{n-1}$, so that
\begin{align*}
    \mathcal P(g_n \leq y) &\leq F(2y)^n + \int_{2y}^{\infty} nf(x)F(x)^{n-1} \left[ 1 - \left(\frac{F(x-y)}{F(x)}\right)^{n-1} \right] dx \\
    &= (1-(2y)^{-\alpha})^n + \int_{2y}^{\infty} \alpha n x^{-(\alpha +1)} \left[ (1-x^{-\alpha})^{n-1} - (1-(x-y)^{-\alpha})^{n-1} \right] dx \\
    %&\leq (1-(2y)^{-\alpha})^n + \int_{2y}^{\infty} \alpha n x^{-(\alpha +1)} y(n-1) \left[ (x-y)^{-(\alpha+1)}(1-(x-y)^{-\alpha})^{n-2} \right]dx \\
    &\leq (1-(2y)^{-\alpha})^n + n y^{-\alpha}(1-(2y)^{-\alpha})^{n-1},
   % &\leq e^{-(2y)^{-\alpha}n} + n y^{-\alpha} e^{-y^{-\alpha}(n-1)}.
\end{align*}
i.e. \eqref{eqn:gap prob}. In particular, \eqref{eqn:gap prob} decays to zero as $y, n \to \infty$ provided $y = o (n^{\frac{1}{\alpha}})$. Also, if $y = \frac{1}{4} n^{\frac{1}{\alpha}} (\log n)^{\frac{-1}{\alpha}}$, then this probability is upper bounded by $(1+4^{\alpha}\log n) n^{-2^{\alpha}}$ for all sufficiently $n$, which is summable in $n$. Therefore, since the volumes of the sets $(B_r)_{r \geq 1}$ are strictly increasing, we deduce from Borel-Cantelli that
\[
P\left(g_{B_r} \leq \frac{1}{3}(\#B_r)^{\frac{1}{\alpha}} (\log (\#B_r)^{\frac{-1}{\alpha}} \text{ i.o.}\right)=0.
\]
Since it also follows from Corollary \ref{cor:volume bounds} that there exists $c>0$ such that
\[
\prb{\#B_r \leq cr^{\frac{\beta}{\beta - 1}} (\log \log r)^{-\frac{\beta + \epsilon}{\beta - 1}} \text{ i.o.}} =0,
\]
we deduce from Borel-Cantelli that for any $\epsilon > 0$
\[
g_{B_r} \geq c'r^{\frac{\beta}{\alpha (\beta - 1)}} (\log r)^{\frac{-(1+\epsilon)}{\alpha}}
\]
eventually $P$-almost surely.
\end{proof}

Finally, the following lemma shows that eventually the maximisers $Z_{B_r}$ and $\tilde Z_{B_r}$ coincide.

\begin{lem}\label{lem:equalZonB}
Take any $\epsilon > 0$. Under Assumption \ref{assn:whp}, it holds eventually almost surely that
\begin{align*}
Z_{B_r} &= \tilde{Z}_{B_r}, \hspace{1cm} \tilde{g}_{B_r} \geq r^{\frac{\beta}{\alpha (\beta - 1)}} (\log r)^{\frac{-(1+\epsilon)}{\alpha}}.
\end{align*}
\end{lem}
\begin{proof}
First note that it follows from Lemma \ref{cor:sup xi LB} and Lemma \ref{prop:max degree Ft Et}$(ii)$ that ${2\deg Z_{B_r} \leq g_{B_r}}$ eventually almost surely. This proves the first statement, since on this event we have for all $z \in B_r$ with $z \neq  Z_{B_r}$ that
\begin{align*}
    \left[ \xi ( {Z}_{B_r}) -\deg ( {Z}_{B_r}) \right] - \left[ \xi (z) -\deg z \right] &= \left[ \xi ( Z_{B_r}) -\xi (z)  \right] - \left[ \deg ( {Z}_{B_r}) -\deg z \right] \geq \frac{g_{B_r}}{2},
    \end{align*}
from which it follows that $\tilde{Z}_{B_r} = {Z}_{B_r}$ (uniquely) and $\tilde g_{B_r} \geq \frac{g_{B_r}}{2}$.
\end{proof}

%%%%%%%%%%%%%%%%%%%%%%%%%%%%%%%%%%%%%%%%%%%%%%%%
%%%%%%% Section: PAM on T
%%%%%%%%%%%%%%%%%%%%%%%%%%%%%%%%%%%%%%%%%%%%%%%%
\section{PAM on $\Ti$}\label{sctn:PAM on Ti}
It follows exactly as in \cite[Section 2]{GaertnerMolchanov} for the $\Z^d$ case that \eqref{eqn:PAM} has a solution if and only if $\alpha > d$. Roughly speaking, this is because the rate of growth of the maximum potential attained on the annulus $B_{2r} \setminus B_r$ is less than the rate of decay for ($\frac{1}{t}$ times the log of) the probability of reaching the annulus, so the solution givenby the Feynman-Kac formula \eqref{eqn:FK formula} can't blow up. Since the volume fluctuations for the volume of $B_r$ are at most logarithmic in $r$ (by Proposition \ref{prop:AR vol growth and exit time}), and since most vertices have degrees of constant order, this does not affect the existence threshold obtained in the $\Z^d$ case, and the arguments carry through with minor adaptations.

Before proving the main localisation result, we prove some asymptotics for the total mass and the location of the localisation site. We then show that the solution at any single site $v$ can be well-approximated by considering trajectories that spend most of their time at a nearby ``good'' site, and then jump to $v$ just before time $t$. The strategy to prove this approximation broadly follows that used to prove analogous results for the $\Z^d$ case in \cite[Section 7]{OrtgieseRobertsIntermittency}, though we have to work harder to control the extra randomness in $\Ti$. We work under Assumption \ref{assn:whp} for the whole of Section \ref{sctn:PAM on Ti}.

\subsection{Asymptotics: proof of Theorems \ref{thm:total mass asymp intro} and \ref{thm:Zt asymp intro}}
In this subsection we prove Theorems \ref{thm:total mass asymp intro} and \ref{thm:Zt asymp intro}. Recall $\psi_t$ and $\hat{Z}_t^{(1)}$ as defined in \eqref{eqn:psit def} and \eqref{eqn:Zt def intro}.

\begin{proof}[Proof of Theorems \ref{thm:total mass asymp intro} and \ref{thm:Zt asymp intro}]
\textbf{Upper bounds.} We start with the upper bound on $U(t)$ and then on $|\hat{Z}_t^{(1)}|$. First note that for all $i \geq 1$, $\lambda, t> 0$, 
\begin{align}\label{eqn:annulus max xi lambda}
\begin{split}
\p{ \sup_{v \in B_{2^{i+1}r(t)\lambda}} \xi (v) \geq \frac{1 \wedge q}{8} 2^{i}\frac{r(t)}{5t}(\log t) \lambda} &= \p{ \sup_{v \in B_{2^{i+1}r(t)\lambda}} \xi (v) \geq 2^{\frac{(i+1)d}{\alpha}}r(t)^{\frac{d}{\alpha}}(\lambda)^{\frac{d}{\alpha}} \frac{1 \wedge q}{40} 2^{i(1-\frac{d}{\alpha}) - \frac{d}{\alpha}}\lambda^{(1-\frac{d}{\alpha})}} \\
&\leq c_q \left(2^{i(1-\frac{d}{\alpha})}\lambda^{(1-\frac{d}{\alpha})}\right)^{\frac{-(\alpha - \epsilon)}{d}}.
\end{split}
\end{align}
Now let $A_{\lambda} = \{ \exists i: \sup_{v \in B_{2^{i+1}r(t)\lambda}} \xi (v) \geq \frac{1 \wedge q}{8} 2^{i}\frac{r(t)}{5t}(\log t) \lambda \}$ and $B_{\lambda} = \{\sup_{v \in B_{r(t)\lambda}} \xi(v) \geq \frac{1}{2} a(t) \lambda \}$.
By a union bound, $\p{A_{\lambda}} \leq c\lambda^{-\frac{(\alpha-\epsilon)}{d}-1}$. Also, $\p{B_{\lambda}} \leq \lambda^{-(\frac{\alpha}{d}-1- \epsilon)}$ by Lemma \ref{cor:whppotbo}. Moreover, by Corollary \ref{cor:exitball}, it follows that, as $t \to \infty$,
\[
\prb{\pr{\tau_{B_{2^{i-1}r(t) \lambda-1}} \leq t} \leq \exp \left\{ -\frac{2^{i-1}r(t) \lambda}{5}\log \left(\frac{2^{i-1}r(t) \lambda}{et}\right) \right\} \forall t > T} \geq 1- o(1).
\]
For the rest of the proof we work on this event. We therefore have for all $t>T$ that
\begin{align*}
U(t) &\leq \exp \{ t\sup_{v \in B_{r(t)\lambda}} \xi(v) \} + \sum_{i \geq 1} \mathbb E\left[ \exp \left\{t \sup_{v \in B_{2^ir(t) \lambda}} \xi(v)\right\}\mathbbm 1\left\{2^{i-1}r(t) \lambda \leq \sup_{s\leq t}|X_s| \leq 2^ir(t) \lambda\right\}\right] \\
&\leq \exp \{ t\sup_{v \in B_{r(t)\lambda}} \xi(v) \} + \sum_{i \geq 1} \exp \left\{t \sup_{v \in B_{2^{i}r(t) \lambda}} \xi(v)\right\} \pr{\tau_{B_{2^{i-1}r(t) \lambda-1}} \leq t}.
\end{align*}
Therefore, when the events $A_{\lambda}^c$ and $B_{\lambda}^c$ additionally occur it follows that
\begin{align*}
U(t) &\leq \exp \left\{\frac{1}{2}ta(t) \right\}  + \sum_{i \geq 1} \exp \left\{ \frac{1 \wedge q}{4} \frac{2^{i-1}r(t)\lambda}{5}(\log t) \right\} \exp \left\{ -\frac{2^{i-1}r(t) \lambda}{5}\log \left(\frac{2^{i-1}r(t) \lambda}{et}\right) \right\} \leq e^{ta(t) \lambda}.
\end{align*}
 $U(t) \leq ta(t) \lambda$. Also, since $\psi_t(v) \leq \xi (v) - \frac{|v|}{t}\left(\log  \left(\frac{|v|}{et} \right) \right)$ (by respectively taking $\rho = 0$ and $\rho=1$ in the two terms in \eqref{eqn:psit def}), it additionally follows that on the event $A_{\lambda}^c$,
\begin{align*}
\sup_{v \notin B_{r(t) \lambda}} \psi_t(v) &\leq \sup_{i \geq 0} \left\{ \frac{1 \wedge q}{8} 2^{i}\frac{r(t)}{5t}(\log t) \lambda - \frac{2^{i}r(t)\lambda}{t}\left(\log  \left(\frac{2^{i}r(t)\lambda}{et} \right) \right) \right\} \\
&\leq \sup_{i \geq 0} \left\{ \frac{q2^{i}r(t) \lambda}{40t}(\log t) - \frac{q2^{i}r(t)\lambda}{2t}\left(\log  \left(t\lambda \right) \right) \right\} \leq 0.
\end{align*}
Since $\psi_t(O) \to \infty$ as $t \to \infty$, it therefore follows that $|\hat{Z}_t^{(1)}| \leq r(t) \lambda$. This establishes the upper bounds in Theorems \ref{thm:total mass asymp intro} and \ref{thm:Zt asymp intro}.

\textbf{Lower bounds.} We first make two observations.
\begin{enumerate}
    \item[(a)] We have from Lemma \ref{cor:whppotbo} that with probability at least $1-c\lambda^{\frac{-(d-\epsilon)}{2}}$,
\begin{align*}
\sup_{v\in B_{r(t)\lambda^{-1}}}\psi_t(v)\leq \sup_{v\in B_{r(t)\lambda^{-1}}} \xi(v) &\leq \frac{1}{4}(r(t))^{d/\alpha} \lambda^{-d/2\alpha}.
\end{align*}
\item[(b)] Recall that $\tilde{B}_r = \{v \in B_r: \deg v \leq 4\}$. Choose some $\epsilon>0$ small enough that $\frac{(1+\epsilon)d}{4 \alpha} < \frac{1}{4} \wedge \frac{d}{2\alpha}$. By the same calculation as in \eqref{eqn:pot LB calc}, instead using that $\# \tilde{B}_r \geq r^{d} \lambda^{\frac{-\epsilon d}{4}}$ with probability at least $1-Ce^{-c \lambda^{\frac{-\epsilon}{4}}}$, we obtain that, with probability at least $1-Ce^{-c \lambda^{\frac{-\epsilon d}{8}}}$,
\[
\sup_{v\in \tilde{B}_{r(t)\lambda^{-1/4}}} {\xi} (v) \leq r(t)^{d/\alpha} \lambda^{\frac{-(1+\epsilon)d}{4 \alpha}} .
\]
\end{enumerate}
Now suppose that the two events in (a) and (b) hold. If $v\in \tilde{B}_{r(t)\lambda^{-1/4}}$ satisfies ${\xi} (v) > r(t)^{d/\alpha} \lambda^{\frac{-(1+\epsilon)d}{4 \alpha}}$, then:
\begin{align*}
\psi_t(v) > r(t)^{d/\alpha} \lambda^{\frac{-(1+\epsilon)d}{4 \alpha}} - 4 - \frac{r(t)}{t}\lambda^{-1/4}\log (r(t)\lambda^{\frac{-1}{4}}) &> \frac{1}{2}r(t)^{d/\alpha} \lambda^{\frac{-(1+\epsilon)d}{4 \alpha}} - 4 - \frac{(q+1)r(t)}{t}\lambda^{-1/4}\log t \\
&> \frac{1}{4}r(t)^{\frac{d}{\alpha}} \lambda^{\frac{-(1+\epsilon)d}{4 \alpha}} \\
&> \sup_{v\in B_{r(t)\lambda^{-1}}}\psi_t(v).
\end{align*}
This shows that $\hat{Z}_t^{(1)}$ as defined by \eqref{eqn:Zt def intro} satisfies $|\hat{Z}_t^{(1)}| \geq r(t) \lambda^{-1}$. Finally, for the lower bound on $\log U(t)$, take the same $v$ and note that from Lemma \ref{lem:Br log sum bound} and Proposition \ref{prop:hitting time prob bounds}$(ii)$, and taking $\rho = \frac{|v|}{t(\xi (v) - \deg v)} < \frac{r(t)\lambda^{-1/4}}{t(r(t)^{d/\alpha} \lambda^{\frac{-(1+\epsilon)d}{4 \alpha}} - 4)}<1$ below, it holds with probability $1-o(1)$ as $t \to \infty$ that
\begin{align*}
U(t) &\geq \sup_{\rho \in(0,1)}\left\{\pr{H_{v} \leq \rho t, X_s = v \forall s \in [H_{v}, H_{v} + (1-\rho)t]} \exp\{\xi (v) (1-\rho)t\} \right\}\\
&\geq \sup_{\rho \in(0,1)}\left\{ \pr{H_{v} \leq \rho t} \exp\{-(1-\rho)t\deg (v) + \xi (v) (1-\rho)t\}\right\} \\
&\geq\exp\{t\psi_t(v) - C|v| + O(t)\}.
\end{align*}
We deduce that $\log U(t) \geq \frac{1}{2}ta(t) \lambda^{\frac{-(1+\epsilon)d}{4 \alpha}}$ on these events, which proves the lower bound (i.e. by replacing $\lambda^{\frac{(1+2\epsilon)d}{4 \alpha}}$ with $\lambda$ for the latter event).
\end{proof}

\subsection{Solution to the PAM: proof of Theorem \ref{thm:log conv result intro}}
Let $u(t,v)$ denote the solution to the PAM at vertex $v \in \Ti$ at time $t$. In this subsection we prove Theorem \ref{thm:log conv result intro}.

Recall from \eqref{eqn:lambda intro def} that, analogously to \cite[Section 1.4]{OrtgieseRobertsIntermittency}, for $v \in \Ti$ we define
\begin{align*}
\lambda(t,v,y) &= t(\xi(y) - \deg y) - |y| \log \left( \frac{|y|}{te}\right) -  |v-y| \log \left(\frac{|v-y|}{te}\right) ,\\
\lambda(t,v) &= \sup_{y \in \Ti} \lambda(t,v,y),
\end{align*}
and set $\lambda_+(t,v) = \lambda(t,v) \vee 0$. We also fix some $\delta \in (0,1)$ and some $m \in (q, \infty)$ and set
\[
Y(t,v) = \arg \max_{y \in B_{r(t)(\log t)^{m}}} \left\{ \lambda(t,v,y) \right\}, \qquad \tilde{Y}(t,v) = \arg \max_{y \in B_{r(t)(\log t)^{m}}} \left\{ \lambda(t,v,y) + t \deg y \right\}.
\]
Since the law of $\xi$ is non-atomic, these are well-defined, almost surely.

Recall from Proposition \ref{prop:hitting time prob bounds} that if the random walk $(X_t)_{t \geq 0}$ starts at some point $y \in \Ti$ and $r=|v-y|$, then 
\begin{equation}\label{eqn:RW hitting probs upper and lower}
\left( \prod_{u \in [[y, v))} \frac{1}{\deg(u)} \right) e^{-r \left( [\log r - \log t - 1] \vee 0 \right) - O(t)} \leq \prstart{H_v < t}{y} \leq e^{-r \left( [\log r - \log t - 1] \vee 0 \right))}.
\end{equation}

We first state a couple of lemmas and then show how they can be used to prove Theorem \ref{thm:log conv result intro}.

The first lemma is similar in flavour to some of the results in Section \ref{sctn:sites of high potential}, and allows us to control the location and potential of $Y(t,v)$.

\begin{lem}\label{lem:sup in ball}
Under Assumption \ref{assn:whp}, it holds with high $P$-probability as $t \to \infty$ that
\begin{enumerate}[(i)]
    \item $\lambda(t,v,y) +t \deg y \leq 0$ whenever $v \in B_{r(t)(\log t)^{\delta}}^c$ or $y \in B_{r(t)(\log t)^{\delta}}^c$. Consequently, $Y(t,v)$ and $\tilde{Y}(t,v)$ are in $B_{r(t)(\log t)^{\delta}}$ whenever $\lambda(t,v) >0$.
    \item $\lambda(t,v) = \sup_{y \in B_{r(t)(\log t)^{\delta}}} \lambda(t,v,y)$ for all $v \in B_{r(t)(\log t)^{\delta}}$.
    \item $\sup_{v \in B_{r(t)(\log t)^{\delta}}} \xi (Y(t,v)) \leq a(t) (\log t)^{\frac{2\delta}{\alpha}}$.
    \item $\log u(t,v) \leq 0$ for all $v \in B_{r(t)(\log t)^{\delta}}^c$.
    \item For all $v \in V(\Ti)$, $\estart{\exp\left\{\int_0^t\xi(X_s)~\text{d} s\right\}\mathbbm 1\{\sup_{0\leq s\leq t} \xi(X_s) \notin B_{r(t)(\log t)^m}, X_t = v\}}{O} \leq 1$.
\end{enumerate}
\end{lem}

The second lemma allows us to control the degree of good vertices.

\begin{lem}\label{lem:Y(t,v) deg bound}
Take any $\delta, \eta >0$. Under Assumption \ref{assn:whp}, it holds with high $P$-probability as $t \to \infty$ that
\begin{align*}
\sup_{v \in B_{r(t) (\log t)^{\delta}}} \left\{ \left[ \deg \left( Y(t,v) \right) \vee \deg \left( \tilde{Y}(t,v) \right)\right] \mathbbm{1}\{\lambda (t,v) > 0\} \right\} \leq (\log t)^{\frac{d\delta}{\beta} + \eta}.
\end{align*}
\end{lem}

Finally, let $\mathcal{L} = \{u \in \Ti: \deg u = 1\}$ denote the set of leaves in $\Ti$. The third lemma will be useful as we will be able to lower bound the solution at a high degree vertex by considering trajectories that spend a lot of a time at a neighbouring leaf and then jump to the high degree vertex just before time $t$.

\begin{lem}\label{lem:high deg next to leaf}
Under Assumption \ref{assn:whp}, it holds with high $P$-probability as $t \to \infty$ that
\[
\prb{\exists v \in B_{r(t)(\log t)^{\delta}}: \deg v > a(t), \nexists u \sim v \text{ with } u\in \mathcal{L}} \to 0.
\]
\end{lem}

On the event that $\deg v > a(t)$ and $\exists u \sim v \text{ with } p(u)=v$ and $u\in \mathcal{L}$, let $\ell (v)$ denote the leftmost such choice of $u$ amongst the children of $v$.

\begin{proof}[Proof of Theorem \ref{thm:log conv result intro}, assuming Lemmas \ref{lem:sup in ball}, \ref{lem:Y(t,v) deg bound} and \ref{lem:high deg next to leaf}]
By Lemma \ref{lem:sup in ball}$(i)$ and $(iv)$, it is sufficient to show that $$\left( ta(t) \right)^{-1} \sup_{v \in B_{r(t)(\log t)^{\delta}}} |\log_+ (u(t,v)) - \sup_{y \in B_{r(t)(\log t)^{\delta}}} \left\{\lambda(t,v,y) \vee 0\right\}| \to 0$$ with high $P$-probability.

\textit{Upper bound.} For all $v \in B_{r(t)(\log t)^{\delta}}$, we deduce from Lemma \ref{lem:sup in ball}$(v)$ then \eqref{eqn:RW hitting probs upper and lower} then Lemma \ref{lem:sup in ball}$(i)$ that
\begin{align*}
u(t,v) &\leq \sum_{y \in V(\Ti) \setminus B_{r(t)(\log t)^{m}}} \estart{\exp\left\{\int_0^t\xi(X_s)~\text{d} s\right\}\mathbbm 1\{\sup_{0\leq s\leq t} \xi(X_s) = \xi(y), X_t = v\}}{O} \\
&\qquad + \sum_{y \in B_{r(t)(\log t)^m}} \estart{\exp\left\{\int_0^t\xi(X_s)~\text{d} s\right\}\mathbbm 1\{\sup_{0\leq s\leq t} \xi(X_s) = \xi(y), X_t = v\}}{O} \\
&\leq 1 + \sum_{y \in B_{r(t)(\log t)^m}} \exp \{t\xi(y)\} \pr{H_y < t} \prstart{H_v < t}{y} \\
&\leq 1 + \sum_{y \in B_{r(t)(\log t)^m}} \exp \{t\xi(y)\} \exp \{-|y| \left( [\log |y| - \log t - 1] \vee 0 \right)) - |v-y| \left( [\log |v-y| - \log t - 1] \vee 0 \right))\} \\
&\leq 1 + \#(B_{r(t) (\log t)^{m}}) \sup_{y \in B_{r(t) (\log t)^{\delta}}} \exp \left\{\left(\lambda(t,v,y) + t \deg \tilde{Y}(t,v) \right) \vee 0 \right\}.
%&\qquad + \#(B_{r(t) (\log t)^{m}} \setminus B_{r(t) (\log t)^{\delta}}) \sup_{y \in B_{r(t) (\log t)^{m}} \setminus B_{r(t) (\log t)^{\delta}}} \exp \left\{\lambda(t,v,y) + t \deg y \right\}.
\end{align*}
Since $\#(B_{r(t)(\log t)^{m}}) \leq \left(r(t)\right)^{\frac{\beta + 1}{\beta - 1}}$ with high probability as $t \to \infty$ by Proposition \ref{prop:AR vol growth and exit time}$(i)$, taking logarithms we get from Lemma \ref{lem:sup in ball}$(ii)$ that with high $P$-probability as $t \to \infty$,
\begin{align*}
    \log (u(t,v)) %&\leq \sup_{y \in B_{r(t)  (\log t)^{\delta}}} \left\{t\xi(y) - |y| \log_+ \left(\frac{|y|}{te}\right) - |v-y| \log_+ \left(\frac{|v-y|}{te}\right) \right\} + \frac{\beta + 1}{\beta - 1} \log r(t) \\
    &\leq \lambda(t,v) +t \deg \tilde{Y}(t,v) + \frac{\beta + 1}{\beta - 1} \log r(t),
\end{align*}
so that, by Lemma \ref{lem:Y(t,v) deg bound},
\begin{align*}
\sup_{v \in B_{r(t)(\log t)^{\delta}}} (t a(t))^{-1} \left(\log_+ (u(t,v)) - \lambda_+(t,v)\right) \leq (t a(t))^{-1} \left[ t(\log t)^{\frac{d\delta}{\beta} + \eta} + \frac{\beta + 1}{\beta - 1} \log r(t) \right] \to 0.
\end{align*}
Finally we can extend this to the supremum over all $v \in V(\Ti)$ using  Lemma \ref{lem:sup in ball}$(iv)$. 

\textit{Lower bound.} We break the proof into two cases, depending on $\deg v$.
\begin{enumerate}[a)]
    \item \textbf{Case 1: $\mathbf{\deg v \leq a(t).}$} By Lemma \ref{lem:sup in ball}$(i)$, we can assume wlog that $v \in B_{r(t)(\log t)^{\delta}}$.
We obtain the complementary lower bound from \eqref{eqn:RW hitting probs upper and lower} by writing 
\begin{align}\label{eqn:utv LB}
\begin{split}
u(t,v) &\geq \sup_{y \in B_{r(t)(\log t)^{\delta}}} \estart{\exp\left\{\int_0^t\xi(X_s)~\text{d} s\right\}\mathbbm 1\left\{H_y \leq \frac{t}{\log t}, X_s = y \ \forall s \in \left[\frac{t}{\log t}, t - \frac{t}{\log t}\right], X_t = v\right\}}{O} \\
%&\geq \sup_{y \in \Ti} \exp \left\{ \left( t - \frac{2t}{\log t} \right) \xi(y) \right\} \prstart{H_y\leq \frac{t}{\log t}}{O} \prstart{X_s=y \forall s\in [0,t]}{y} \prstart{H_v\leq \frac{t}{\log t}}{y} \prstart{X_s=v \forall s\in [0,\frac{t}{\log t}]}{v}\\
&\geq \sup_{y \in B_{r(t)(\log t)^{\delta}}} \exp \left\{\left( t - \frac{2t}{\log t} \right) \xi(y) \right\}  \prstart{H_y \leq \frac{t}{\log t}, X_s = y \ \forall s \in \left[H_y, t - \frac{t}{\log t}\right], X_t = v}{O}.
%&\geq \sup_{y \in \Ti} \exp \left\{\left( t - \frac{2t}{\log t} \right) \xi(y) \right\} \exp \left\{\left[\sum_{u \in [[O,y]]} \log (\deg u) \right] -|y| \left( \log_+ \left( \frac{|y|}{t} \right) + \log \log t \right)\right\} e^{-t\deg y} \\
%&\ \ \ \ \ \ \ \times\exp \left\{\left[\sum_{u \in [[y,v]]} \log (\deg u) \right] -|v-y| \left( \log_+ \left( \frac{|v-y|}{t} \right) + \log \log t \right)\right\} e^{-\frac{t}{\log t}\deg v},
\end{split}
\end{align}
To bound the probability appearing here, we use the Markov property to write
\begin{align}\label{eqn:utv LB hitting prob}
\begin{split}
&\prstart{H_y \leq \frac{t}{\log t}, X_s = y \ \forall s \in \left[H_y, t - \frac{t}{\log t}\right], X_t = v}{O} \\
&\qquad \geq \prstart{H_y \leq \frac{t}{\log t}}{O} \prcond{X_s = y \ \forall s \in \left[0, t - \frac{t}{\log t} - H_y \right]}{H_y \leq \frac{t}{\log t}}{y}\prstart{X_{\frac{t}{\log t}} = v}{y} \\
%\prstart{X_{\frac{t}{\log t}} = y}{O} \prstart{X_s = y \ \forall s \in \left[\frac{t}{\log t}, t - \frac{t}{\log t}\right]}{y}\prstart{X_{\frac{t}{\log t}} = v}{y} 
&\qquad \geq \prstart{H_y \leq \frac{t}{\log t}}{O} e^{-\left( t - \frac{t}{\log t} \right)\deg y}\prstart{H_v \leq \frac{t}{\log t}}{y} e^{- \frac{t}{\log t} \deg v}.
\end{split}
\end{align}

Combining these and using \eqref{eqn:RW hitting probs upper and lower}, we deduce that $u(t,v)$ is lower bounded by
\begin{align*}
&\sup_{y \in B_{r(t)(\log t)^{\delta}}} \exp \left\{\left( t - \frac{2t}{\log t} \right) \xi(y) \right\} \exp \left\{-\left[\sum_{u \in [[O,y]]} \log (\deg u) \right] -|y| \left( \log_+ \left( \frac{|y|}{te} \right) + \log \log t \right) \right\} e^{-t\deg y} \\
&\ \ \ \ \ \ \ \times\exp \left\{-\left[\sum_{u \in [[y,v]]} \log (\deg u) \right] -|v-y| \left( \log_+ \left( \frac{|v-y|}{te} \right)  + \log \log t \right) + O(t) \right\}e^{- \frac{t}{\log t} \deg v}.
\end{align*}
Also recall from Lemma \ref{lem:sup in ball}$(i)$ that on the event $\lambda(t,v) > 0$, $Y(t,v) \in B_{r(t)(\log t)^{\delta}}$. (If $\lambda(t,v) < 0$ then the lower bound is trivial). In this case we can therefore write, conditioning on the high probability events in Lemma \ref{lem:sup in ball}:
\begin{align*}
\log u(t,v)
&\geq \sup_{y \in B_{r(t)(\log t)^{\delta}}}  t \xi(y) -|y| \left( \log_+ \left( \frac{|y|}{te} \right) \right) -|v-y| \left( \log_+ \left( \frac{|v-y|}{te} \right) \right) -t\deg y -  \frac{2t}{\log t} \xi(y)\\
&\ \ \ \ \ - \sum_{u \in [[O,y]]} \log (\deg u) - \sum_{u \in [[y,v]]} \log (\deg u) - \frac{t}{\log t} \deg v - (|y| + |v-y|)\log \log t - O(t). \\
&\geq \lambda (t,v) -\frac{2t}{\log t} \xi(Y(t,v))- \sum_{u \preceq Y(t,v)} \log (\deg u) - \sum_{u \in [[Y(t,v),v]]} \log (\deg u) - \frac{ta(t)}{\log t} - O(r(t)(\log t)^{\delta} \log \log t) \\
&\geq \lambda (t,v) -a(t) (\log t)^{\frac{2\delta}{\alpha} - 1} - 4\tilde{B} r(t) (\log t)^{\delta} - 4(q+1)B \log t  - \frac{ta(t)}{\log t} - O(r(t)(\log t)^{\delta} \log \log t).
\end{align*}
Here the penultimate line follows by the assumption that $\deg v \leq a(t)$, and the last line follows by Lemmas \ref{lem:Br log sum bound} and \ref{lem:sup in ball}$(iii)$. Since $\delta<1$, the result then follows since
\[
\left( ta(t)\right)^{-1} \left( a(t) (\log t)^{\frac{2\delta}{\alpha} - 1} + 4\tilde{B} r(t) (\log t)^{\delta} + 4(q+1)B \log t + \frac{ta(t)}{\log t} + O(r(t)(\log t)^{\delta} \log \log t) \right) \to 0
\]
as $t \to \infty$ (note that $ta(t) = r(t) \log t$).
\item \textbf{Case 2: $\mathbf{\deg v > a(t).}$} In this second case we lower bound $u(t,v)$ using a slightly different event. Again we can assume wlog that $v \in B_{r(t)(\log t)^{\delta}}$. We again proceed as in \eqref{eqn:utv LB} and \eqref{eqn:utv LB hitting prob} but this time we observe that, letting $H_v^+ = \inf\{s \geq H_{\ell (v)}: X_s = v\}$, it follows from \eqref{eqn:RW hitting probs upper and lower} and Lemma \ref{lem:high deg next to leaf} that with high $\bPb$-probability
\begin{align*}
&\prstart{X_{\frac{t}{\log t}} = v}{y} \\
&\geq \prstart{H_{\ell(v)} \leq \frac{t}{2\log t}, H_v^+ \in \left(\frac{t}{\log t} - \frac{2}{\deg v}, \frac{t}{\log t} - \frac{1}{\deg v} \right), X_{s} = v \forall s \in \left[H_v^+, \frac{t}{\log t}\right]}{y} \\
&\geq \prstart{H_{\ell(v)} \leq \frac{t}{2\log t}}{y}  \left\{ \inf_{T \in \left(0, \frac{t}{2\log t}\right)} \pr{\textsf{exp}(1) \in \left(\frac{t}{\log t} - \frac{2}{\deg v} - T, \frac{t}{\log t} - \frac{1}{\deg v} - T \right) } \right\} \pr{\textsf{exp}(\deg v) \geq \frac{2}{\deg v}} \\
%&\geq \exp \left\{-\left[\sum_{u \in [[y,v]]} \log (\deg u) \right] -(|v-y|+1) \left( \log_+ \left( \frac{|v-y|+1}{2te}  + \log \log t \right) \right) + O(t) \right\} \\
%&\quad \times \left\{ \inf_{T \in \left(0, \frac{t}{\log t}\right)} e^{-\left(\frac{2t}{\log t} - \frac{2}{\deg v} - T\right)}\left(1 - e^{- \frac{1}{\deg v}}\right) \right\} e^{-2} \\
&\geq \exp \left\{-\left[\sum_{u \in [[y,v]]} \log (\deg u) \right] -(|v-y|+1) \left( \log_+ \left( \frac{2(|v-y|+1)}{te}  + \log \log t \right) \right) + O(t) \right\}  e^{-\left(\frac{2t}{\log t} \right)}\frac{1}{2\deg v} e^{-2}.
\end{align*}
Substituting back into the second line of \eqref{eqn:utv LB hitting prob} and then the second line of \eqref{eqn:utv LB}, and then conditioning on the event in Lemma \ref{lem:sup in ball}$(i)$, we see that
\begin{align*}
\log u(t,v) &\geq \sup_{y \in B_{r(t)(\log t)^{\delta}}} \left( t - \frac{2t}{\log t} \right) \xi(y) - \left\{\left[\sum_{u \in [[O,y]]} \log (\deg u) \right] + |y| \left( \log_+ \left( \frac{|y|}{te} \right) + \log \log t \right) \right\} -t\deg y \\
&\qquad - \left\{\left[\sum_{u \in [[y,v]]} \log (\deg u) \right] +(|v-y|+1) \left( \log_+ \left( \frac{2(|v-y|+1)}{te}  + \log \log t \right) \right) + \right\} - \frac{2t}{\log t} - O(t) \\
&= \lambda(t,v) - \frac{2t}{\log t} \xi(Y(t,v)) - \sum_{u \in [[O,Y(t,v)]]} \log (\deg u)  - \sum_{u \in [[Y(t,v),v]]} \log (\deg u) - 2(|Y(t,v)| + |v|) \log \log t - O(t) \\
&\geq \lambda(t,v) - O(a(t) (\log t)^{\frac{2\delta}{\alpha} - 1}) - O(r(t) (\log t)^{\delta} \log \log t),
\end{align*}
which implies the result as in a), since $\delta < 1$. \qedhere
\end{enumerate}
\end{proof}
We therefore just need to prove Lemmas \ref{lem:sup in ball}, \ref{lem:Y(t,v) deg bound} and \ref{lem:high deg next to leaf}.

\begin{proof}[Proof of Lemma \ref{lem:sup in ball}]
\begin{enumerate}[(i)]
    \item Just as in \eqref{eqn:annulus max xi lambda}, it follows from Lemma \ref{cor:whppotbo} that for any $i \geq 1, \epsilon > 0$,
\begin{align}\label{eqn:annulus max xi}
\begin{split}
\p{ \sup_{v \in B_{2^{i+1}r(t)(\log t)^{\delta}}} \xi (v) \geq \frac{1 \wedge q}{8} 2^{i}\frac{r(t)}{t}(\log t)^{1+\delta}} %&= \p{ \sup_{v \in B_{2^{i+1}r(t)(\log t)^{\delta}}} \xi (v) \geq 2^{\frac{(i+1)d}{\alpha}}r(t)^{\frac{d}{\alpha}}(\log t)^{\frac{d \delta}{\alpha}} \frac{1 \wedge q}{8} 2^{i(1-\frac{d}{\alpha}) - \frac{d}{\alpha}}(\log t)^{\delta(1-\frac{d}{\alpha})}} \\
&\leq c_q \left(2^{i(1-\frac{d}{\alpha})}(\log t)^{\delta(1-\frac{d}{\alpha})}\right)^{\frac{-(\alpha - \epsilon)}{d}}.
\end{split}
\end{align}
Therefore, given $v \in B_{r(t)(\log t)^{\delta}}$ and choosing $\epsilon > 0$ small enough that $\frac{\alpha - \epsilon}{d}$ is positive, and since $|y| \vee |v-y| \geq \frac{1}{2} |v|$ for all $v,y \in \Ti$, we have that 
\begin{align*}
    &\p{\exists i \geq 0, y \in B_{2^{i+1}r(t)(\log t)^{\delta}} \setminus B_{2^{i}r(t)(\log t)^{\delta}}: \lambda(t,v,y) + t \deg y \geq 0} \\
    &\qquad \leq \sum_{i \geq 0} \p{ \sup_{v \in B_{2^{i+1}r(t)(\log t)^{\delta}}} \xi (v) \geq \frac{1}{8} 2^{i}\frac{r(t)}{t}(\log t)^{1+\delta}} \\
    &\qquad \leq \sum_{i \geq 0} c 2^{\frac{-i(1-\frac{d}{\alpha})(\alpha - \epsilon)}{d}} (\log t)^{\frac{-\delta(1-\frac{d}{\alpha})(\alpha - \epsilon)}{d}},
\end{align*}
which vanishes as $t \to \infty$.
\item Follows from $(i)$, since $\lambda(t,v,y) \leq \lambda(t,y)$ for all $v \in \Ti$.
\item Follows from $(i)$ and Lemma \ref{cor:whppotbo}.
\item Choose some $\epsilon \in (0, m(\alpha - d) - d)$ (by our choice of $m$, such an $\epsilon$ always exists). Note that by Corollary \ref{cor:exitball}, Lemma \ref{lem:RW ball exit time restricted radii} and \eqref{eqn:annulus max xi}, we have whp as $t \to \infty$ that
\begin{align*}
\sum_{v \notin B_{r(t) (\log t)^{\delta}}} u(t,v) &\leq \estart{\exp\left\{\int_0^t\xi(X_s)~\text{d} s\right\}\mathbbm 1\{\tau_{B_{r(t) (\log t)^{\delta}}}\leq t\}}{O} \\
&\leq \sum_{r\geq r(t)(\log t)^{\delta}} \pr{\tau_{B_r} \leq t} \exp \{t \sup_{v \in B_r} \xi(v)\} \\
&\leq \sum_{r(t) (\log t)^{\delta} \leq r \leq r(t) (\log t)^{m}} \exp\left\{-r \log \left( \frac{r}{et} \right) + o \left( r(t)\right) \right\} \exp \left\{ \frac{q}{4} r (\log t) \right\} \\
&\qquad + \sum_{r\geq r(t)(\log t)^{m}} \exp \left\{-\frac{r}{5} \left( [\log r - \log t - 1] \vee 0 \right))\right\}\exp \left\{t r^{\frac{d}{\alpha}}(\log r)^{\frac{d + \epsilon}{\alpha}}\right\},
%&\leq \sum_{r(t) (\log t)^{\delta} \leq r \leq r(t) (\log t)^{m}} \exp\left\{-r\frac{q}{2} \left(\log t \right) + o \left( r(t)\right) \right\} \exp \left\{ \frac{q}{4} r (\log t) \right\} \\
%&\qquad + \sum_{r\geq r(t)(\log t)^{m}} \exp \{-\frac{r}{5} \left( [\log r - \log t - 1] \vee 0 \right))\}\exp \left\{t r^{\frac{d}{\alpha}}(\log r)^{\frac{d + \epsilon}{\alpha}}\right\} \\
%&\leq 1.
\end{align*}
The choice of $\epsilon$ implies that $m > \frac{d}{\alpha} + \frac{d+\epsilon}{\alpha}$, which is exactly what is required for this sum to vanish as $t \to \infty$.
\item The calculation in $(iv)$ similarly gives $(v)$ as well.
\end{enumerate}
\end{proof}

\begin{proof}[Proof of Lemma \ref{lem:Y(t,v) deg bound}]
Recall that
\begin{align*}
\lambda(t,v,y) = \left\{ t(\xi(y) - \deg(y))- |y| [\log |y| - \log t - 1] -  |v-y| [\log |v-y| - \log t - 1] \right\},
\end{align*}
and $\lambda(t,v) = \sup_{y \in \Ti} \lambda(t,v,y)$.
We also define the set
\begin{align*}
C_{t, \epsilon} &= \left\{y \in B_{r(t)(\log t)^{\delta}}: t \xi (y) > \epsilon r(t) \log \left(\frac{r(t)}{t}\right) \right\}.
%D_{t} &= \left\{y \in B_{r(t)}: t \xi (y) > |y| \log \left( \frac{|y|}{t}\right) \right\}.
\end{align*}
We will show that for any $\delta' > 0$, we can choose $\epsilon > 0$ small enough so that for all sufficiently large $t$ we have with probability at least $1-\delta'$ that
\begin{enumerate}
\item $C_{t, \epsilon} \neq \emptyset$;
\item $\left\{ Y(t,v): v \in B_{r(t)}\right\} \subset C_{t, \epsilon}$ and $\left\{ \tilde{Y}(t,v): v \in B_{r(t)}\right\} \subset C_{t, \epsilon}$;
\item For any $\epsilon' > 0$, the cardinality of the set $C_{t, \epsilon}$ is upper bounded by $(\log t)^{d \delta + \epsilon'}$;
\item For any $\epsilon' > 0$, $\sup_{v \in C_{t, \epsilon}} \deg v \leq (\log t)^{\frac{d\delta}{\beta} + \epsilon'}$.
\end{enumerate}
We deal with each of these four points in four separate steps below.

\textbf{Step 1.} To deal with the first point, choose some $k< \alpha$ and note from Proposition \ref{prop:AR vol growth and exit time}$(iii)$ that 
\begin{align*}
\p{C_{t, \epsilon} = \emptyset} &\leq \prb{V(B_{r(t)}) \leq \epsilon^{k} r(t)^d} + \prcondc{C_{t, \epsilon} = \emptyset}{V(B_{r(t)}) > \epsilon^{k} r(t)^{\frac{\beta}{\beta - 1}}}{} \\
&\leq e^{-\epsilon^{\frac{-(\beta - 1)k}{\beta}}} + \left( 1 - t^{\alpha} r(t)^{-\alpha} \epsilon^{-\alpha} \left(\log \left(\frac{r(t)}{t}\right)\right)^{-\alpha} \right)^{\epsilon^{k} r(t)^{d}}
%&\leq e^{-\epsilon^{\frac{-(\beta - 1)k}{\beta}}} + \exp \{- t^{\alpha} r(t)^{-\alpha} \epsilon^{-\alpha} (\log \left(\frac{r(t)}{t}\right))^{-\alpha}\epsilon^{k} r(t)^{d} \} 
\leq e^{-\epsilon^{\frac{-(\beta - 1)k}{\beta}}} + \exp \{- \epsilon^{k-\alpha} q^{-\alpha} \},
\end{align*}
which clearly vanishes as $\epsilon \downarrow 0$.

%------------------------------------------

\textbf{Step 2.} To deal with the second point: firstly, note that if $Y(t,v) \notin B_{\epsilon r(t)}$ or $v \notin B_{\epsilon r(t)}$, then it must be that $Y(t,v) \in C_{t, \epsilon}$, since $\lambda(t,v, Y(t,v)) > 0$ on the event $\lambda (t,v) > 0$ (and similarly for $\tilde{Y}(t,v)$).

We will now show that, with high probability as $t \to \infty$, $\lambda(t,v) \geq \epsilon r(t) \log \left(\frac{r(t)}{t}\right)$ for all $v \in B_{\epsilon r(t)}$, which implies the result since $t\xi (Y(t,v)) \geq \lambda(t,v)$. To do this, we first let $C_{k,t} = B_{(k+1) \epsilon r(t)} \setminus A_{k \epsilon r(t)}$ for $1 \leq k \leq \lfloor \epsilon^{-1} \rfloor$, where $A_r$ is the connected component containing the root obtained by deleting $s_r$ from $\Ti$ (recall that $s_r$ is the $r^{th}$ backbone vertex). Note that the definition of the set $A_r$ ensures that $(C_{k,t})_{k=1}^{\lfloor \epsilon^{-1} \rfloor}$ is a sequence of independent sets, each with law $B_{\epsilon r(t)}$. Moreover, by Lemma \ref{lem:deg4set}, for each $k \geq 4$ we have for all sufficiently large $t$ that
\begin{align*}
    &\p{\exists v \in C_{k,t}: t(\xi(v) - \deg v) \geq ((2k+4)\epsilon) r(t) \log \left(\frac{r(t)}{t}\right)} \\
    &\geq \p{\exists v \in C_{k,t}: \deg v \leq 4, \xi(v) \geq ((2k+5)\epsilon) \frac{r(t)}{t} \log \left(\frac{r(t)}{t}\right)} \\
    &\geq \prb{|\{v \in B_{\epsilon r(t)}: \deg v \geq 4 \}| \geq \epsilon^d r(t)^d} \left[ 1 - \left(1-((2k+5)\epsilon)^{-\alpha} r(t)^{-\alpha} t^{\alpha} \left(\log \left(\frac{r(t)}{t}\right)\right)^{-\alpha}\right)^{\epsilon^d r(t)^d} \right]\\
    &\geq C\left( 1 - \exp \left\{-\frac{q}{2}((2k+5)\epsilon)^{-\alpha} {\epsilon^d}\right\}\right).
\end{align*}
In particular, for $1 \leq k \leq \epsilon^{\frac{-(\alpha - d)}{\alpha}}$ this is uniformly bounded below by a positive constant, so that
\[
\p{\exists k \leq \lfloor \epsilon^{\frac{-(\alpha - d)}{\alpha}}\rfloor, v \in C_{k,t}: t(\xi(v) - \deg v) \geq ((2k+4)\epsilon) r(t) \log \left( \frac{r(t)}{t}\right)} \geq 1 - Ce^{-c{\epsilon^{\frac{-(\alpha - d)}{\alpha}}}}.
\]
Moreover, if $y \in C_{k,t}$ with $t(\xi(y) - \deg y) \geq ((2k+4)\epsilon) r(t) \log \left(\frac{r(t)}{t}\right)$ and $v \in B_{\epsilon r(t)}$ we have that $|y| \leq (k+1) \epsilon r(t)$ and $|v-y| \leq (k+2) \epsilon r(t)$ so that $\lambda(t,v,y) \geq \epsilon r(t) \log \left(\frac{r(t)}{t}\right)$, which proves the claim.

\textbf{Step 3.} The third step is therefore to bound the cardinality of this set. Setting $N_t = \left(\log \left(\frac{r(t)}{t} \right)\right)^{c_1} r(t)^d$ and $p_t = t^{-\alpha} r(t)^{-\alpha} \epsilon^{-\alpha} (\log r(t))^{-\alpha}$, we have from Proposition \ref{prop:AR vol growth and exit time}$(i)$ and a Chernoff bound that, if $c_2 > d \delta$ and $c_1 \in (d \delta, c_2)$, then
\begin{align*}
\p{|C_{t, \epsilon}| \geq (\log t)^{c_2}} &\leq \prb{V(B_{r(t) (\log t)^{\delta}}) \geq (\log r(t))^{c_1} r(t)^d} + P \left( \textsf{Binomial}(N_t, p_t) \geq (\log t)^{c_2} \right) \\
&\leq (\log t)^{-\frac{1}{2}(\beta - 1)(c_1 - d \delta)} + \exp \{ (e-1)N_t p_t - (\log t)^{c_2}\} \\
&=(\log t)^{-\frac{1}{2}(\beta - 1)(c_1 - d \delta)} + \exp \left\{ - (\log t)^{c_2} \left( 1 - 2 (\log t)^{-(c_2 - c_1)} (e-1) \epsilon^{-\alpha} \right) \right\},
\end{align*}
which clearly vanishes as $t \to \infty$.

\textbf{Step 4.} For the fourth point above, note that since the potential at each vertex in $\Ti$ is independent of the underlying tree structure, each vertex $v$ in $C_{t, \epsilon}$ is a uniform point of $B_{r(t)(\log t)^{\delta}}$, so whp it satisfies
\[
P\left(\deg v \geq x\right) \leq c\frac{r(t)(\log t)^{\delta}}{r(t)^d}x^{-(\beta - 1)} + c x^{-\beta}
\]
for all $x \geq 0$. Therefore, by a union bound, and conditional on $|C_{t, \epsilon}| \leq (\log t)^{d\delta + \epsilon'}$, we have whp that (for any $M>0$):
\[
\prc{\sup_{v \in C_{t, \epsilon}} \deg v \geq (\log t)^{M}} \leq (\log t)^{d\delta + \epsilon' - M \beta}.
\]

\textbf{Conclusion.} Therefore, to tie up, if we choose $M > \frac{d\delta}{\beta}$ this ensures that we can choose $\epsilon'$ small enough that
\[
\pr{\sup_{v \in B_{r(t)(\log t)^{\delta}}} \deg Y(t,v) \geq (\log t)^M} \geq 1 - 2 \delta'
\]
for all sufficiently large $t$. The same holds for $\tilde{Y}(t,v)$. Since $\delta'$ was arbitrary, this proves the result.
\end{proof}

\begin{proof}[Proof of Lemma \ref{lem:high deg next to leaf}]
If $\deg v > a(t)$, then $\prb{\nexists u \sim v \text{ with } u\in \mathcal{L}} \leq (1-p_0)^{a(t)-1} \leq e^{-ca(t)}$. Therefore, taking a union bound over all vertices in $B_{r(t)(\log t)^{\delta}}$, we obtain from Proposition \ref{prop:AR vol growth and exit time}$(i)$ that
\begin{align*}
\prb{\exists v \in B_{r(t)(\log t)^{\delta}}: \deg v > a(t), \nexists u \sim v \text{ with } p(u)=v, u\in \mathcal{L}} \leq \prb{V(B_{r(t)(\log t)^{\delta}}) \geq r(t)^{2d}} + r(t)^{2d} e^{-c a(t)} \to 0
\end{align*}
as $t \to \infty$.
\end{proof}

%%%%%%%%%%%%%%%%%%%%%%%%%%%%%%%%%%%%%%%%%%%%%%%%
%%%% Section: The concentration site 
%%%%%%%%%%%%%%%%%%%%%%%%%%%%%%%%%%%%%%%%%%%%%%%%
\section{The concentration sites: bounds on $|\hat{Z}_t^{(i)}|, \psi_t(\hat{Z}_t^{(i)})$ and $ \psi_t(\hat{Z}_t^{(1)}) -  \psi_t(\hat{Z}_t^{(2)})$}
\label{sec:concentration site}

In this section we derive some asymptotics for the functional $\psi_t$ and the maximisers $\hat{Z}_t^{(1)}$ and $\hat{Z}_t^{(2)}$ introduced in \eqref{eqn:psit def} and \eqref{eqn:Zt def intro}.

\subsection{Maximiser of the functional $\psi_t$}
\label{subsec:maxfunc}
Recall from \eqref{eqn:psit def} and \eqref{eqn:Zt def intro} the definition of the random functional $\psi_t:V(\Ti)\to \mathbb R$ by
\begin{align*}
         \psi_t(z) :=%&\sup_{\rho \in (0,1)} \left\{(1-\rho)(\xi(z)-\deg(z))-\frac{|z|}{t} \log \left(\frac{|z|}{\rho te}\right) \right\}\\ =& 
         \left[ (\xi(z)-\deg(z))-\frac{|z|}{t} \log \left(\xi(z) - \deg (z)\right)\right]\mathbbm 1\{t(\xi(z)-\deg(z)) \geq |z|\},
     \end{align*}
and its maximisers
 \begin{equation}\label{eqn:Zt def}
     \psi_t(\hat Z_t^{(1)})=\max_{z\in V(\Ti)}\psi_t(z), \ \ \ \psi_t(\hat Z_t^{(2)})=\max_{z\in V(\Ti)\setminus \{\hat Z^{(1)}_t\}}\psi_t(z), \ \ \ \psi_t(\hat Z_t^{(3)})=\max_{z\in V(\Ti)\setminus \{\hat Z^{(1)}_t, \hat Z_t^{(2)}\}}\psi_t(z).
 \end{equation}
(We will show in Proposition \ref{prop:Zt bounds as} that these are almost surely uniquely defined). In this section we establish some elementary properties of these first maximisers. We first prove the following.

\begin{prop}\label{prop:Zt bounds as}
$P$-almost surely under Assumption \ref{assn:whp} or Assumption \ref{assn:as}, we have for any $\epsilon > 0$  that for all sufficiently large $t$, and $i=1,2,3$,
\begin{enumerate}[(i)]
    \item $ r(t) (\log t)^{-(1/2+\epsilon)} \leq |\hat{Z}_t^{(i)}|$,
    \item $a(t)(\log t)^{-\epsilon}\leq \psi_t(\hat{Z}_t^{(i)})$,
    \item $|\hat{Z}^{(i)}_t|\leq r(t) (\log t)^{\frac{1}{\alpha - d}+\epsilon}$,
    \item $\psi_t(\hat{Z}_t^{(i)}) \leq  a(t) (\log t)^{\frac{1}{\alpha - d}+\epsilon}$.
\end{enumerate}
Under Assumption \ref{assn:whp} we instead have that for any $\epsilon > 0$, whp:
\begin{enumerate}[(i$'$)]
    \item $ r(t) (\log t)^{-\epsilon} \leq |\hat{Z}_t^{(i)}|$,
    \item $a(t)(\log t)^{-\epsilon}\leq \psi_t(\hat{Z}_t^{(i)})$,
    \item $|\hat{Z}^{(i)}_t|\leq r(t) (\log t)^{\epsilon}$,
    \item $\psi_t(\hat{Z}_t^{(i)}) \leq  a(t) (\log t)^{\epsilon}$.
\end{enumerate}
\end{prop}
\begin{proof} Note that $d=2$ under Assumption \ref{assn:as}, but we carry $d$ through the proofs in order to make the high probability extensions clearer.
\begin{enumerate}[(i)]
    \item Let $\epsilon >0$ and let also $\epsilon' \in (0, \frac{\epsilon}{2})$. Let $f(t)=(\log t)^{-\frac{1+\epsilon}{2}}$, $g(t)=(\log t)^{-\epsilon/16}$. We perform two calculations.
\begin{enumerate}
    \item[(a)] We have from Lemma \ref{lem:xi in Ar} that eventually almost surely 
\begin{align*}
\sup_{v\in B_{r(t)f(t)}}\psi_t(v)\leq \sup_{v\in B_{r(t)f(t)}} \xi(v) &\leq (r(t)f(t))^{d/\alpha} (\log(r(t)f(t)))^{\frac{(2+\epsilon)}{2\alpha}} \leq C_q a(t)(\log t)^{-\epsilon/2\alpha}.
\end{align*}

\item[(b)] We have from Lemma \ref{cor:sup xi LB} that eventually almost surely, for any $\delta > 0$ there exist distinct $u_1,u_2,u_3$ in $B_{r(t)g(t)}$ such that
\begin{align*}
    \xi (u_i) - \deg (u_i) &\geq (r(t)g(t))^{d/\alpha}(\log(r(t)g(t)))^{-\epsilon \delta /8\alpha} \geq c_q a(t) (\log t)^{-\epsilon (1+\delta)/8\alpha},
\end{align*}
Consequently, provided $t$ is sufficiently large, and again using Lemma \ref{lem:xi in Ar} for an upper bound on $\xi(u_i)$, and that $\alpha > 2$, we can choose $\delta >0$ small enough that
\begin{align*}
    \psi_t(u_i)= (\xi (u_i) - \deg (u_i))-\frac{|u_i|}{t}\log(\xi (u_i) - \deg (u_i)) &\geq c_q a(t)(\log t)^{-\epsilon(1+\delta) /8\alpha} - c_q' a(t)(\log t)^{-\epsilon /16} \\
    &> \frac{c_q}{2} a(t)(\log t)^{-\epsilon (1+\delta) /8\alpha}.
\end{align*}
\end{enumerate}
Combining $(a)$ and $(b)$ and assuming that $\delta < 3$ we deduce that $\psi_t(u_i)>\sup_{v\in B_{r(t)f(t)}}\psi_t(v)$ each $i=1,2,3$. Since $u_1, u_2, u_3$ are distinct, it follows that $\psi_t ( \hat Z_t^{(3)})$ enjoys the same lower bound, and $|\hat Z_t^{(i)}|>r(t)f(t)$ eventually almost surely.
\item For the lower bound on $\psi_t$, note that we can conclude from the proof of part $(i)$ that for $i=1,2,3$
\begin{align*}
    \psi_t(\hat Z_t^{(i)})\geq \min_{i=1,2,3}\psi_t(u_i)> a(t)(\log t)^{-\epsilon d/\alpha}.
\end{align*}
\item Take $b=\frac{1+2\epsilon}{\alpha - d}$. If $x > r(t) (\log t)^b$, then
\[
x^{\frac{d}{\alpha}}(\log x)^{\frac{1+\epsilon}{\alpha}} \leq \frac{x}{t} \log t
\]
for all sufficiently large $t$. Therefore, if $v \in V(\Ti)$ is such that $|v| \geq  r(t) (\log t)^b$, then either $\xi (v) - \deg (v) \leq t$, in which case $v \neq \hat{Z}_t^{(i)}$ by part $(ii)$ (since $\xi(\hat{Z}_t^{(i)}) \geq \psi_t(\hat{Z}_t^{(i)})$), or otherwise we have from Lemma \ref{lem:xi in Ar} that
\[
\psi_t(v) \leq |v|^{\frac{d}{\alpha}}(\log |v|)^{\frac{1+\epsilon}{\alpha}} - \frac{|v|}{t}\log (\xi(v) - \deg (v)) \leq |v|^{\frac{d}{\alpha}}(\log |v|)^{\frac{1+\epsilon}{\alpha}} - \frac{|v|}{t}\log t \leq 0
\]
by the computation above, so again $v \neq \hat{Z}_t^{(i)}$ by part $(ii)$.
\item From $(iii)$ and Lemma \ref{lem:xi in Ar}, almost surely for all sufficiently large $t$ we have that
\[
\psi_t(\hat{Z}_t^{(i)}) \leq t\sup_{v \in B_{ r(t) (\log t)^{\frac{1}{\alpha - d}+\epsilon}}} \xi(v) \leq C_q \left(r(t) (\log t)^{\frac{1}{\alpha - d}+\epsilon}\right)^{\frac{d}{\alpha}} \left( \log t \right)^{\frac{1+\epsilon}{\alpha}} = r(t)^{\frac{d}{\alpha}}\left( \log t \right)^{\frac{1}{\alpha - d}+2\epsilon}.
\]
\end{enumerate}
\begin{enumerate}[(i$'$)]
\item This is the same as above but instead taking $f(t) = (\log t)^{-\epsilon}, g(t) = (\log t)^{-\epsilon/2}$.
\item Similarly follows from considering the largest value of $\xi$ on $B_{r(t) f(t)}$ as above.
\item This follows from the computation in \eqref{eqn:annulus max xi}.
\item This also follows from the computation in \eqref{eqn:annulus max xi}. \qedhere
\end{enumerate}
\end{proof}

\begin{rmk}
Proposition \ref{prop:Zt bounds as} also shows that $\hat{Z}_t^{(i)}$ is well-defined for each $i =1,2,3$.
\end{rmk}

The next lemma will be useful for converting bounds on $\psi_{t_n}(\hat{Z}_{t_n}^{(i)})$ along some subsequence to almost sure bounds on $\psi_t(\hat{Z}_t^{(i)})$ for all sufficiently large $t$.

\begin{lem}\label{lem:Zestimates}
Let $\epsilon>0$ be arbitrary. Under Assumption \ref{assn:as}, we have for each $i=1,2,3$ that eventually $P$-almost surely:
\begin{enumerate}[(i)]
        \item $\psi_u(\hat Z_u^{(i)})-\psi_t(\hat Z_t^{(i)})\leq \frac{u-t}{t}a(u)(\log u)^{\frac{1}{\alpha - d}+\epsilon}$, for all $u>t$,
    \item $t\mapsto \psi_t(\hat Z_t^{(i)})$ is increasing.
\end{enumerate}
\end{lem}
\begin{proof}
This follows exactly as in \cite[Lemma 3.2(iv)]{KLMStwocities09}, using Proposition \ref{prop:Zt bounds as}$(iii)$.

\end{proof}

%%%%%%%%%%%%%%%%%%%%%%%%%%%%%%%
The third lemma will be useful for later running Borel-Cantelli arguments on the maximisers $\hat{Z}_t^{(i)}$ for each $i=1,2,3$.

\begin{lem}\label{lem:BC counting}
Let $t_n = 2^n$. Under Assumption \ref{assn:as}, we have for any $\epsilon > 0$ that $P$-almost surely for all sufficiently large $n$ and each $i=1,2$ that
$$\#\left\{\hat{Z}^{(i)}_t:t\in[t_n,t_{n+1}]\right\} \leq (\log t_n)^{\epsilon}.$$

\end{lem}
\begin{proof} Set $r_n = 2^n$, define the set 
$$N_n:=\left\{v\in B_{r_{n+1}}\setminus B_{r_n}: \xi(v)>2^{-d/\alpha}r_n^{d/\alpha}(\log r_n)^{-\epsilon}\right\}$$ 
and set $$p_{n}:=\mathcal P\left( \xi(O)>2^{-d/\alpha}r_n^{d/\alpha}(\log r_n)^{-\epsilon}\right)=2^{d}r_n^{-d} (\log  r_n)^{\alpha \epsilon}.$$ Note that, conditionally on $\Ti$, $|N_n|\sim \Binom(\#(B_{r_{n+1}}\setminus B_{r_n}),p_{n})$.
By Corollary \ref{cor:volume bounds}, almost surely for all $n$ large enough it holds that $\#(B_{r_{n+1}}\setminus B_{r_n})\leq r_n^{d}(\log r_n)^{\epsilon}$. Using a Chernoff bound we thus obtain that with probability $1-o(1)$ as $N \to \infty$, it holds for all $n \geq N$ that 
\begin{align*}P(|N_n|\geq (\log r_n)^{2(1+\alpha)\epsilon})\leq \frac{E[e^{|N_n|}\mid| \#(B_{r_{n+1}}\setminus B_{r_n})\leq r_n^{d}(\log r_n)^{\epsilon}]}{e^{(\log r_n)^{2(1+\alpha)\epsilon}}} %&\leq e^{c(\log r_n)^{(1+\alpha)\epsilon} - (\log r_n)^{2(1+\alpha)\epsilon}} 
\leq e^{- c(\log r_n)^{(2+\alpha)\epsilon}}.
\end{align*}
This probability is summable over $n$ and thus by Borel Cantelli $|N_n|\leq (\log r_n)^{2(1+\alpha)\epsilon}$ eventually almost surely.

To conclude, note that by Proposition \ref{prop:Zt bounds as}$(ii)$ that for each $i=1,2$, we have eventually almost surely that
\begin{equation*}
\left\{\hat Z_{t}^{(i)}:r\in[r_n,r_{n+1}]\right\}\subset N_n.
\end{equation*}
Since $\epsilon > 0$ was arbitrary we deduce the result.
\end{proof}

\subsection{The gap between the highest values of $\psi_t$}
To prove the localisation result, it will be important to have a bound between the largest values of $\psi_t$. We start with a tail bound.

\begin{lem} \label{lem:probgap}
Under Assumption \ref{assn:whp}, for any $k<\infty$ and $0<\epsilon < k$, we have for all $t$ large enough that
\begin{align}
    P(\psi_t(\hat Z_t^{(1)})-\psi_t(\hat Z_t^{(2)})<a(t) (\log t)^{-k}) &\leq 4\alpha (\log t)^{-(k - \epsilon)}, \label{eqn:gap12} \\
    P(\psi_t(\hat Z_t^{(2)})-\psi_t(\hat Z_t^{(3)})<a(t) (\log t)^{-k}) &\leq 4\alpha (\log t)^{-(k - \epsilon)}.\label{eqn:gap23}
\end{align}
\end{lem}
\begin{proof}
We just prove \eqref{eqn:gap12} (the proof of \eqref{eqn:gap23} is identical). Fix $k, \epsilon$ as above and note that it follows from the proofs of Proposition \ref{prop:Zt bounds as}$(i),(ii)$ and Lemma \ref{cor:sup xi LB} that $$\p{\psi(\hat{Z}^{(2)}_t) \leq a(t) (\log t)^{-\epsilon}} + \p{|\hat{Z}_t^{(1)}| \leq et} \leq (\log t)^{-(k - \epsilon)}$$ for all sufficiently large $t$.

We can therefore write for all sufficiently large $t$ and any $x>0, z > a(t) (\log t)^{-\epsilon} + 1$ that
    \begin{align}\label{eqn:psi gap prob decomp}
    \begin{split}
    & P(\psi_t(\hat Z_t^{(1)})-\psi_t(\hat Z_t^{(2)}) > x) \\
    &\geq \inf_{z> a(t)(\log t)^{-\epsilon} + 1} \pcond{\psi_t(\hat Z_t^{(1)})>x+z}{\psi_t(\hat Z_t^{(2)})\in [z-1,z], \psi(\hat{Z}^{(2)}_t) \geq a(t) (\log t)^{-\epsilon},  \hat{Z}_t^{(1)} > et}{} \\
    &\qquad - \p{\psi(\hat{Z}^{(2)}_t) \leq a(t) (\log t)^{-\epsilon}} - \p{|\hat{Z}_t^{(1)}| \leq et}\\
    &\geq \sup_{\mathbb{S}_{\infty} \subset \mathbb{T}_{\infty} \prb{\mathbb{S}_{\infty}}=1} \inf_{T\in \mathbb{S}_{\infty}}\inf_{v\in T} \inf_{z > a(t)(\log t)^{-\epsilon} + 1} P\left(\psi_t(v)>x+z \Big| \psi_t(\hat Z_t^{(2)})\in[z-1,z], \psi_t(v)>z-1, |v| > et \right) - (\log t)^{-(k - \epsilon)}\\
    &= \left\{ \sup_{\mathbb{S}_{\infty} \subset \mathbb{T}_{\infty} \prb{\mathbb{S}_{\infty}}=1} \inf_{T\in \mathbb{S}_{\infty}}\inf_{\substack{v\in T : \\ |v| > et}} \inf_{z> a(t)(\log t)^{-\epsilon} + 1}\frac{ \mathcal P\left(\psi_t(v)>x+z  \right)}{ \mathcal P\left(\psi_t(v)>z-1\right)} \right\} - (\log t)^{-(k - \epsilon)}.
    \end{split}
    \end{align}
    (Here the condition that $|v| > et$ is just for technical reasons later). Our aim is therefore to bound the quantity above. For $v \in V(T), t > 0$ we define the function $$f_{v,t}(y) = y - \deg v - \frac{|v|}{t}\log \left((y-\deg v)\right),$$ so that $\psi_t(v) = f_{v,t}(\xi(v)) \mathbbm{1}\{t (\xi (v) - \deg (v)) > |v|\}$. Note that for $y\in I_{v,t}:= (|v|/t+\deg(v),\infty)$,
$$\frac{d}{dy}f_{v,t}(y)=1-\frac{|v|}{t}\frac{1}{y-\deg(v)}>0;$$
 that is, $f_{v,t}$ has a continuous, increasing inverse on the interval $I_{v,t}$, which we denote by $f_{v,t}^{-1}$. Also note that for $y\in I_{v,t}$ we have that $f_{v,t}(y)\in J_{v,t}:=(|v|/t-|v|/t\log(|v|/t),\infty)$, and vice versa. We give two preliminary claims.
 
\begin{itemize}
 \item 
\textit{Claim 1}: Let $w>\deg(v)+e$ and $u>0$. If $\frac{f_{v,t}(w+u)}{f_{v,t}(w)}<\frac{1}{1-2(\log t)^{-(k - \epsilon)}}$ then
\begin{equation}\label{eqn:gap ratio bound}
\frac{w-\deg(v)}{w+u-\deg(v)}>1-2(\log t)^{-(k - \epsilon)}.
\end{equation}

\textit{Proof of Claim 1}: We prove the contrapositive. Note that for $y>e$ the function $g(y)=\frac{\log(y)}{y}$ is decreasing. % Thus $$\log((u+w-\deg(v)))\leq \frac{u+w-\deg(v)}{w-\deg(v)}\log((w-\deg(v))),$$
Therefore, if \eqref{eqn:gap ratio bound} does not hold then
\begin{align*}
    f_{v,t}(w+u)=u+w-\deg(v)-\frac{|v|}{t}\log((u+w-\deg(v))) &\geq \frac{u+w-\deg(v)}{w-\deg(v)}\left[w-\deg(v)-\frac{|v|}{t}\log((w-\deg(v)))\right]\\
    %&= \frac{u+w-\deg(v)}{w-\deg(v)} f(w)\\
    &>\frac{1}{1-2(\log t)^{-(k - \epsilon)}}f_{v,t}(w),
\end{align*}
which proves the claim. 
\item
\textit{Claim 2:} Set $x_t =a(t)(\log t)^{-k}$. For all (deterministically) sufficiently large $t$ and any $z \geq  a(t)(\log t)^{-\epsilon} + 1$,  
$$\frac{x_t+z}{z-1}<\frac{1}{1-2(\log t)^{-(k - \epsilon)}}.$$
\textit{Proof of Claim 2:} 
Note that for fixed $x$ the function $h(z)=\frac{x+z}{z-1}$ is decreasing on $z>1$. We can thus compute, provided $t$ is sufficiently large, that
\begin{align*}
    \frac{x_t+z}{z-1}&=1+\frac{x_t+1}{z-1} \leq 1+\frac{a(t)(\log t)^{-k}+1}{a(t)(\log t)^{-\epsilon}}
   % &< 1+2 (\log t)^{g-k}\\
    < \frac{1}{1-2(\log t)^{-(k - \epsilon)}},
\end{align*}
which proves the claim. 
\end{itemize}

Recall from the last line of \eqref{eqn:psi gap prob decomp} that we are assuming the following:
\begin{itemize}
\item $z> a(t)(\log t)^{-\epsilon} + 1$,
\item $|v| > et$.
\end{itemize}
Together, these imply that $z-1 \in J_{v,t}$. Now set $w = f_{v,t}^{-1}(z-1)$ and $u=f_{v,t}^{-1}(z+x_t)-w$, so that $\frac{x_t+z}{z-1} = \frac{f_{v,t}(u+w)}{f_{v,t}(w)}$. (Recall that $x_t =a(t)(\log t)^{-k}$). Since $z-1 \in J_{v,t}$ and $|v| > et$, it follows that $w \in I_{v,t}$ and in particular that $w>\deg(v)+e$. Moreover, since $u$ is increasing on $I_{v,t}$ it follows that $u>0$.

First note that, by Claim 2 and the assumption that $z> a(t)(\log t)^{-\epsilon} + 1$,
\[
\frac{f_{v,t}(u+w)}{f_{v,t}(w)} = \frac{x_t+z}{z-1} \geq \frac{x_t+z}{z-1}<\frac{1}{1-2(\log t)^{-(k - \epsilon)}}.
\]

It therefore follows from Claim 1 that 
\begin{equation}\label{eqn:f inverse ratio}
  \frac{f_{v,t}^{-1}(z-1)}{f_{v,t}^{-1}(x+z)} = \frac{w}{w+u} > \frac{w - \deg v}{w+u - \deg v}>1-2(\log t)^{-(k - \epsilon)}.
\end{equation}

Also note that the final expression in \eqref{eqn:psi gap prob decomp} can be written as
\begin{align*}
\frac{ \mathcal P\left(\psi_t(v)>x_t+z  \right)}{ \mathcal P\left(\psi_t(v)>z-1\right)} = \frac{ \prc{f_{t,v}(\xi(v))>x_t+z \text{ and } t (\xi (v) - \deg (v)) > |v|}}{\prc{f_{t,v}(\xi(v))>z-1 \text{ and } t (\xi (v) - \deg (v)) > |v|}} = \frac{ \prc{f_{t,v}(\xi(v))>x_t+z}}{\prc{f_{t,v}(\xi(v))>z-1 }}.
\end{align*}
Here, the final equality holds because of our restrictions on $v$ and $z$; in particular, since $z-1 \in J_{v,t}$ and $w = f_{v,t}^{-1}(z-1) > \deg v + e$, the requirement that $\xi(v) > f_{v,t}^{-1}(z-1)$ already implies that $\xi(v) > \deg v + e$, so we deduce that necessarily
\[
\xi(v) - \deg v - \frac{|v|}{t} \geq f_{t,v}(\xi(v)) > 0.
\]

From \eqref{eqn:psi gap prob decomp} and \eqref{eqn:f inverse ratio}, for any $z > a(t)(\log t)^{-\epsilon}$ we thus obtain the bound
    \begin{align*}
     P(\psi_t(\hat Z_t^{(1)})-\psi_t(\hat Z_t^{(2)}) > x_t) 
    &\geq \sup_{\mathbb{S}_{\infty} \subset \mathbb{T}_{\infty}: \prb{\mathbb{S}_{\infty}}=1} \inf_{T\in \mathbb{S}_{\infty}} \inf_{\substack{v\in T \\ |v| >et}} \inf_{z: z - 1 > a(t)(\log t)^{-\epsilon}}\left(\frac{f_{v,t}^{-1}(z-1)}{f_{v,t}^{-1}(x_t+z)}\right)^{\alpha} - (\log t)^{-(k - \epsilon)}\\
    &>(1-2(\log t)^{-(k - \epsilon)} )^\alpha - (\log t)^{-(k - \epsilon)}.
\end{align*}
Recalling that $1-e^{-x}\leq 3x/2$ for $x\leq 2/3$ we can conclude that 
\begin{align*}
\p{\psi_t(\hat Z_t^{(1)})-\psi_t(\hat Z_t^{(2)})<a(t) (\log t)^{-k}}
&\leq 1-(1-2(\log t)^{-(k - \epsilon)} )^\alpha + (\log t)^{-(k - \epsilon)}\\
%&\leq 1-\exp(-4\alpha (\log t)^{-(k - \epsilon)}) \\
&\leq 4\alpha (\log t)^{-(k - \epsilon)},
\end{align*}
as required.
\end{proof}

We can now use this to derive an almost sure lower bound on the gap of $\psi_t$ between the maximiser and the third maximiser.

\begin{prop}\label{prop:gap a.s. LB}
Take any $k> 1+ \frac{1}{\alpha - d}$. Under Assumption \ref{assn:as}, eventually $P$-almost surely:
$$\psi_t(\hat Z_t^{(1)})-\psi_t(\hat Z_t^{(3)})\geq  a(t)(\log t)^{-k}.$$
\end{prop}
\begin{proof}
It holds that
\begin{align*}
P(\psi_t(\hat Z_t^{(1)})-\psi_t(\hat Z_t^{(3)})<x) &\leq P(\psi_t(\hat Z_t^{(1)})-\psi_t(\hat Z_t^{(2)})<x,\psi_t(\hat Z_t^{(2)})-\psi_t(\hat Z_t^{(3)})<x)\\
    &=  P(\psi_t(\hat Z_t^{(1)})-\psi_t(\hat Z_t^{(2)})<x| \psi_t(\hat Z_t^{(2)})-\psi_t(\hat Z_t^{(3)})<x)  P(\psi_t(\hat Z_t^{(2)})-\psi_t(\hat Z_t^{(3)})<x).
\end{align*}
 Note from the proof of Lemma \ref{lem:probgap} that the estimate in \eqref{eqn:gap12} for $\psi_t(\hat Z_t^{(1)}) - \psi_t(\hat Z_t^{(2)})$ to be small does not depend on $\psi_t(\hat Z_t^{(3)})$. Thus we can use the two estimates \eqref{eqn:gap12} and \eqref{eqn:gap23} to bound the probability above. That is, choosing some small $\epsilon > 0$ (to be fixed precisely later), we obtain that
\begin{equation}\label{eqn:squaredprob}
    P\Big(\psi_t(\hat Z_t^{(1)})-\psi_t(\hat Z_t^{(3)})<a(t)(\log t)^{-k}\Big)\\
    \leq C (\log t)^{-2(k - \epsilon)}.
\end{equation}
Now, take $t_n=e^{n^\gamma}$ for some $\frac{1}{2(k-\epsilon)}<\gamma<\frac{1}{k+1+\frac{1}{\alpha - d}}<1$ (since $k> 1+ \frac{1}{\alpha - d}$, we can always choose $\epsilon>0$ small enough that this is possible). Then
\begin{equation}\label{eqn:squareestimate}
\sum_{n\in \mathbb N} P\Big(\psi_{t_n}(\hat Z_{t_n}^{(1)})-\psi_{t_n}(\hat Z_{t_n}^{(3)})<a(t_n)(\log t_n)^{-k}\Big)\leq C \sum_{n\in \mathbb N}n^{-\gamma  2(k-\epsilon)}<\infty,
\end{equation}
so by Borel-Cantelli along the sequence $t_n$ we have that eventually almost surely 
$$\psi_{t_n}(\hat Z_{t_n}^{(1)})-\psi_{t_n}(\hat Z_{t_n}^{(3)})\geq a(t_n)(\log t_n)^{-k}.$$
Finally, for $t\in[t_n,t_{n+1})$, since $\psi_t(\hat{Z}_t^{(i)})$ is an increasing function for each $i=1,2,3$, we compute invoking Lemma \ref{lem:Zestimates} that
\begin{align*}
    \psi_t(\hat Z_t^{(1)})-\psi_t(\hat Z_t^{(3)})
&\geq [\psi_{t_n}(\hat Z_{t_n}^{(1)})-\psi_{t_n}(\hat Z_{t_n}^{(3)})]- [\psi_{t_{n+1}}(\hat Z_{t_{n+1}}^{(3)})-\psi_{t_n}(\hat Z_{t_n}^{(3)})]\\
&\geq a(t_n)(\log t_n)^{-k}-\frac{t_{n+1}-t_n}{t_n}a(t_{n+1})(\log t_{n+1})^{\frac{1}{\alpha - d}+\frac{\epsilon}{2}}\\
&\geq a(t_n)(\log t_n)^{-k}-(\log t_n)^{-(1-\gamma)/\gamma}a(t_{n+1})(\log t_{n})^{\frac{1}{\alpha - d}+\epsilon}.\qedhere
\end{align*}
\end{proof}

\begin{rmk}
Note that in the proof above it was essential for \eqref{eqn:squareestimate} to have the $2$ in the exponent in \eqref{eqn:squaredprob} coming from multiplying the probabilities for the gaps $\psi_t(\hat Z_t^{(1)})-\psi_t(\hat Z_t^{(2)})$ and $\psi_t(\hat Z_t^{(2)})-\psi_t(\hat Z_t^{(3)})$. For this reason the argument above would not give an almost sure bound for the gap  $\psi_t(\hat Z_t^{(1)})-\psi_t(\hat Z_t^{(2)})$ (for which we know the result is not true).
 \end{rmk}
 
%%%%%%%%%%%%%%%%%%%%%%%%%%%%%%%%%%%%%%%%%%%%%%%%
%%%%% Section: Spectral results
%%%%%%%%%%%%%%%%%%%%%%%%%%%%%%%%%%%%%%%%%%%%%%%%%

\section{Spectral results}
\label{sec:spectralresults}
In this section we derive a concentration result for the principal eigenfunction of the Anderson Hamiltonian $H=\Delta+\xi$ restricted to a suitable subset of $V(\Ti)$. This will be our main tool for proving the concentration result for $u(t, \cdot)$.

In Section \ref{sctn:spectral setup} we explain how eigenfunction localisation implies a similar localisation result for $u(t, \cdot)$, and rephrase localisation of $u(t, \cdot)$ in terms of localisation of the principal eigenfunction (Lemma \ref{lem:efunction compare}). In Sections \ref{sctn:choice of Lambdat} and \ref{sctn:efunction bounds} we establish an almost sure version of this criterion under Assumption \ref{assn:as}, and in Section 6.4 we establish an analogous high probability version under Assumption \ref{assn:whp}.

\subsection{Role of the principal eigenfunction}\label{sctn:spectral setup}
Take some finite, connected set $\Lambda \subset V(\Ti)$, with $O \in \Lambda$, and consider the equation \eqref{eqn:PAM} restricted to $\Lambda$ with zero boundary condition. As before, there is a Feynman-Kac representation for the solution on $\Lambda$, namely
\begin{equation*}\label{eqn:FeynmanKac bounded}
u_{\Lambda}(t,z)=\estart{\exp\left\{\int_0^t\xi(X_s)~\text{d} s\right\} \mathbbm{1} {\{X_t=z, \tau_{\Lambda} > t\}}}{O},\ \ \ t>0,z\in \Lambda.
\end{equation*}
Let $\lambda_{\Lambda}^{(1)} \geq \ldots \geq \lambda_{\Lambda}^{(|\Lambda|)}$ and $\phi^{(1)}, \ldots, \phi_{\Lambda}^{(|\Lambda|)}$ be the respective eigenvalues and corresponding orthogonal eigenfunctions of the Anderson Hamiltonian $H_{\Lambda}$ restricted to the class of real-valued functions supported on $\Lambda$. If $y \in \Lambda$ and $\phi^{(1)}_{\Lambda}$ is normalised so that $\phi^{(1)}_{\Lambda}(y)=1$, then (e.g. see \cite[Proposition 3.3]{MuirheadPymarBAMLocalisation}; the same proof applies in our setting) we have the representation
\begin{equation}\label{eqn:eigenfunction rep}
\phi^{(1)}_{\Lambda}(x) = \estart{\exp\left\{\int_0^{H_y} (\xi(X_s) - \lambda^{(1)}_{\Lambda})~\text{d} s\right\} \mathbbm{1} {\{H_{y} < \tau_{\Lambda}\}}}{x}.
\end{equation}

We are interested in using \eqref{eqn:eigenfunction rep} when $y$ is a ``good'' site, in order to show that the contribution from other sites is negligible. In particular, suppose that $\Omega \subset \Lambda$ is a small set of good sites, and set 
\begin{equation}\label{eqn:FeynmanKac bounded Omega}
u_{\Omega, \Lambda}(t,z):=\estart{\exp\left\{\int_0^t\xi(X_s)~\text{d} s\right\} \mathbbm{1} {\{X_t=z, \tau_{\Lambda} > t, H_{\Omega} < t\}}}{O},\ \ \ t>0,z\in \Lambda.
\end{equation}
We also define the following gap
\begin{equation*}
\tilde g_{\Omega,\Lambda}=\min_{z\in \Omega}[\xi(z)-\deg(z)]-\max_{z\in \Lambda\setminus \Omega}[\xi(z)-\deg(z)].
\end{equation*}
For $A\subset \Lambda$ also recall from page \pageref{sctn:gap potential} that
$$\tilde{Z}_{A} = \arg_{z \in A} \max \{\xi(z) - \deg(z) \},  \hspace{1cm}
\tilde{g}_{A} = \xi(\tilde Z_{A}) - \deg(\tilde Z_{A}) - \max_{z \in A, z \neq \tilde Z_{A}} \{\xi(z) - \deg(z) \}.$$
We will use \eqref{eqn:eigenfunction rep} to show that the principal eigenfunction localises on $\Omega$, and transfer this result to \eqref{eqn:FeynmanKac bounded Omega} via the following lemma.

\begin{lem}(cf \cite[Lemma 2.4]{KLMStwocities09})\label{lem:u efunction comp}\label{lem:efunction compare}
If $\tilde g_{\Omega,\Lambda}>0$, then for all $v\in V(\Ti)$ and $t>0$ we have 
\begin{enumerate}[(i)]
    \item 
$u_{\Omega, \Lambda}(t,v) \leq \sum_{y \in \Omega} u_{\Omega, \Lambda}(t,y) ||\phi_{(\Lambda \setminus \Omega) \cup \{y\}}||^2_2 \phi_{(\Lambda \setminus \Omega) \cup \{y\}}(v)$,
\item $\frac{\sum_{v \in \Lambda \setminus \Omega} u_{\Omega, \Lambda}(t,v)}{\sum_{v \in \Lambda} u_{\Omega, \Lambda}(t,v)} \leq \sum_{y \in \Omega}||\phi_{(\Lambda \setminus \Omega) \cup \{y\}}||^2_2 \sum_{v \in \Lambda \setminus \Omega} \phi_{(\Lambda \setminus \Omega) \cup \{y\}}(v)$.
\end{enumerate}
\end{lem}
\begin{proof}
It follows from an eigenvalue expansion exactly as in the lower bound of \cite[Lemma 2.1]{HollanderKoenigSantosPAMOnTrees} that for any $y \in \Ti$, 
\[
\estart{\exp\left\{\int_0^u\xi(X_s)~\text{d} s\right\} \mathbbm{1} {\{X_t=y, \tau_{\Lambda} > t, H_{\Omega \setminus \{y\}} > t\}}}{y} \geq \frac{e^{u \Lambda^{(1)}_{\Lambda \setminus \Omega \cup \{y\}}}\phi_{(\Lambda \setminus \Omega) \cup \{y\}}(y)}{||\phi_{(\Lambda \setminus \Omega) \cup \{y\}}||^2_2} = \frac{e^{u \Lambda^{(1)}_{\Lambda \setminus \Omega \cup \{y\}}}}{||\phi_{(\Lambda \setminus \Omega) \cup \{y\}}||^2_2}.
\]
(note that \cite[Lemma 2.1]{HollanderKoenigSantosPAMOnTrees} is for the normalised eigenvectors, hence the extra renormalisation constant in our case). Recall that we also established the time reversal property for the Feynman-Kac formula in Proposition \ref{prop:FK time reversal}. The result therefore follows by exactly the same proof as in \cite[Lemma 2.1]{KLMStwocities09}, which gives the corresponding statement on $\Z^d$.
\end{proof}

\subsection{Choice of $\Lambda_t$}\label{sctn:choice of Lambdat}
To apply Lemma \ref{lem:u efunction comp}, for each $t>0$ we will need to make an appropriate choice of a set $\Lambda$. Following \cite[Section 5.1]{KLMStwocities09}, recall that $z=\frac{2\alpha}{\alpha - 2}$ under Assumption \ref{assn:as} and let us define for $i=1,2$ the sets
\[
\Gamma^{(i)}_t := \left\{ v \in V(\Ti): d(v, \hat{Z}_t^{(i)}) + \min \{ |v|, |\hat{Z}_t^{(i)}|\} \leq  \left( 1 + (\log t)^{-z}\right)|\hat{Z}_t^{(i)}|\} \right\},
\]
i.e. the set of all points contained within a path of length $ \left( 1 + (\log t)^{-z}\right)|\hat{Z}_t^{(i)}|$ that also contains $\hat{Z}_t^{(i)}$. We will make the following choices for $\Lambda$ and $\Omega$ at time $t$:
$$\Lambda_t:=\Gamma_t^{(1)}\cup \Gamma_t^{(2)} \text{ and } \Omega_t:=\{\hat Z_t^{(1)},\hat Z_t^{(2)}\}.$$ Note that, on the event $\{\tilde g_{\Omega_t,\Lambda_t} > 0\}$
$$\tilde g_{\Omega_t,\Lambda_t} = \min\{\tilde g_{\Lambda\setminus \{\hat Z_t^{(1)}\}},\tilde g_{\Lambda\setminus\{\hat Z_t^{(2)}\}}\}.$$ 

We first show in Lemma \ref{lem:Gammat gap} that for all $\epsilon>0$, we have that $\tilde{g}_{\Omega_t,\Lambda_t} \geq a(t)(\log t)^{-\epsilon}$ eventually almost surely. Then we show in Lemma \ref{lem:boundeigenfunction} that we can restrict to a direct path of $(X_s)_{s \geq 0}$ to bound \eqref{eqn:eigenfunction rep}. Estimating \eqref{eqn:eigenfunction rep} by the contribution by the direct path, we then show that the value of the eigenfunction summed over points not in $\{\hat{Z}_t^{(1)}, \hat{Z}_t^{(2)}\}$ is negligible in Lemma \ref{lem:loceigenfunction}. Finally we use Lemma \ref{lem:u efunction comp} to transfer this result back to the solution of the PAM.

It will first be important to bound the volume of $\Lambda_t$.

\begin{lem}\label{lem:volGamma}
Under Assumption \ref{assn:as}, for any $\delta > 0$, eventually $P$-almost surely, for each $i=1,2$:
\[
\#\Gamma_t^{(i)}\leq \frac{1}{2}(\log t)^{-(z-\delta)} |\hat{Z}_t^{(i)}|^2 \leq \frac{1}{2}(\log t)^{-(z-\delta)+\frac{2}{\alpha - 2}} r(t)^2.
\]
\end{lem}
\begin{proof}
We just bound $\#\Gamma^{(1)}_t$ (the same argument works for $\#\Gamma^{(2)}_t$). Note that $d=2$ under Assumption \ref{assn:as}. Set $R = |\hat{Z}_t^{(1)}|$, $r= \frac{1}{(\log t)^z} |\hat{Z}_t^{(1)}|$, and let $\gamma^{(1)}_{O,t}$ denote the direct path from $O$ to $\hat{Z}_t^{(1)}$. Note that, if $v \in \gamma^{(1)}_{O,t}$ and $u_1, \ldots, u_{\deg v - 1}$ are the offspring of $v$, then 
\begin{equation}\label{eqn:size biased}
\p{\deg v = k + 1} = \mu_k \sum_{i=1}^k \pcond{u_i \preceq \hat{Z}_t^{(1)}}{\deg v = k + 1}{} \leq k\mu_k.
\end{equation}
We deduce that the offspring tails $\hat{\mu}$ along the path $\gamma^{(1)}_{O,t}$ independently satisfy a size-biased tail bound, which is therefore also super-polynomial under Assumption \ref{assn:as}. Moreover, due to the Galton-Watson structure, the subtrees rooted at each of the $u_i$ are independent Galton-Watson trees, of which one is conditioned to survive to $\hat{Z}_t^{(1)}$, a second of which may be conditioned to survive forever, and the rest of which are unconditioned. Let $M=\sum_{v \in \gamma^{(1)}_{O,t}} (\deg v - 2)+1$ be the number of such subtrees attached to $\gamma^{(1)}_{O,t}$ (including at the final vertex $Z_t^{(1)}$). We deduce the following (working under Assumption \ref{assn:as} throughout).
\begin{enumerate}[(a)]
    \item For any $p>0, q< \infty$, there exists $c<\infty$ such that $\p{\sum_{v \in \gamma^{(1)}_{O,t}} \deg v \geq R \lambda^{1/2}} \leq c\lambda^{-q}$ (e.g. by \cite[Lemma A.3(iii)]{ArcherRWDecGWTrees})
    \item Conditionally on $\sum_{v \in \gamma^{(1)}_{O,t}} \deg v < R \lambda^{1/2}$, the number $N$ of subtrees attached to $\gamma^{(1)}_{O,t}$ with volume exceeding $r^{2}$ is stochastically dominated by two more than a \textsf{Binomial}$\left(R \lambda^{1/2}, r^{-1} \right)$ random variable (by Lemma \ref{lem:finite GW prob bounds}, since $\beta = 2$ under Assumption \ref{assn:as}). Therefore $\p{N > 2(\log t)^z \lambda^{1/2}} \leq e^{-c(\log t)^z \lambda^{1/2}}$ (by a Chernoff bound).
    \item Let $(T_i)_{i=1}^{M}$ be the subtrees attached to the path $\gamma^{(1)}_{O,t}$. Conditionally on $N \leq 2(\log t)^z \lambda^{1/2}$, we have by \cite[Lemma A.1]{ArcherRWDecGWTrees} and a union bound that $\p{\sum_{i=1}^M \#V(T_i) \mathbbm{1}\{\#V(T_i) \leq r^2\} \geq (\log t)^z \lambda r^2} \leq (\log t)^z \lambda^{1/2}e^{-c\lambda^{1/2}}$.
    \item Now consider $\#V(T_i)$ for some Galton-Watson tree where $\#V(T_i) \geq r^2$. If $\rho_i$ is the ``root'' of $T_i$, we then have for any $q< \infty$ that $\p{\#B_{T_i}(\rho_i, r) \geq \lambda r^2} \leq c\lambda^{-3q}$ by Proposition \ref{prop:AR vol growth and exit time}$(ii)$. (Proposition \ref{prop:AR vol growth and exit time}$(ii)$ was written for $\Ti$ but applies to a finite GW tree by the same proof).
    \item It follows from the previous point that $\sum_{i=1}^M B_{T_i}(\rho_i, r) \mathbbm{1}\{\#V(T_i) \geq r^2\} \overset{s.d.}{\preceq} r^2 \left( Y + \sum_{i=1}^N X_i\right)$ where $\p{X_i \geq \lambda} \leq c\lambda^{-3q}$ independently for each $i$, and $\p{Y \geq \lambda} \leq \lambda^{-3q}$ too ($Y$ is the sum of the volumes of the balls of radius $r$ in the two subtrees that may be conditioned to survive). Conditionally on $N \leq 2(\log t)^z \lambda^{1/2}$, we therefore have from (d) and \cite[Lemma A.3(iii)]{ArcherRWDecGWTrees} that for any $q< \infty$, there exists $c< \infty$ such that
    \begin{align*}
    \p{\sum_{i=1}^M \#B_{T_i}(\rho_i, r) \mathbbm{1}\{\#V(T_i) \geq r^2\} \geq r^2 (\log t)^z \lambda} \leq \p{Y + \sum_{i=1}^N X_i \geq N \lambda^{1/2}} &\leq c\lambda^{-q}.
    \end{align*}
\end{enumerate}
We combine the bounds in (a), (b), (c) and (e) above in the following union bound (to include the vertices in $\gamma^{(1)}_{O,t}$ we have simply replaced $2(\log t)^z \lambda r^2$ with $4(\log t)^z \lambda r^2$ below):
\begin{align*}
\p{\#\Gamma^{(1)}_t  \geq 4(\log t)^z r^2 \lambda} &\leq \p{\sum_{v \in \gamma^{(1)}_{O,t}} \deg v \geq R \lambda^{1/2}} + \pcond{N \geq 2(\log t)^z \lambda^{1/2}}{\sum_{v \in \gamma^{(1)}_{O,t}} \deg v < R \lambda^{1/2}}{} \\
&+ \pcond{\sum_{i=1}^M \#V(T_i) \mathbbm{1}\{V(T_i) \leq r^2\} \geq r^2 (\log t)^z \lambda}{N < 2(\log t)^z \lambda^{1/2}}{} \\
&+ \pcond{\sum_{i=1}^M \#B_{T_i}(\rho_i, r) \mathbbm{1}\{V(T_i) \geq r^2\} r^2 (\log t)^z \lambda}{N < 2(\log t)^z \lambda^{1/2}}{} \\
&\leq (\log t)^z \lambda^{1/2}e^{-c\lambda^{1/2}} + C\lambda^{-q}.
\end{align*}
Taking $\lambda = (\log t)^{\delta}$ and recalling that $r= \frac{1}{(\log t)^z} |\hat{Z}_t^{(1)}|$ we get for any $q>0$ that there exists $C<\infty$ such that
\[
\p{\#\Gamma^{(1)}_t  \geq 4(\log t)^{z+\delta-2z} |\hat{Z}_t^{(1)}|^2} \leq C(\log t)^{-q}.
\]
Therefore, taking $q>1$, applying a union bound and using Lemma \ref{lem:BC counting}, we deduce that
\begin{align*}
\p{\#\Gamma^{(1)}_t  \geq 4 (\log t)^{\delta-z} |\hat{Z}_t^{(1)}|^2 \text{ i.o.}} &\leq \limsup_{T \to \infty} \sum_{t \geq T: \hat{Z}_t^{(1)} \neq \hat{Z}_{t^-}^{(1)}} \p{\#\Gamma^{(1)}_t  \geq 4(\log t)^{z+\delta-2z} |\hat{Z}_t^{(1)}|^2}\\
&\leq \limsup_{N \to \infty} \sum_{n=N}^{\infty} \sum_{\substack{t \in [2^n, 2^{n+1}]: \\ \hat{Z}_t^{(1)} \neq \hat{Z}_{t^-}^{(1)}}} C(\log t)^{-q} \\
&\leq \limsup_{N \to \infty} \left\{o(1) + \sum_{n=N}^{\infty} cn^{\epsilon} n^{-q} \right\} =0.
\end{align*}

The same bound holds for $\#\Gamma^{(2)}_t$ in place of $\#\Gamma^{(1)}_t$ by exactly the same arguments. Finally, since $\delta > 0$ was arbitrary, we can replace $4$ with $\frac{1}{2}$ in the statement of the result. The final upper bound follows from applying Proposition \ref{prop:Zt bounds as}$(iii)$.
\end{proof}

As a consequence of Lemma \ref{lem:volGamma}, we deduce the following.

\begin{lem}\label{lem:Gammat gap}
$P$-almost surely under Assumption \ref{assn:as}, we have for all sufficiently large $t$ that
\begin{enumerate}[(i)]
\item $\sup_{v \in \Lambda_t \setminus \{\hat{Z}_t^{(1)}, \hat{Z}_t^{(2)}\}} \xi(v) \leq (\log t)^{-\frac{1}{\alpha}} a(t)$,
\item $\tilde{Z}_{\Lambda_t \setminus \{\hat{Z}_t^{(1)}\}} = \hat{Z}_t^{(2)}$ and $\tilde{Z}_{\Lambda_t \setminus \{\hat{Z}_t^{(2)}\}} = \hat{Z}_t^{(1)}$,
\item For any $\epsilon>0$ and each $i=1,2$, $\tilde{g}_{\Lambda_t \setminus \{\hat{Z}_t^{(i)}\}} \geq a(t)(\log t)^{-\epsilon}$.
\end{enumerate}
\end{lem}
\begin{proof} 
\begin{enumerate}[(i)]
\item Choose some $\delta < \frac{2}{\alpha-2}$. We have from Lemma \ref{lem:volGamma} and Proposition \ref{prop:Zt bounds as}$(iii)$ that
\[
\#\Lambda_t\leq (\log t)^{-(z-\delta)+\frac{2}{\alpha - 2}} r(t)^2
\]
eventually almost surely. On this event, since the potentials at different vertices are independent of each other, and since conditioning on $u \neq \hat{Z}^{(1)}_t, \hat{Z}^{(2)}_t$ can only (stochastically) reduce $\xi (u)$, we have from a union bound that for all sufficiently large $t$ and any $m>0$,
\begin{align*}
\p{\sup_{v \in \Lambda_t \setminus \{\hat{Z}_t^{(1)}, \hat{Z}_t^{(2)}\}} \xi(v) \geq (\log t)^{-m} a(t)} &\leq (\log t)^{-(z-\delta)+\frac{2}{\alpha - 2}} r(t)^2  (\log t)^{m\alpha} a(t)^{-\alpha} \leq (\log t)^{-y},
\end{align*}
where $y=(z-\delta)-\frac{2}{\alpha - 2}-m\alpha>1$ by our choice of $m$ and $z$. Then, by Lemma \ref{lem:BC counting} we can choose $\epsilon > 0$ small enough that
\begin{align*}
\p{\sup_{v \in \Lambda_t \setminus \{\hat{Z}_t^{(1)}, \hat{Z}_t^{(2)}\}} \xi(v) \geq (\log t)^{-m} a(t) \text{ i.o.}} &\leq \limsup_{T \to \infty} \sum_{t \geq T: \hat{Z}_t^{(2)} \neq \hat{Z}_{t^-}^{(2)}} \p{\sup_{v \in \Lambda_t \setminus \{\hat{Z}_t^{(1)}, \hat{Z}_t^{(2)}\}} \xi(v) \geq (\log t)^{-m} a(t)}\\
&\leq \limsup_{N \to \infty} \left\{ o(1) +  \sum_{n=N}^{\infty} \sum_{\substack{t \in [2^n, 2^{n+1}]: \\ \hat{Z}_t^{(2)} \neq \hat{Z}_{t^-}^{(2)}}} (\log t)^{-y} \right\}\\
&\leq \limsup_{N \to \infty} \left\{ o(1) + \sum_{n=N}^{\infty} cn^{\epsilon} n^{-y} \right\}=0.
\end{align*}
This gives the result.
\item Take some $\epsilon < m$. It follows from $(i)$ and Proposition \ref{prop:Zt bounds as}$(ii)$ that eventually almost surely, for each $i=1,2$,
\begin{align*}
    \xi(\hat Z_t^{(i)})-\deg(\hat Z_t^{(i)})\geq a(t)(\log t)^{-\epsilon} > a(t) (\log t)^{-\frac{1}{\alpha}} &\geq \sup_{v \in \Lambda_t \setminus \{\hat{Z}_t^{(1)}, \hat{Z}_t^{(2)}\}} (\xi(v)-\deg(v)),
\end{align*}
\item For $0<\epsilon<m$, the statement follows from (i) since $\xi(\hat{Z}_t^{(i)}) - \deg \hat{Z}_t^{(i)} \geq 2a(t) (\log t)^{-\epsilon}$ eventually almost surely by Proposition \ref{prop:Zt bounds as}$(ii)$ (clearly this implies the result for all $\epsilon \geq m$ as well).
\end{enumerate}
\end{proof}

We now use Lemma \ref{lem:Gammat gap} to control the paths contributing to \eqref{eqn:eigenfunction rep}, first defining some notation.

Given $x \in \Lambda_t$, we define the set of all paths from $x$ to $ \hat{Z}^{(1)}_t$ that end on the first hit to $ \hat{Z}^{(1)}_t$ as $\Gamma^{(1)}_{x,t}$. Similarly, if $x \in \Gamma^{(2)}_t$ let $\Gamma^{(2)}_{x,t}$ denote the set of all paths from $x$ to $\hat{Z}^{(2)}_t$ that end on the first hit to $ \hat{Z}^{(2)}_t$. If $x \in \Gamma^{(1)}_t \cup \Gamma^{(2)}_t$ we denote the direct path from $x$ to $\hat{Z}_t^{(1)}$ (respectively $\hat{Z}_t^{(2)}$) as $\gamma^{(1)}_{x,t}$ (respectively $\gamma^{(2)}_{x,t}$). Given $t>0$, we also introduce the notation $\pi(X_{[0,t]})$ to denote the path that consists of all the steps taken by the random walk $(X_s)_{s\geq 0}$  between times $0$ and $t$. 

For $t > 0$ and $x \in \Lambda_t$, we let
\begin{align}\label{eqn:gambad def}
\begin{split}
\gambadi &:= \left\{\gamma \in \Gamma^{(1)}_{x,t}: \prod_{v \in \gamma\setminus \gamma_{x,t}^{(1)}} \frac{\deg v}{\deg v - [\xi(v) + \deg \hat{Z}^{(1)}_{t} - \xi( \hat{Z}^{(1)}_{t} )]} \geq 1\right\}, \\
\gambadii &:= \left\{\gamma \in \Gamma^{(2)}_{x,t} : \prod_{v \in \gamma\setminus \gamma_{x,t}^{(2)}} \frac{\deg v}{\deg v - [\xi(v) + \deg \hat{Z}^{(2)}_{t} - \xi( \hat{Z}^{(2)}_{t} )]} \geq 1\right\}.
\end{split}
\end{align}

\begin{lem}\label{lem:eigenfunction direct path excursion contribution bound}
$P$-almost surely under Assumption \ref{assn:as}, there exists $t_0 < \infty$ such that for all $t \geq t_0$ and all $x \in \Lambda_t \setminus \{\hat Z_t^{(1)}, \hat Z_t^{(2)}\}$, $\gambadi = \gambadii = \emptyset$.
\end{lem}
\begin{proof}
It follows exactly as in the proofs of Proposition \ref{prop:Zt bounds as}$(ii)$ and $(iii)$ that eventually almost surely, for all $v \in \Lambda_t \setminus \{\hat Z_t^{(1)}, \hat Z_t^{(2)}\}$, $$\xi(v) < \xi(\hat{Z}_t^{(i)}) - \deg \hat{Z}_t^{(i)}$$
for each $i=1,2$, which implies that each term in the products in \eqref{eqn:gambad def} is less than $1$.
\end{proof}

\begin{lem}\label{lem:B(ZC) volume bound}
Fix some constant $C_{\alpha,\beta}$. Under Assumption \ref{assn:as}, for any $\epsilon > 0$, $\#B(\hat Z_t^{(2)},C_{\alpha,\beta})\leq~( \log t)^{\epsilon}$ eventually $P$-almost surely.
\end{lem}
\begin{proof}
By very similar arguments to Proposition \ref{prop:AR vol growth and exit time}$(ii)$, it holds that $\prb{\#V(B(\hat Z_t^{(2)},C_{\alpha,\beta})) \geq \lambda} \leq C\lambda^{-K}$ for any $K< \infty$. The proposition therefore follows by chooing $K> 2\epsilon^{-1}$ and applying Borel-Cantelli and Lemma \ref{lem:BC counting} similarly to the proofs of Lemmas \ref{lem:volGamma} and \ref{lem:Gammat gap}.
\end{proof}

\subsection{Bounds on eigenfunctions}\label{sctn:efunction bounds}

In this section we apply these results on $\Lambda_t$ to bound $\phi^{(1)}_{\Lambda_t \setminus \{\hat{Z}_t^{(1)}\}}$ and $\phi^{(1)}_{\Lambda_t \setminus \{\hat{Z}_t^{(2)}\}}$. In what follows we will assume that $\phi^{(1)}_{\Lambda_t \setminus \{\hat{Z}_t^{(1)}\}}$ is always normalised so that $\phi^{(1)}_{\Lambda}(\hat{Z}_t^{(2)})=1$, and vice versa.

\begin{lem}\label{lem:boundeigenfunction} 
Eventually $P$-almost surely under Assumption \ref{assn:as}, for each $i=1,2$ and each $x \in \Lambda_t\setminus~\{\hat Z_t^{(1)}, \hat Z_t^{(2)}\}$,
\[
\phi^{(1)}_{\Lambda_t \setminus \{\hat{Z}_t^{(1)}\}}(x)\leq \prod_{v \in \gamma_{x,t}^{(2)}} \frac{\deg v}{\tilde{g}_{\Lambda_t \setminus \{\hat{Z}_t^{(1)}\}}} \qquad \mathrm{and } \qquad  \phi^{(1)}_{\Lambda_t\setminus \{\hat{Z}_t^{(2)}\}}(x)\leq \prod_{v \in \gamma_{x,t}^{(1)}} \frac{\deg v}{\tilde{g}_{\Lambda_t \setminus \{\hat{Z}_t^{(2)}\}}}.
\]
\end{lem}
\begin{proof}
The proof follows a similar strategy to \cite[Lemma 2.3]{KLMStwocities09}. We only prove the first statement; the proof of the second statement works analogously. Firstly, it follows from the Rayleigh-Ritz formula (e.g. see \cite[Equation (2.7)]{HollanderKoenigSantosPAMOnTrees}) for the principal eigenvalue of the Anderson Hamiltonian that 
\begin{align*}
\lambda^{(1)}_{{\Lambda_t\setminus \{\hat Z_t^{(1)}\}}} &= \sup \{  \langle (\xi + \Delta)f, f\rangle_{\ell^2(\Ti)}: f \in \ell^2(\Ti), \text{supp}(f) \subset {\Lambda_t\setminus \{\hat Z_t^{(1)}\}}, ||f||_{2} = 1\} \\
&\geq \sup_{z \in {\Lambda_t\setminus \{\hat Z_t^{(1)}\}}} \{ \langle(\xi + \Delta) \delta_z, \delta_z\rangle\}\\& = \sup_{z \in {\Lambda_t\setminus \{\hat Z_t^{(1)}\}}} \{ \xi(z) - \deg (z)\}=\xi(\tilde Z_{\Lambda_t\setminus \{\hat Z_t^{(1)}\}})-\deg(\tilde Z_{\Lambda_t\setminus \{\hat Z_t^{(1)}\}}).
\end{align*}
Then, taking $y=\hat{Z}_t^{(2)}$ in \eqref{eqn:eigenfunction rep} we have for $x\in \Lambda_t\setminus\{\hat Z_t^{(1)},\hat Z_t^{(2)}\}$,
\begin{align*}
    \phi^{(1)}_{{\Lambda_t\setminus \{\hat Z_t^{(1)}\}}}(x) &= \estart{\exp\left\{\int_0^{H_{\hat{Z}_t^{(2)}}} (\xi(X_s) - \lambda^{(1)}_{{\Lambda_t\setminus \{\hat Z_t^{(1)}\}}})~\text{d} s\right\} \mathbbm{1} {\{H_{\hat{Z}_t^{(2)}} < \tau_{{\Lambda_t\setminus \{\hat Z_t^{(1)}\}}}\}}}{x} \\
    &\leq \estart{\exp\left\{\int_0^{H_{\hat{Z}_t^{(2)}}} (\xi(X_s) - [\xi (\tilde{Z}_{\Lambda_t\setminus \{\hat Z_t^{(1)}\}}) - \deg \tilde{Z}_{\Lambda_t\setminus \{\hat Z_t^{(1)}\}}]~\text{d} s\right\} \mathbbm{1} {\{H_{\hat{Z}_t^{(2)}} < \tau_{{\Lambda_t\setminus \{\hat Z_t^{(1)}\}}}\}}}{x}.
\end{align*}
Now let $(T_v)_{v \in \gamma_{x,t}^{(2)}}$ denote a set of independent exponential random variables with respective parameters $(\deg v)_{v \in \gamma_{x,t}^{(2)}}$. By Lemma \ref{lem:Gammat gap}$(ii)$, we can work on the event $\{\tilde Z_{\Lambda_t\setminus \{\hat Z_t^{(1)}\}}=\hat Z_t^{(2)}\}$. Since every path $\gamma \in \Gamma_{x,t}^{(2)}$ contains $\gamma_{x,t}^{(2)}$ as a subset, we can therefore apply Lemma \ref{lem:eigenfunction direct path excursion contribution bound} to deduce that, eventually $P$-almost surely,
\begin{align*}
    \phi^{(1)}_{{\Lambda_t\setminus \{\hat Z_t^{(1)}\}}}(x) &\leq \estart{\exp\left\{\int_0^{H_{\hat Z_t^{(2)}}} (\xi(X_s) - [\xi (\hat Z_t^{(2)}) - \deg \hat Z_t^{(2)}])~\text{d} s\right\} \mathbbm{1} {\{H_{\hat Z_t^{(2)}} < \tau_{{\Lambda_t\setminus \{\hat Z_t^{(1)}\}}}\}}}{x} \\
    &\leq \sum_{\gamma \in \Gamma_{x,t}^{(2)}} \pr{\pi(X_{[0,H_{\hat Z_t^{(2)}}]})=\gamma}   \prod_{v \in \gamma} \E{\exp \left\{ T_v [\xi(v) + \deg \hat Z_t^{(2)} - \xi( \hat Z_t^{(2)} )]\right\}} \\
      &\leq \sum_{\gamma \in \Gamma_{x,t}^{(2)}} \pr{\pi(X_{[0,H_{\hat Z_t^{(2)}}]})=\gamma}   \prod_{v \in \gamma}  \frac{\deg v}{\deg v - [\xi(v) + \deg \hat Z_t^{(2)} - \xi( \hat Z_t^{(2)} )]}.
    \end{align*}  
Applying  Lemma \ref{lem:eigenfunction direct path excursion contribution bound} and decomposing the product according to whether vertices are in $\gamma_{x,t}^{(2)}$ or not (if some vertices in $\gamma_{x,t}^{(2)}$ are visited more than once, we put repeat visits into the second product), we obtain that
    \begin{align*}
    \phi^{(1)}_{{\Lambda_t\setminus \{\hat Z_t^{(1)}\}}}(x)
       &\leq \prod_{v\in \gamma_{x,t}^{(2)}} \frac{\deg v}{\tilde g_{\Lambda_t\setminus \{\hat Z_t^{(1)}\}}} \sum_{\gamma \in \Gamma_{x,t}^{(1)}} \pr{\pi(X_{[0,H_{\hat Z_t^{(2)}}]})=\gamma}   \prod_{v \in \gamma\setminus \gamma_{x,t}^{(2)}}  \frac{\deg v}{\deg v - [\xi(v) + \deg \hat Z_t^{(2)}- \xi( \hat Z_t^{(2)} )]} \\
       &\leq \prod_{v\in \gamma_{x,t}^{(2)}} \frac{\deg v}{\tilde g_{\Lambda_t\setminus \{\hat Z_t^{(1)}\}}}.
\end{align*}
\end{proof}

Finally, we will need the following lemma to bound the function appearing in the statement of Lemma \ref{lem:boundeigenfunction}.

\begin{lem} \label{lem:technical} 
$P$-almost surely under Assumption \ref{assn:as}, for all sufficiently large $t$ we have the following.
\begin{enumerate} 
    \item [(i)]  There exist $D, E<\infty$ such that for each $i=1,2$, $$\sup_{x\in \Lambda_t \setminus \{\hat Z_t^{(1)}, \hat Z_t^{(2)}\}}\left\{\sum_{v\in \gamma_{x,t}^{(i)}}\log(\deg v)- [D\log t+E |\gamma_{x,t}|]\right\}\leq 0.$$
    \item [(ii)] Fix any $C_{\alpha,\beta}< \infty$. Then for any $\delta >0$, it holds for each $i=1,2$ that $$\sup_{x\in \left( \Lambda_t \setminus \{\hat Z_t^{(1)}, \hat Z_t^{(2)}\}\right) \cap V(B(\hat{Z}_t^{(i)} ,C_{\alpha,\beta}))} \left\{\sum_{v\in \gamma_{x,t}^{(i)}}\log(\deg v)- \delta\log \log t \right\}\leq 0.$$
\end{enumerate}
\end{lem}
\begin{proof}
 \begin{enumerate}
     \item [(i)] The statement follows from Lemma \ref{lem:Br log sum bound} and Proposition \ref{prop:Zt bounds as}$(iii)$.
 \item[(ii)] Fix $i=1$ or $i=2$ and write $C=C_{\alpha,\beta}$. Take any $\epsilon>0$, and choose $c> 2(\epsilon +1)\delta^{-1} + 1$. Working backwards along the path $\gamma^{(i)}_{x,t}$ starting from $\hat{Z}^{(i)}_{t}$ and ending at $x$, the degrees of the vertices independently satisfy the size-biased tail bound of \eqref{eqn:size biased} (by the same logic), so that $A:=~\sup_{t \geq 0} \sup_{v \in \gamma^{(i)}_{x,t}} \Eb{(\deg v)^{c}}<~\infty$.

Now let $x\in V(B(\hat{Z}^{(i)}_{t},C)) \setminus \{\hat{Z}^{(i)}_{t}\}.$ We compute, using a Chernoff bound, that
      \begin{align*}
\mathbf P\left(\sum_{v\in\gamma^{(i)}_{x,t}}\log(\deg v)\geq \delta\log \log t \right)&\leq \mathbf E\left[\exp\left(c\sum_{v\in \gamma^{(i)}_{x,t}}\log(\deg v)\right)\right]\exp(-c\delta\log \log t)\\
&\leq \exp\left(A|\gamma^{(i)}_{x,t}|-c\delta\log \log t\right)\\
&\leq e^{AC}(\log t)^{-c\delta}.
 \end{align*}

Now for $t>0$ let us define the event in question as
\begin{equation*}
    A_{t}:=\left\{\sup_{x\in \left( \Lambda_t\setminus \{\hat Z_t^{(1)}, \hat Z_t^{(2)}\}\right) \cap V(B(\hat{Z}_t^{(i)} ,C))} \left\{\sum_{v\in \gamma^{(i)}_{x,t}}\log(\deg v)- \delta\log \log t \right\}\geq 0\right\}.
\end{equation*}
Applying a union bound over all vertices in $B(\hat{Z}^{(i)}_{t},C)$ with the bound above and the bound in the proof of Lemma \ref{lem:B(ZC) volume bound} yields 
 \begin{align*}
P \left(A_t\right)
%&=\mathbf P\left(\sup_{x\in \left( \Lambda_t\setminus \{\hat Z_t^{(1)}, \hat Z_t^{(2)}\}\right) \cap B(\hat{Z}_t^{(i)} ,C)}\left\{\sum_{v\in \gamma^{(i)}_{x,t}}\log(\deg v)- \chi_{x,t} |\gamma^{(i)}_{x,t}|\right\}\geq 0\right) \\
\leq (\log t)^{\delta/2} e^{AC} (\log t)^{-c\delta}+\mathbf P(\#B(\hat{Z}^{(i)}_{t},C)>(\log t)^{\delta/2})
    &= e^{AC} (\log t)^{-(c-1)\delta/2}+c'(\log t)^{-c\delta/2}.
 \end{align*}
Combining with Lemma \ref{lem:BC counting} and a union bound, we deduce that
\begin{align*}
 \p{A_{t} \text{ i.o.}}= \limsup_{N \to \infty} \p{\left(A_{t}\right)_{t \geq 2^N} \text{ i.o.}} &\leq \p{\exists n \geq N: \#\left\{\hat{Z}^{(i)}_t:t\in[2^n,2^{n+1}]\right\} \leq (\log t_n)^{\epsilon}} + \sum_{n \geq N} n^{\epsilon} P(A_{2^n}) \\
    &\leq o(1) + C\sum_{n \geq  N} n^{\epsilon} n^{-(c-1)\delta/2} \to 0,
\end{align*}
provided $\hat{B}$ is large enough. This concludes the proof.
\end{enumerate}
\end{proof}
%%%%%%%%%%%%%%%%%%%%%%%%%%%%%%%%%%%%%%%%%%%

%%%%%%%%%%%%%%%%%%%%%%%%%%%%%%%%%%%%%%%%%%%%%%%%%%%%%%%%%%%

\begin{lem}\label{lem:loceigenfunction}
%\begin{enumerate}[(i)]
    %\item 
    Eventually $P$-almost surely under Assumption \ref{assn:as}, as $t \to \infty$
\begin{equation}\label{eqn:efunction loc as}
||\phi^{(1)}_{\Lambda_t \setminus \{\hat{Z}_t^{(1)}\}}||^2_2 \sum_{v \in \Lambda_t \setminus \{\hat{Z}_t^{(1)}, \hat{Z}_t^{(2)}\}} \phi^{(1)}_{\Lambda_t \setminus \{\hat{Z}_t^{(1)}\}}(v) + ||\phi^{(1)}_{\Lambda_t \setminus \{\hat{Z}_t^{(2)}\}}||^2_2 \sum_{v \in \Lambda_t \setminus \{\hat{Z}_t^{(1)}, \hat{Z}_t^{(2)}\}} \phi^{(1)}_{\Lambda_t \setminus \{\hat{Z}_t^{(2)}\}}(v)\to 0.
\end{equation}
\end{lem}
\begin{proof}
We start by bounding the first sum above. Take $\epsilon \in (0, \beta-1)$ and let $c:=\beta-1-\epsilon>0$. We set 
 $$C_{\alpha,\beta}=\frac{3(q+1)}{q}$$
and split the sum into two parts, summing separately over $v \in U_t := \Lambda_t \setminus \left(\{\hat{Z}_t^{(1)}, \hat{Z}_t^{(2)}\} \cup B(\hat{Z}_t^{(2)}, C_{\alpha, \beta}) \right)$ and $v \in W_t := \left(\Lambda_t \setminus \{\hat{Z}_t^{(1)}, \hat{Z}_t^{(2)}\} \right) \cap B(\hat{Z}_t^{(2)}, C_{\alpha, \beta})$.

We estimate the two parts separately in the following.
 \begin{itemize}
     \item \textbf{Part 1: $U_t$.} 
     Take some $\delta > 0, \epsilon > 0$ and $E< \infty$ as in Lemma \ref{lem:technical}$(i)$. Also let $h(t)=8(\log t)^{-(z-\delta)+\frac{2}{\alpha-2}}r(t)^2$ (as in Lemma \ref{lem:volGamma}). Define the event
 \begin{align*}
 E_1:=&\{\#\Lambda_t\leq h(t)\} \cap \left\{\sup_{v \in U_t}\left\{\sum_{x\in\gamma_{v,t}^{(2)}}\log(\deg x)- 2E|\gamma_{v,t}^{(2)}|\right\}\leq 0\right\}\\&\cap\left \{\tilde g_{\Lambda_t\setminus\{\hat Z_t^{(1)}\}}\geq a(t)(\log t)^{-\epsilon}\right\}\cap\left\{\phi^{(1)}_{\Lambda_t\setminus\{\hat Z_t^{(1)}\}}(v)\leq \prod_{x \in \gamma_{v,t}^{(2)}} \frac{\deg x}{\tilde{g}_{\Lambda_t\setminus\{\hat Z_t^{(1)}\}}}\text{ for all }v\in \Lambda_t\setminus\{\hat Z_t^{(1)},\hat Z_t^{(2)}\}\right\}.
\end{align*}
Note that $E_1$ holds eventually almost surely by Proposition \ref{prop:Zt bounds as}$(i)$ and Lemmas \ref{lem:volGamma}, \ref{lem:technical}, \ref{lem:Gammat gap}$(iii)$ and \ref{lem:boundeigenfunction}. Almost surely, i.e. on the event $E_1$, we can therefore calculate for all sufficiently large $t$ that
\begin{align*}
   \sum_{v \in U_t} \phi^{(1)}_{\Lambda_t \setminus \{\hat{Z}_t^{(1)}\}}(v)
   % &\leq \sum_{v \in U_t} \prod_{x \in \gamma_{v,t}^{(2)}} \frac{\deg x}{\tilde{g}_{\Lambda_t\setminus\{\hat Z_t^{(1)}\}}}\\
    &\quad \leq \sum_{v \in U_t} \exp\left(\sum_{x \in \gamma_{v,t}^{(2)}}[\log(\deg(x))-\log(\tilde g_{\Lambda_t\setminus \{\hat Z_t^{(1)}\}})]\right)\\
    &\quad \leq h(t) \sup_{v \in U_t}\exp\left(|\gamma_{v,t}^{(2)}|\left(q\log t - O(\log \log t) \right) \right)\\
    &\quad \leq \exp\left(-(C_{\alpha,\beta} q - 2(q+1))\log t + O(\log \log t)\right),
    %&\leq \exp\left(d \log t-C_{\alpha,\beta}\left[q\log t-(g+q)\log \log t-E\right]\right)\\
   % &\leq \exp\left(-C_{\alpha,\beta}\left[\frac{q}{2}\log t-(g+q)\log \log t-E\right]\right),
\end{align*}
which converges to $0$ as $t\rightarrow \infty$ by our choice of $C_{\alpha, \beta}$.
\item\textbf{ Part 2: $W_t$.}
 On the event 
 \begin{align*}E_2:=E_1 \cap\{\#B(\hat Z_t^{(2)},C_{\alpha,\beta})\leq \log t\}, \end{align*}
 By Lemma \ref{lem:B(ZC) volume bound} it follows that and Part 1 above, it follows that $E_2$ holds eventually almost surely. We therefore calculate similarly to the first part that almost surely for sufficiently large $t$, i.e. on the event $E_2$,
\begin{align*}
   \sum_{v \in W_t} \phi^{(1)}_{\Lambda_t \setminus \{\hat{Z}_t^{(1)}\}}(v) 
    &\leq \sum_{v \in W_t}  \exp\left(\sum_{x\in \gamma_{v,t}^{(2)}}[\log(\deg(x))-\log(\tilde g_{\Lambda_t\setminus\{\hat Z_t^{(1)}\}})]\right)\\
    &\leq \exp\left(2E C_{\alpha, \beta}- q \log t + O(\log \log t) \right),%\\
    %&\leq (\log t)^{O(1)}t^{-q},
\end{align*}
which again goes to $0$ as $t \to \infty$. \end{itemize}
Combining these two calculations we deduce that $\sum_{v \in \Lambda_t \setminus \{\hat{Z}_t^{(1)}, \hat{Z}_t^{(2)}\}} \phi^{(1)}_{\Lambda_t \setminus \{\hat{Z}_t^{(1)}\}}(v) \to 0$ almost surely.

This also implies that $\|\phi_{\Lambda_t\setminus\{\hat Z_t^{(1)}\}}^{(1)}\|_2^2 \leq \|\phi_{\Lambda_t\setminus\{\hat Z_t^{(1)}\}}^{(1)}\|_1^2 \leq (1 + o(1))^2$ almost surely as $t \to \infty$, so we deduce that the first term in \eqref{eqn:efunction loc as} goes to $0$ almost surely as $t \to \infty$. The proof for the second term works analogously. 
\end{proof}

\subsection{High probability spectral statements}

To prove the single site localisation whp (Theorem \ref{thm:main whp localisation intro}) we will have to prove localisation on a slightly bigger set. The strategy is the same as in the previous subsections, except that the results hold whp rather than almost surely. For this reason we do not give the full details of the proofs.

Define $\overline{R}_t = |\hat{Z}_t^{(1)}| \left(1 + \frac{1}{(\log t)^{\frac{1}{2}(1 + \frac{1}{\alpha})}} \right)$. We will prove localisation on $B_{\overline{R}_t}$ whp, for which we will need the following results.

Recall from Section \ref{sctn:gap potential} that 
\begin{align*}
Z_{B_{\overline{R}_t}} &= \arg_{z \in B_{\overline{R}_t}} \max \{\xi(z) \}, \hspace{2.3cm} \tilde{Z}_{B_{\overline{R}_t}} = \arg_{z \in B_{\overline{R}_t}} \max \{\xi(z) - \deg z \},\\ g_{B_{\overline{R}_t}} &= \xi(Z_{B_{\overline{R}_t}}) - \max_{z \in B_{\overline{R}_t}, z \neq Z_{B_{\overline{R}_t}}} \{\xi(z) \}, \hspace{1cm} \tilde{g}_{B_r} = \xi(\tilde Z_{B_r}) - \deg(\tilde Z_{B_{\overline{R}_t}}) - \max_{z \in B_{\overline{R}_t}, z \neq \tilde Z_{B_{\overline{R}_t}}} \{\xi(z) - \deg(z) \}.
\end{align*}

\begin{lem}\label{lem:boundeigenfunction whp}[cf Lemma \ref{lem:boundeigenfunction}].
Under Assumption \ref{assn:whp},
\[
P\left(\phi^{(1)}_{B_{\overline{R}_t}}(x)\leq \prod_{v \in \gamma_{x,r}} \frac{\deg v}{\tilde{g}_{B_{\overline{R}_t}}} \text{ for all } x \in V(B_{\overline{R}_t}) \right) \to 1\text{ as }r \to \infty.
\]
\end{lem}
\begin{proof}
Exactly as in Lemma \ref{lem:boundeigenfunction},
\begin{align*}
    \phi^{(1)}_{{B_{\overline{R}_t}}}(x) &\leq \prod_{v\in \gamma_{x,t}^{(1)}} \frac{\deg v}{\tilde g_{\Lambda_t\setminus \{\hat Z_t^{(1)}\}}} \sum_{\gamma \in \Gamma_{x,t}^{(1)}} \pr{\pi(X_{[0,H_{\hat Z_t^{(1)}}]})=\gamma}   \prod_{v \in \gamma\setminus \gamma_{x,t}^{(1)}}  \frac{\deg v}{\deg v - [\xi(v) + \deg \hat Z_t^{(1)}- \xi( \hat Z_t^{(1)} )]}.
\end{align*}
Now recall from Proposition \ref{prop:Zt bounds as}$(ii')$ and Lemma \ref{prop:max degree Ft Et}$(ii)$ that whp $\deg Z_{B_{\overline{R}_t}} \leq (\log t)^B$, and from Lemma \ref{lem:gap to infinity} that $g_{B_{\tilde{R}(t)}} \geq 2a(t) (\log t)^{\frac{-(1+\epsilon)}{\alpha}}$ for all sufficiently large $t$, almost surely. On these events, we have for all $v \neq \hat{Z}_t^{(1)}$ that
\[
\xi(\hat Z_t^{(1)}) - \deg \hat Z_t^{(1)} - \xi (v) > g_{B_{\tilde{R}(t)}}- \deg \hat{Z}_t^{(1)} > 0, 
\]
which implies that the corresponding term in the product on the right-hand side above is at most $1$.
\end{proof}

We will also need the following lemma concerning the degrees of vertices in the direct path $\gamma_{x,r}$.

\begin{lem} \label{lem:technical whp} [cf Lemma \ref{lem:technical}].
Under Assumption \ref{assn:whp},
\begin{enumerate} 
    \item [(i)]  There exist constants $D, E < \infty$ such that  $$P \left(\sup_{x\in B_r}\left\{\sum_{v\in \gamma_{x,r}}\log(\deg v)- [D\log r+ E |\gamma_{x,r}|]\right\}\leq 0\right)\rightarrow 1  \text{ as }r\to\infty,$$
    \item [(ii)]  Fix any $C< \infty$. Then for any $\delta >0$, $$P \left(\sup_{x\in V(B(\tilde Z_{B_r},C))\cap B_r\setminus \tilde Z_{B_r}}\left\{\sum_{v\in \gamma_{x,r}}\log(\deg v)- \delta \log \log r \right\}\leq 0\right)\rightarrow 1 \text{ as }r\to \infty.$$
\end{enumerate}
\end{lem}
\begin{proof}
The proof is exactly the same as that of Lemma \ref{lem:technical}, except that some of the relevant events hold whp but not almost surely. For part $(ii)$ we adapt the proof slightly by taking some $c \in (0, \beta - 1)$ and considering the event $\{\#B(\hat{Z}^{(i)}_{t},C)>(\log t)^{c\delta/2}\}$.
\end{proof}

%%%%%%%%%%%%%%%%%%%%%%%%%%%%%%%%%%%%%%%%%%%
We are now able to show that the principal eigenfunction is concentrated in $\tilde Z_{B_r}$. 
\begin{lem}\label{lem:loceigenfunction whp}
Under Assumption \ref{assn:whp},
\[ \|\phi_{B_r}^{(1)}\|_2^2\sum_{z\in B_r\setminus \tilde Z_{B_r}} \phi^{(1)}_{B_r}(z)\to 0 \text{ in $P$-probability as } r\rightarrow \infty.
\]
\end{lem}
\begin{proof}
The proof is exactly the same as that of Lemma \ref{lem:loceigenfunction}, except that some of the relevant events hold whp but not almost surely, we use Proposition \ref{prop:AR vol growth and exit time}$(i)$ to consider the event $\# \{ B_r \leq r^d(\log r)^{\epsilon}\}$ in place of $\{\Lambda_t \leq h(t)\}$, and we instead take $C_{\alpha,\beta}=\frac{(d+1)(q+1)}{q}$.
\end{proof}

%%%%%%%%%%%%%%%%%%%%%%%%%%%%%
%%% Section: Localisation 
%%%%%%%%%%%%%%%%%%%%%%%%%%%%%
\section{Almost sure two site localisation}
\label{sec:localisation}
In this section we prove Theorem \ref{thm:main a.s. localisation intro}. 
The strategy is a refined version of that of \cite[Section 5]{KLMStwocities09} and proceeds roughly as follows. Although this section is written under Assumption \ref{assn:as} in which case $d=2$, we write some proofs under Assumption \ref{assn:whp} in order to use them later for the proof of Theorem \ref{thm:main whp localisation intro}.

Recall from Section \ref{sctn:sites of high potential} that $R(t)=r(t)(\log t)^p$ where $p=q+2$. Recall also that $C = \frac{pd}{\alpha} + 1$, and
\begin{align*}%\label{eqn:Ft def reminder}
F_t &=\{v\in B_{R(t)}:\xi(v) \geq R(t)^{\frac{d}{\alpha}} (\log t)^{-C}\}, \\
E_t &= \{v \in B_{R(t)} \setminus F_t: \exists z \in F_t \text{ such that } \xi(z) - \deg (z) \leq \xi(v) - \deg (v)\}, \\
\Lambda_t &= \bigcup_{i=1,2} \left\{ v \in V(\Ti): d(v, \hat{Z}_t^{(i)}) + \min \{ |v|, |\hat{Z}_t^{(i)}|\} \leq  \left( 1 + (\log t)^{-z}\right)|\hat{Z}_t^{(i)}|\} \right\}.
\end{align*}
Now take some $\epsilon > 0$. We define
\begin{align}\label{eqn:u1 def}
u_1(t,v)&:=\mathbb E_0\left[\exp\left\{\int_0^t\xi(X_s)ds\right\}\mathbbm 1\{X_t=v\}\mathbbm 1\{\tau_{\Lambda_t}> t,H_{\left\{\hat Z_t^{(1)},\hat{Z}_t^{(2)}\right\}}\leq t\}\right].
\end{align}
and set $U_1(t) := \sum_{v \in V(\Ti)} u_1(t,v)$.

Also let $J_t$ denote the number of jumps made by $(X_s)_{s \geq 0}$ up until time $t$, and set $j_t = t^{\frac{d+\epsilon}{\beta}+1}r(t)$. We also define the five following contributions to the remaining total mass:
\begin{align}\label{eqn:Ui defs}
\begin{split}
U_2(t)&:=\mathbb E_0\left[\exp\left\{\int_0^t\xi(X_s)ds\right\} \right.\\
&\qquad \qquad \left. \mathbbm 1\{\tau_{B_{R(t)}}>t, \tau_{\Lambda_t}< t,H_{\left\{\hat Z_t^{(1)},\hat{Z}_t^{(2)}\right\}}\leq t,\arg \max_{s\in[0,t]}[\xi(X_s)-\deg(X_s)]\in\{\hat Z_t^{(1)},\hat Z_t^{(2)}\}, J_t\leq j_t\}\right],\\
U_3(t)&:=\mathbb E_0\left[\exp\left\{\int_0^t\xi(X_s)ds\right\} \mathbbm 1\{\tau_{B_{R(t)}}\leq t\}\right],\\
U_4(t)&:=\mathbb E_0\left[\exp\left\{\int_0^t\xi(X_s)ds\right\} \right.\\
&\qquad \qquad \left.\mathbbm 1\{\tau_{B_{R(t)}}>t,\exists u\in (F_t \cup E_t) \setminus \{\hat Z_t^{(1)},\hat Z_t^{(2)}\}: \arg \max_{s\in[0,t]}[\xi(X_s)-\deg(X_s)]=u,J_t\leq j_t\}\right],\\
U_5(t)&:=\mathbb E_0\left[\exp\left\{\int_0^t\xi(X_s)ds\right\} \mathbbm 1\{\tau_{B_{R(t)}}>t,H_{F_t}>t\} \right],\\
U_6(t)&=\mathbb E_0\left[\exp\left\{\int_0^t\xi(X_s)ds\right\} \mathbbm 1\{\tau_{B_{R(t)}}>t, J_t\geq j_t\}\right],
\end{split}
\end{align}
We will show in Proposition \ref{prop:u1} that $u_1(t,v)$ localises on $\left\{\hat{Z}^{(1)}_t, \hat{Z}^{(2)}_t\right\}$, and in Propositions \ref{prop:U2 zero} to \ref{prop:U6 0} that $\frac{U_i(t)}{U(t)} \to 0$ for all $i\in\{2,3,4,5,6\}$. This implies the result since $u(t,v)$ is non-negative and $u(t,v) \leq \sum_{i \leq 6} u_i(t,v)$ for all $t > 0, v \in V(\Ti)$.

We first derive some bounds for the contribution to $U$ from a specific path. This will be particularly useful for proving a lower bound on $U(t)$ in Proposition \ref{prop:total mass LB}, since the strategy employed to prove the analogous result in \cite[Proposition 4.2]{KLMStwocities09} does not carry through in the variable degree setting.

\begin{lem}\label{lem:pathbounds}
Let $\gamma = v_0, \ldots, v_n$ be a path in $\Ti$, and let $i^* = \arg \max_{i \leq n} \left[ \xi(v_i) - \deg (v_i)\right]$ (taking $i^*$ to be the index of the first visit to $v_{i^*}$ if $v_{i^*}$ is visited more than once). Then:
\begin{enumerate}[(i)]
\item \begin{align*}
\E{\exp \left\{ \int_0^t \xi (X_s) ds \right\}\mathbbm{1}\{{\pi}(X_{[0,t]}) = \gamma\}} &\leq e^{t (1+\xi (v_{i^*}) - \deg (v_{i^*}))} \prod_{i \neq i^*} \frac{1}{|1 + [\xi(v_{i^*}) - \deg (v_{i^*})] - [\xi( v_i) - \deg (v_i)]|}.
\end{align*}
\item \begin{align*}
\E{\exp \left\{ \int_0^t \xi (X_s) ds \right\}\mathbbm{1}\{{\pi}(X_{[0,t]}) = \gamma\}} &\geq e^{t (\xi (v_{i^*}) - \deg (v_{i^*}))} \prod_{i \neq i^*} \frac{1-e^{-\frac{t}{n}|[\xi(v_{i^*}) - \deg (v_{i^*})] - [\xi( v_i) - \deg (v_i)]|}}{|[\xi(v_{i^*}) - \deg (v_{i^*})] - [\xi( v_i) - \deg (v_i)]+1|}.
\end{align*}
\end{enumerate}
If $\gamma$ is a direct path, the ``$+1$'' terms are not necessary.
\end{lem}
\begin{proof}
Exactly the same computation as in the proof of Proposition \ref{prop:FK time reversal}, but but exchanging the roles of $s_0$ and $s_{i^*}$ (i.e. instead applying the substitution $s_n = t-\sum_{i=0}^{n-1}s_i = t-\sum_{i \neq i^*}s_i - s_{i^*}$ to the integral over $s_{i^*}$) shows that
\begin{align*}
&\E{\exp \left\{ \int_0^t \xi (X_s) ds \right\} \mathbbm{1}\{{\pi}((X)_{[0,t]}) = \gamma\}} \\
&= e^{t(\xi(v_{i^*}) - \deg v_{i^*})} {\displaystyle\int}_{(0, \infty)^n} \exp\left\{\sum_{i \neq i^*} y_i \left([\xi( v_i) - \deg v_i] - [\xi(v_{i^*}) - \deg v_{i^*}]\right) \right\} \mathbbm{1}\left\{\sum_{i \neq i^*} y_i < t \right\} \prod_{i\neq i^*} d y_i
\end{align*}
\begin{enumerate} [(i)]
    \item To obtain the upper bound, we suppose instead that $\xi(v_{i^*})$ is replaced with $\xi(v_{i^*})+1$ only on the first visit to $v_{i^*}$. This can only increase the expectation in question. We can then remove the indicator function from the above integral and integrate over $(0, \infty)^n$ to obtain the stated upper bound.
    \item For the lower bound, we instead replace $\xi(v_{i})$ with $\xi(v_{i})-1$ for all $i$ such that $v_i = v_{i^*}$ but $i \neq i^*$. We then integrate each variable over the interval $(0, \frac{t}{n})$. \qedhere
\end{enumerate}
\end{proof}

We start by using this to give a lower bound for the total mass $U(t)$. 

\begin{prop}\label{prop:total mass LB}
\begin{enumerate}[(i)]
    \item Under Assumption \ref{assn:as}, there exists $c>0$ such that for any $\epsilon > 0$,
\begin{align*}
U(t) = \E{\exp \left\{ \int_0^t \xi (X_s) ds \right\}} &\geq \exp \{t \psi_t (\hat{Z}_t^{(1)}) - r(t) t^{-c} \} \geq \exp \{-ta(t) (\log t)^{-\epsilon}\}
\end{align*}
eventually almost surely.
    \item Under Assumption \ref{assn:whp}, the same statement holds with high probability as $t \to \infty$.
\end{enumerate}
\end{prop}
\begin{proof}
\begin{enumerate}[(i)]
    \item

 From Lemma \ref{lem:pathbounds}(ii), taking $\gamma$ to be the direct path from $O$ to $\hat{Z}_t^{(1)}$, in which case $v_{i^*} = \hat{Z}_t^{(1)}$, $n = |\hat{Z}_t^{(1)}|$ and $\psi_t (\hat{Z}_t^{(1)}) > 0$, we see that
\begin{align}\label{eqn:mass LB product}
\begin{split}
&\E{\exp \left\{ \int_0^t \xi (X_s) ds \right\}\mathbbm{1}\{{\pi}((X)_{[0,t]}) = \gamma\}} \\
&\qquad \geq e^{t (\xi (v_{i^*}) - \deg (v_{i^*}))} \prod_{i \neq i^*} \frac{1-e^{-\frac{t}{n}|[\xi(v_{i^*}) - \deg (v_{i^*})] - [\xi( v_i) - \deg (v_i)]|}}{|[\xi(v_{i^*}) - \deg (v_{i^*})]|} \\
&\qquad \geq e^{t (\xi (v_{i^*}) - \deg (v_{i^*})) - |v_{i^*}| \log [\xi(v_{i^*}) - \deg (v_{i^*})]} \prod_{i \neq i^*} \left( 1-e^{-\frac{t}{n}|[\xi(v_{i^*}) - \deg (v_{i^*})] - [\xi( v_i) - \deg (v_i)]|} \right).
%&\geq e^{t (\xi (v_{i^*}) - \deg (v_{i^*})) - |v_{i^*}| \log [\xi(v_{i^*}) - \deg (v_{i^*})]} \prod_{i \neq i^*} \left( 1-e^{-\frac{ta(t)(\log t)^{-g}}{r(t) (\log t)^{\frac{1}{\alpha - d} + \epsilon}}} \right) \\
%&\geq e^{t \psi_t (\hat{Z}_t^{(1)})} \left( 1-e^{-(\log t)^{1-g - \frac{1}{\alpha - d} - \epsilon}} \right)^{|\hat{Z}_t^{(1)}|}.
\end{split}
\end{align}

Since $v_{i^*} = \hat{Z}_t^{(1)}$, $n = |\hat{Z}_t^{(1)}|$ and $\psi_t (\hat{Z}_t^{(1)}) > 0$, it follows from Proposition \ref{prop:Zt bounds as}$(ii)$ that eventually almost surely,
\[
\xi(v_{i^*}) - \deg (v_{i^*}) \geq \frac{q}{2}\frac{n \log t}{t}.
\]
Therefore, additionally applying Lemma \ref{lem:Gammat gap}$(i)$ and Proposition \ref{prop:Zt bounds as}$(ii)$ we deduce that eventually almost surely for all $i \neq i^*$
\[
\frac{t}{n}|[\xi(v_{i^*}) - \deg (v_{i^*})] - [\xi( v_i) - \deg (v_i)]| \geq \frac{1}{2} \frac{t}{n}[\xi(v_{i^*}) - \deg (v_{i^*})] \geq \frac{q}{4}\log t.
\]
Therefore, the final term in \eqref{eqn:mass LB product} above is lower bounded by
\begin{align*}
e^{t (\xi (v_{i^*}) - \deg (v_{i^*})) - |v_{i^*}| \log [\xi(v_{i^*}) - \deg (v_{i^*})]} \prod_{i \neq i^*} \left( 1-e^{-\frac{q}{4}\log t} \right) = e^{t \psi_t (\hat{Z}_t^{(1)})} \left( 1-e^{-q(\log t)/4} \right)^{|\hat{Z}_t^{(1)}|}.
\end{align*}

Therefore we have for all sufficiently large $t$ that (using that $\log (1-x) \geq -2x$ for $|x| <1/2$ and then Proposition \ref{prop:Zt bounds as}$(i)$),
\begin{align*}
\E{\exp \left\{ \int_0^t \xi (X_s) ds \right\}\mathbbm{1}\{{\pi}((X)_{[0,t]}) = \gamma\}} &\geq \exp \left\{t \psi_t (\hat{Z}_t^{(1)}) + |\hat{Z}_t^{(1)}| \log \left( 1-e^{-q(\log t)/4} \right) \right\} \\
&\geq \exp \{t \psi_t (\hat{Z}_t^{(1)}) - |\hat{Z}_t^{(1)}| e^{-q(\log t)/4} \} \\
&\geq \exp \{t \psi_t (\hat{Z}_t^{(1)}) - r(t) e^{-q(\log t)/8} \}.
\end{align*}
The deterministic lower bound then follows from Proposition \ref{prop:Zt bounds as}$(ii)$.
\item We can modify the proof above, noting that for any $\epsilon>0, |v_{i^*}| = |\hat{Z}_t^{(1)}| \leq r(t) (\log t)^{\epsilon}$ whp. Moreover, the same proof as in Lemma \ref{lem:Gammat gap}$(i)$ shows that $\inf_{i \neq i^*} |[\xi(v_{i^*}) - \deg (v_{i^*})] - [\xi( v_i) - \deg (v_i)]| \geq c \log t$ whp as $t \to \infty$, as above. The deterministic lower bound then follows from Proposition \ref{prop:Zt bounds as}$(ii')$.
\end{enumerate}
\end{proof}

\begin{prop}\label{prop:u1}
Under Assumption \ref{assn:as}, $P$-almost surely as $t\to \infty$
\begin{equation}\label{eqn:u1}
\frac{\sum_{v\in \Ti\setminus\{\hat Z_t^{(1)},\hat Z_t^{(2)}\}}u_1(t,v)}{U(t)}\to 0
\end{equation}
\end{prop}
\begin{proof}
By the definition in \eqref{eqn:u1 def}, $u_1$ is of the form \eqref{eqn:FeynmanKac bounded Omega} with $\Lambda=\Lambda_t$ and $\Omega=\{\hat Z_t^{(1)},\hat Z_t^{(2)}\}$. Moreover $\tilde g_{\Omega,\Lambda} >0$ by Lemma \ref{lem:Gammat gap}$(ii)$. The result therefore follows directly from Lemmas \ref{lem:efunction compare}$(ii)$ and \ref{lem:loceigenfunction}.
\end{proof}

\begin{prop}\label{prop:U2 zero}
Under Assumption \ref{assn:as}, $P$-almost surely as $t\to \infty$
$$\frac{U_2(t)}{U(t)} \to 0.$$
\end{prop}
\begin{proof}
Recall the definition of $F_t$ from Section \ref{sctn:sites of high potential}.

The proof is similar to \cite[Lemma 4.1]{KLMStwocities09}. Wlog we assume that $\xi (\hat{Z}^{(1)}_t) - \deg (\hat{Z}^{(1)}_t) \geq \xi (\hat{Z}^{(2)}_t) - \deg (\hat{Z}^{(2)}_t)$ (otherwise we can just switch the roles of $\hat{Z}^{(1)}_t$ and $\hat{Z}^{(2)}_t$ in all that follows). %Let $\Gamma_{B_{R(t)}}$ denote the set of all paths in $B_{R(t)}$.
 We first consider the following sets of paths:
\begin{align*}
\hat \Gamma_t^{(1)} &:=\left\{\gamma \subset B_{R(t)}: \hat{Z}_t^{(1)}\in \gamma, \gamma\nsubseteq \Gamma_t^{(1)}, \arg \max_{v\in \gamma}[\xi(v)-\deg(v)]=\hat Z_t^{(1)},|\gamma|\leq j_t\right\}, \\
\hat \Gamma_t^{(2)} &:=\left\{\gamma \subset B_{R(t)}: \hat{Z}_t^{(2)}\in \gamma, \hat{Z}_t^{(1)}\notin \gamma, \gamma\nsubseteq \Gamma_t^{(2)}, \arg \max_{v\in \gamma}[\xi(v)-\deg(v)]=\hat Z_t^{(2)},|\gamma|\leq j_t\right\}.
\end{align*}
Furthermore, recall that $\gamma_{O,t}^{(1)}$ denotes the direct path from $O$ to $\hat{Z}_t^{(1)}$. Now let $\gamma \in \hat{\Gamma}_t^{(1)}$ with $\gamma = v_1, \ldots, v_n$ and let $v^* = \arg_{v \in \gamma} \max (\xi(v) - \deg v)$, $i^* = \inf\{j \geq 1: v_j = v^*\}$. From Lemma \ref{lem:pathbounds}$(i)$ we have that (since $v_{i^*} = \hat{Z}_t^{(1)}$ for $\gamma \in \hat{\Gamma}_t^{(1)}$)
\begin{align}\label{eqn:equationproduct}
&\E{\exp \left\{ \int_0^t \xi (X_s) ds \right\}\mathbbm{1}\{{\pi}((X)_{[0,t]}) = \gamma\}} \leq e^{t (1+\xi (v_{i^*}) - \deg (v_{i^*}))} \prod_{i \neq i^*} \frac{1}{1+[\xi(v_{i^*}) - \deg (v_{i^*})] - [\xi( v_i) - \deg (v_i)]}.
\end{align}
Note that since $\gamma$ escapes $\Gamma_t^{(1)}$ and hits $\hat{Z}_t^{(1)}$ we have that $n\geq (1+(\log t)^{-z})|\hat Z_t^{(1)}|$ by definition of $\Gamma_t^{(1)}$. Since $\gamma$ hits $\hat Z_t^{(1)}$ it must also contain $\gamma_{O,t}^{(1)}$ as a subpath, which contains at least $|\hat Z_t^{(1)}|-(\log t)^D$ vertices in $F_t^c$ by Proposition \ref{prop:max degree Ft Et}$(i)$. Furthermore, for the excursions from the direct path we have by Proposition \ref{lem:useful things for Ft and Gt}$(ii)$ and the fact that all excursions must have even length that at least $\frac{n-|\hat Z_t^{(1)}|}{2}$ vertices are in $F_t^c$. The intuition behind the proof is that each of these vertices contributes a penalty term to \eqref{eqn:equationproduct}, which reduces the contribution from paths in $\hat{\Gamma}_t^{(1)}$.

Each term in the product above corresponding to a vertex in $F_t^c$ is upper bounded by $\frac{1}{1+\xi(\hat{Z}_t^{(1)}) - \deg (\hat{Z}_t^{(1)}) - R(t)^{\frac{d}{\alpha}}(\log t)^{-C}}$. Moreover, for any $\delta > 0$, $\xi(\hat{Z}_t^{(1)})-\deg(\hat{Z}_t^{(1)}) \geq r(t)(\log t)^{-\delta} = R(t)^{d/\alpha}(\log t)^{-(C-1-z+\delta)}$ eventually almost surely by Proposition \ref{prop:Zt bounds as}$(ii)$. Incorporating this logic into \eqref{eqn:equationproduct}, and using that $\log (1-x) \geq -2x$ for all sufficiently small $x$, we therefore deduce that $\E{\exp \left\{ \int_0^t \xi (X_s) ds \right\}\mathbbm{1}\{{\pi}((X)_{[0,t]}) = \gamma\}}$ is upper bounded by
\begin{align*}
%&\E{\exp \left\{ \int_0^t \xi (X_s) ds \right\}\mathbbm{1}\{{\pi}((X)_{[0,t]}) = \gamma\}} \\ &\leq
&e^{t (1+\xi (\hat{Z}_t^{(1)}) - \deg (\hat{Z}_t^{(1)}))} \left( \frac{1}{1+[\xi(\hat{Z}_t^{(1)}) - \deg (\hat{Z}_t^{(1)})] - R(t)^{\frac{d}{\alpha}}(\log t)^{-C}}\right)^{\frac{n-|\hat Z_t^{(1)}|}{2}+|\hat Z_t^{(1)}|-(\log t)^D}\\
&\leq \exp\left\{t(\xi(\hat{Z}_t^{(1)})-\deg(\hat{Z}_t^{(1)}))-|\hat Z_t^{(1)}|\log(\xi(\hat{Z}_t^{(1)})-\deg(\hat{Z}_t^{(1)}))+t\right.\\
&\qquad \left. -\left(\frac{n-|\hat Z_t^{(1)}|}{2}-(\log t)^D\right)\log(\xi(\hat{Z}_t^{(1)})-\deg(\hat{Z}_t^{(1)}))-\left(\frac{n+|\hat Z_t^{(1)}|}{2}-(\log t)^D\right) \log\left(1-\frac{R(t)^{\frac{d}{\alpha}}(\log t)^{-C}-1}{\xi(\hat{Z}_t^{(1)})-\deg(\hat{Z}_t^{(1)})}\right)\right\} \\
&\leq \exp\left\{t\psi_t(\hat{Z}_t^{(1)}) -\frac{n-|\hat{Z}_t^{(1)}|}{2}\left[\log(\xi(\hat{Z}_t^{(1)})-\deg(\hat{Z}_t^{(1)}))- (\log t)^{-(1+z-\delta)}\right] + |\hat{Z}_t^{(1)}|(\log t)^{-(1+z-\delta)} + O \left(t\right) \right\} \\
&\leq \exp\left\{t\psi_t(\hat{Z}_t^{(1)}) - c_q(n-|\hat{Z}_t^{(1)}|) \log t + 2|\hat{Z}_t^{(1)}|(\log t)^{-(1+z-\delta)} \right\},
\end{align*}
for example taking $c_q = \frac{q}{3}$ (in fact we can take any $c_q < \frac{q}{2}$). Here we used that $ \log(\xi(\hat{Z}_t^{(1)})-\deg(\hat{Z}_t^{(1)}))- (\log t)^{-(1+z-\delta)} \geq \frac{2q}{3} \log t$ eventually almost surely by Proposition \ref{prop:Zt bounds as}$(ii)$.

Thus, defining $U(\hat \Gamma_t^{(1)}):=\E{\exp \left\{ \int_0^t \xi (X_s) ds \right\}\mathbbm{1}\{{\pi}((X)_{[0,t]})\in \hat \Gamma_t^{(1)}\}}$ and applying Lemma \ref{lem:no of paths}, we deduce from the final line above that for any $\epsilon > 0$, eventually almost surely:
\begin{align*}\label{eqn:eps in path effect}
    &\log(U(\hat \Gamma_t^{(1)}))\\&=\log\left(\sum_{(1+(\log t)^{-z})|\hat Z_t^{(1)}| \leq n \leq j_t} \sum_{\gamma: |\gamma|=n, \gamma \in \hat{\Gamma}_t^{(1)}} \E{\exp \left\{ \int_0^t \xi (X_s) ds \right\}\mathbbm{1}\{{\pi}((X)_{[0,t]})= \gamma\}}\right) \\
     &\leq \sup_{(1+(\log t)^{-z})|\hat Z_t^{(1)}| \leq n}\left\{ \max_{\gamma \in \hat{\Gamma}_t^{(1)}: |\gamma|=n} \log \left(\E{\exp \left\{ \int_0^t \xi (X_s) ds \right\}\mathbbm{1}\{{\pi}((X)_{[0,t]})= \gamma\}}\right) +\eta_{R(t)}(n,|\hat Z_t^{(1)}|) \right\} +\log\left(j_t\right)\\
%     &\leq \sup_{(1+A(\log t)^{-z})|\hat Z_t^{(1)}| \leq n \leq j_t}\left\{ t\psi_t(\hat Z_t^{(1)})+t+C_1\log(t)+C_2t^{2\alpha(q+1)\delta}\log(t)-\frac{n-|\hat Z_t^{(1)}|}{2}\log(\xi(\hat{Z}_t^{(1)})-\deg(\hat{Z}_t^{(1)}))\right.\\
%    &\qquad \left.-\left(\frac{n-|\hat Z_t^{(1)}|}{2}+|\hat Z_t^{(1)}|\right) \log\left(1-\frac{R(t)^{\frac{d}{\alpha}-\delta}-1}{\xi(\hat{Z}_t^{(1)})-\deg(\hat{Z}_t^{(1)})}\right)+\eta_{R(t)}(n,|\hat Z_t|)\right\}+\log\left(j_t\right)\\
    &\leq \sup_{(1+(\log t)^{-z})|\hat Z_t^{(1)}| \leq n}\left\{ t\psi_t(\hat Z_t^{(1)}) + (n-|\hat Z_t^{(1)}|)\left(\epsilon \log (R(t)) + \log \left( \frac{n}{n-|\hat Z_t^{(1)}|} \right)- c_q \log t \right) + 3|\hat{Z}_t^{(1)}|(\log t)^{-(1+z-\delta)} \right\}
\end{align*}
For all $n \geq (1 + (\log t)^{-z})|\hat{Z}_t^{(1)}|$, the coefficient of $n$ in the final line above is eventually almost surely upper bounded by 
\begin{align*}
-c_q\log t+\epsilon \log (R(t)) + \log \left( \frac{n}{n-|\hat Z_t^{(1)}|} \right) \leq -c_q\log t+\epsilon \log (R(t)) + \log \left( \frac{1+(\log t)^{-z}}{(\log t)^{-z}} \right) < -\frac{c_q}{2}\log t,
\end{align*}
provided we chose $\epsilon > 0$ sufficiently small. %Since $\log \left( \frac{n}{n-|\hat Z_t^{(1)}|} \right)$ is also decreasing in $n$, 
This implies that the supremum above is attained at the minimal value of $n$. Therefore, we deduce that
\begin{align*}
\log(U(\hat \Gamma_t^{(1)})) \leq t\psi_t(\hat Z_t^{(1)}) - (n-|\hat Z_t^{(1)}|) \frac{c_q}{2} \log t + 3|\hat{Z}_t^{(1)}|(\log t)^{-(1+z-\delta)} &\leq t\psi_t(\hat Z_t^{(1)}) - \frac{c_q}{3} |\hat{Z}_t^{(1)}|(\log t)^{1-z}.
%\leq t\psi_t(\hat Z_t^{(1)}) - r(t)(\log t)^{1-z-\epsilon}.
\end{align*}
Invoking Proposition \ref{prop:total mass LB} thus yields that, almost surely,
\[
\log\left(\frac{U(\hat{\Gamma}_t^{(1)})}{U(t)}\right) \leq - \frac{c_q}{3} r(t)(\log t)^{1-z} +r(t)t^{-c} \to -\infty.
\]
%
%Now, we consider the set of paths
%$$\hat \Gamma_t^{(2)}:=\left\{\gamma \in \Gamma_{B_{R(t)}}: \hat{Z}_t^{(2)}\in \gamma, \hat{Z}_t^{(1)}\notin \gamma,  \gamma\nsubseteq \Gamma_t^{(2)}, \arg \max_{v\in \gamma}[\xi(v)-\deg(v)]=\hat Z_t^{(2)},|\gamma|\leq j_t\right\},$$
%and with a similar argument as given above it holds 
%\[ \log\left(\frac{U(\hat{\Gamma}_t^{(2)})}{U(t)}\right) \to -\infty.
%\]
An identical argument shows that same result with $\hat{\Gamma}_t^{(2)}$ in place of $\hat{\Gamma}_t^{(1)}$. This proves the result since $U_2(t) = U(\hat{\Gamma}_t^{(1)})+U(\hat{\Gamma}_t^{(2)})$.
\end{proof}

\begin{rmk}
In the proof above, if the $\epsilon$ in the upper bound for $\eta_{R(t)}(n,|\hat Z_t|)$ is replaced by a constant larger than $c_q$ (as could be the case under Assumption \ref{assn:whp}, see Remark \ref{rmk:whp assn problem}) this argument fails. Since we can take any $c_q < \frac{q}{2}$, this in particular fails when $\frac{d}{2\beta} > \frac{q}{2}$, or equivalently when $\frac{\beta^2}{\beta - 1} = \beta + d < \alpha$.
\end{rmk}

\begin{prop}\label{prop:U3 0}
Under Assumption \ref{assn:whp}, $P$-almost surely as $t\to \infty$
$$\frac{U_3(t)}{U(t)} \to 0.$$
\end{prop}
\begin{proof}
Recall that $p=\max\left\{h,q+1\right\}+1$. Since $\alpha > d$, we can therefore choose $\epsilon > 0$ small enough that
\begin{equation}\label{eqn:stricty}p>\max\left\{q+1,(q+1)\frac{1}{\alpha}\left(d+\varepsilon\right)\right\}.\end{equation} By Corollary \ref{cor:exitball} and Corollary \ref{lem:xi in Ar} we have that $P$-almost surely for $t$ sufficiently large
\begin{align*}
U_3(t) \leq \sum_{r\geq R(t)} \mathbb E\left[ \exp \left\{t \sup_{v \in B_r} \xi(v)\right\}\mathbbm 1\left\{\sup_{s\leq t}|X_s|=r\right\}\right] &\leq \sum_{r\geq R(t)} \exp \left\{t \sup_{v \in B_r} \xi(v)\right\} \pr{\tau_{B_{r-1}} \leq t} \\
&\leq \sum_{r \geq r(t)(\log t)^p} \exp \left\{t r^{\frac{d}{\alpha}}(\log r)^{\frac{1}{\alpha}\left(d+\varepsilon\right)} -\frac{r}{5}\log \left(\frac{r}{et}\right) \right\} .
\end{align*}
Now note that for $r \geq r(t)(\log t)^p$, %and $t$ sufficiently large we have
%\begin{align*}
%r\log\left(\frac{r}{et}\right)&= r(t)(\log t)^p \log \left(\frac{r(t)(\log t)^p}{et}\right)\\
%&=r(t)(\log t)^p \log(t^q(\log t)^{p-q-1}/e)\\
%&\geq q r(t) (\log t)^{p+1},
%\end{align*}
%and, using that $\frac{d}{\alpha}+\frac{1}{q+1}=1$,
%\begin{align*}
%t r^{\frac{d}{\alpha}}(\log r)^{\frac{1}{\alpha}\left(d+\varepsilon\right)}&=r(t)(\log t)^{1+\frac{pd}{\alpha}}(\log(r(t)(\log t)^p))^{\frac{1}{\alpha}\left(d+\varepsilon\right)}\\
%&\leq C r(t)(\log t)^{1+\frac{pd}{\alpha}+\frac{p}{q+1}-\delta}\\
%&=C r(t)(\log t)^{1+p-\delta},
%\end{align*}
%for some $\delta>0$ (since we have a strict inequality in \eqref{eqn:stricty}). 
%In particular this means that for $t$ large enough 
\begin{equation}\label{eqn:smoui}
t r^{\frac{d}{\alpha}}(\log r)^{\frac{1}{\alpha}\left(d+\varepsilon\right)} < \frac{1}{10}r \log \left(\frac{r}{et}\right).
\end{equation}
provided $t$ is sufficiently large.
%Since 
%$$ \frac{d}{dr} \left\{-\frac{1}{2}r \log \left(\frac{r}{et}\right)+t r^{\frac{d}{\alpha}}(\log r)^{\frac{1}{\alpha}\left(d+\varepsilon\right)}\right\}<0,$$ 
%for $r$ large enough, \eqref{eqn:smoui} actually holds for all $r \geq r(t)(\log t)^p$ and $t$ sufficiently large. 
Thus we have for all sufficiently large $t$ that
\begin{align*}
    U_3(t)
&\leq \sum_{r \geq r(t)(\log t)^p} \exp \left\{ -\frac{1}{10}r \log \left(\frac{r}{et}\right) \right\} \to 0
\end{align*}
as $t \to \infty$.
%Now, since for every $r\geq r(t)$ and $t$ large enough 
%\begin{align*}
%\exp\left\{-\frac{1}{5}(r+1)\log \left(\frac{r+1}{et}\right) \right\}&\leq \exp\left\{-\frac{1}{5}\log \left(\frac{r}{et}\right) \right\} \exp\left\{-\frac{1}{5}r\log \left(\frac{r}{et}\right) \right\}\\
%&\leq \frac{1}{1+\frac{1}{5}\log\left(\frac{r(t)}{et}\right)} \exp\left\{-\frac{1}{5}r\log \left(\frac{r}{et}\right) \right\}\\
%&\leq \frac{1}{5}  \exp\left\{-\frac{1}{2}r\log \left(\frac{r}{et}\right) \right\},
%\end{align*}
%we obtain
%\begin{align*}
%    U_3(t)
%&\leq \sum_{r \geq r(t)} \exp \left\{ -\frac{1}{5}r \log \left(\frac{r}{et}\right) \right\} \\&\leq \exp \left\{ -\frac{1}{5}r(t) \log \left(\frac{r(t)}{et}\right) \right\} \sum_{r\geq 0}\frac{1}{5^r}\\
%&\leq \frac{5}{4}\exp \left\{ -\frac{1}{5}r(t) \log \left(\frac{r(t)}{et}\right) \right\} \to 0,
%\end{align*}
%as $t \to \infty$.
Since $U(t) \geq 1$ eventually almost surely by Proposition \ref{prop:total mass LB}, the proof is complete. 
\end{proof}

\begin{prop}\label{prop:U4 zero}
Under Assumption \ref{assn:as}, $P$-almost surely as $t\to \infty$
$$\frac{U_4(t)}{U(t)} \to 0.$$
\end{prop}

\begin{proof} We proceed similarly to the proof of Proposition \ref{prop:U2 zero}. For each $u \in (F_t \cup E_t) \setminus \{\hat Z_t^{(1)},\hat Z_t^{(2)}\}$, let us define the following set of paths:
$$\hat{\Gamma}_t^{(4)}(u):= \left\{\gamma \subset B_{R(t)}: \arg \max_{v\in \gamma}[\xi(v)-\deg(v)]=u,|\gamma|\leq j_t\right\}.$$
Now fix some $\gamma=v_0, \ldots, v_n \in \bigcup_{u \in (F_t \cup E_t) \setminus \{\hat Z_t^{(1)},\hat Z_t^{(2)}\}} \hat \Gamma_{t}^{(4)}(u)$. Let $v^* = \arg_{v \in \gamma} \max (\xi(v) - \deg v)$ and let $i^* = \inf\{j \geq 1: v_j = v^*\}$. Take some $\delta < \frac{1}{2\alpha(q+1)}$ and consider $G_t$ as defined in \eqref{eqn:Gt def}. Then, using again \eqref{eqn:equationproduct} and the fact that $\gamma$ contains at least $\frac{n-|v_{i^*}|}{2}+|v_{i^*}|-t^{2\alpha(q+1)\delta}$ vertices in $G_t^c$ by Lemma \ref{lem:useful things for Ft and Gt}$(i)$, and $v_{i^*} \in B_{R(t)}$ by assumption, we can deduce exactly as in the proof of Proposition \ref{prop:U2 zero} that 
\begin{align*}
&\E{\exp \left\{ \int_0^t \xi (X_s) ds \right\}\mathbbm{1}\{{\pi}((X)_{[0,t]}) = \gamma\}} \\
&\leq \exp\left\{t\psi_t(v_{i^*})+t -\left(\frac{n-|v_{i^*}|}{2} - t^{2\alpha(q+1)\delta}\right)\log(\xi(v_{i^*})-\deg(v_{i^*})-\left(\frac{n+|v_{i^*}|}{2} - t^{2\alpha(q+1)\delta}\right) \log\left(1-R(t)^{\frac{-\delta}{2}}\right)\right\}.
\end{align*}
Therefore, since we chose $\delta < \frac{1}{2\alpha(q+1)}$ and $\log (1-x) \leq -x$ for $x \in (0,1)$, we have for any $\epsilon > 0$ that
\begin{align*}
    &\log(U_4(t))\\ &= \log\left(\sum_{|v_{i^*}| \in (F_t \cup E_t) \setminus \{\hat{Z}_t^{(1)}, \hat{Z}_t^{(2)}\}} \sum_{|v_{i^*}| \leq n \leq j_t} \sum_{\gamma \in \hat{\Gamma}_t^{(4)}(v_{i^*}): |\gamma|=n} \E{\exp \left\{ \int_0^t \xi (X_s) ds \right\}\mathbbm{1}\{{\pi}((X)_{[0,t]})= \gamma\}}\right) \\
     &\leq \sup_{|v_{i^*}| \in (F_t \cup E_t) \setminus \{\hat{Z}_t^{(1)}, \hat{Z}_t^{(2)}\}} \sup_{|v_{i^*}|\leq n}\left\{ t\psi_t(v_{i^*}) \right.\\
&\qquad \left.-\left(\frac{n-|v_{i^*}|}{2}\right)\log(\xi(v_{i^*})-\deg(v_{i^*})) + (n+|v_{i^*}|) R(t)^{\frac{-\delta}{2}} + \eta_{R(t)}(n,v_{i^*})\right\}+ O \left(  t \right).%\\
 %   &\leq \sup_{|v_{i^*}| \in (F_t \cup E_t) \setminus \{\hat{Z}_t^{(1)}, \hat{Z}_t^{(2)}\}} \sup_{|v_{i^*}|\leq n}  \left\{ t\psi_t(v_{i^*}) \right.\\
%&\qquad \left.-\left(\frac{n-|v_{i^*}|}{2}\right)\log(\xi(v_{i^*})-\deg(v_{i^*})+\left(n+|v_{i^*}|\right) R(t)^{\frac{-\delta}{2}} + (n-|v_{i^*}|)\left(\epsilon \log (R(t)) + \log \left( \frac{n}{n-|v_{i}|} \right)\right)\right\}
\end{align*}
Since $v_{i^*}\in F_t\cup E_t$ the same argument as in the proof of Proposition \ref{prop:U2 zero} yields that the maximum in the expression above is attained at the smallest possible value of $n$, provided that we choose $\epsilon > 0$ small enough when applying the upper bound in Lemma \ref{lem:no of paths}. Therefore, substituting $n = |v_{i^*}|$ we obtain that
\begin{equation}
    \log(U_4(t))\leq t\psi(v_{i^*})+O(R(t)^{1-\frac{\delta}{2}}).
\end{equation}
Invoking Propositions \ref{prop:gap a.s. LB} and \ref{prop:total mass LB} therefore yields that almost surely,
\begin{align*} 
\log\left(\frac{U_4(t)}{U(t)}\right) &\leq \sup_{v_{i^*} \in (F_t \cup E_t) \setminus \{\hat Z_t^{(1)},\hat Z_t^{(2)}\}} t(\psi_t(v_{i^*})-\psi_t(\hat Z_t^{(1)})) + r(t) t^{-c} + O(r(t)t^{-\frac{(q+1)\delta}{3}})
%&\leq  -t a(t)(\log t)^{-k}+ r(t) e^{-(\log t)^c} + o(r(t)^{1-\delta/8})\\
\to -\infty,
\end{align*}
as required.
\end{proof}

\begin{prop}\label{prop:U5}
Under Assumption \ref{assn:as}, $P$-almost surely as $t\to \infty$ $$\frac{U_5(t)}{U(t)} \to 0.$$
\end{prop}
\begin{proof}
Clearly, $U(5) \leq \exp \{t \sup_{u \in F_t} \xi(v)\}$. Taking $C$ as in \eqref{eqn:Ft def} and taking $\epsilon=\frac{1}{2}(C-\frac{pd}{\alpha})$ in Proposition \ref{prop:total mass LB}$(i)$, we deduce that
$$\log \left(\frac{{U}_5(t)}{U(t)}\right) \leq ta(t) [(\log t)^{-(C-\frac{pd}{\alpha})} - (\log t)^{-\epsilon}] \to -\infty$$ as required.
\end{proof}

\begin{prop}\label{prop:U6 0}
Under Assumption \ref{assn:as}, $P$-almost surely as $t\to \infty$, $$\frac{U_6(t)}{U(t)} \to 0.$$
\end{prop}
\begin{proof}
Note that, for any $\epsilon > 0$, $\sup_{v \in B_{R(t)}} \deg (v) \leq t^{\frac{d+\frac{\epsilon}{2}}{\beta}}$ eventually almost surely by Lemma \ref{lem:supdeg}$(i)$. Therefore,
\[
\p{J_t \geq t^{\frac{d+\epsilon}{\beta}+1}r(t)} \leq \p{\textsf{Poi}\left(t^{\frac{d+\frac{\epsilon}{2}}{\beta}+1}\right) \geq t^{\frac{d+\epsilon}{\beta}+1}r(t)} \leq e^{-ct^{\frac{\epsilon}{2}}r(t)}.
\]
Therefore, again applying Proposition \ref{lem:xi in Ar} we deduce that
\begin{align*}
\mathbb E_0\left[\exp\left\{\int_0^t\xi(X_s)ds\right\} \mathbbm 1\{\tau_{B_{R(t)}}>t, J_t\geq t^{\frac{d+\epsilon}{\beta}+1}r(t)\}\right] \leq \exp \{t\sup_{v \in B_{R(t)}} \xi (v) - ct^{\frac{\epsilon}{2}}r(t)\} \leq \exp \{r(t)t^{\frac{\epsilon}{3}} - ct^{\frac{\epsilon}{2}}r(t)\}.
\end{align*}
The result then follows from Proposition \ref{prop:total mass LB} similarly to the other cases.
\end{proof}

\begin{proof}[Proof of Theorem \ref{thm:main a.s. localisation intro}]
Since $U(t) - U_1(t) \leq \sum_{i=2}^5 U_i(t)$, the theorem is a direct consequence of Propositions \ref{prop:u1}-\ref{prop:U5}. 
\end{proof}

\section{High probability single site localisation}\label{sec:singlesitelocalisation}

Unlike in \cite{KLMStwocities09}, we cannot directly read off a single site localisation result from the almost sure result because stronger analogues of Propositions \ref{prop:U2 zero} and \ref{prop:U4 zero} for $U_2$ and $U_4$ do not necessarily hold with high probability under Assumption \ref{assn:whp}. In particular since the number of paths through a given vertex of length $n$ grow fairly fast (cf Lemma \ref{lem:no of paths}) we cannot write high probability analogues of these results since we cannot justify the restriction to paths that are ``as direct as possible''.

Instead we give a somewhat modified version of the proof, which nevertheless follows a similar structure to that used in \cite{CompletePreprint}. For this we keep $p= q+2$ and define 
\[
R(t) = r(t) (\log t)^p, \hspace{1cm} \overline{R}_t = |\hat{Z}_t^{(1)}| \left(1 + \frac{1}{(\log t)^{\frac{1}{2}(1 + \frac{1}{\alpha})}} \right).
\]
For ease of notation we define $h_t = (\log t)^{\frac{1}{2}(1 + \frac{1}{\alpha})}$ for the rest of this section. Similar to Section \ref{sec:localisation} we define
\begin{align*}
\tilde{u}_1(t,v)&=\mathbb E_0\left[\exp\left\{\int_0^t\xi(X_s) ds \right\}\mathbbm 1\{X_t=v\}\mathbbm 1\{\tau_{B_{\overline{R}_t}}> t\}\mathbbm 1\{H_{\hat Z_t^{(1)}}\leq t\}\right], \\
\tilde{U}_2(t)&=\mathbb E_0\left[\exp\left\{\int_0^t\xi(X_s)ds \right\}\mathbbm 1\{\tau_{B_{\overline{R}_t}}\leq t, \tau_{B_{R(t)}} > t\}\right],\\
\tilde{U}_3(t)&=\mathbb E_0\left[\exp\left\{\int_0^t\xi(X_s)ds \right\}\mathbbm 1\{\tau_{B_{R(t)}} \leq t\}\right],\\
\tilde{U}_4(t)&=\mathbb E_0\left[\exp\left\{\int_0^t\xi(X_s)ds \right\}\mathbbm 1\{\tau_{B_{\overline{R}_t}}> t\}\mathbbm 1\{H_{\hat Z_t^{(1)}}>t\}\right].
\end{align*}
We show that $\tilde{u}_1$ localises on $\hat{Z}_t^{(1)}$ whp, and that $\tilde{U}_2, \tilde{U}_3$ and $\tilde{U}_4$ are asymptotically negligible whp. For this we define a modified functional
\begin{equation}
\overline{\psi}_t(v) = t\xi(v) - |v| \log \left(\frac{|v|}{et}\right).
\end{equation}
Although $\overline{\psi}_t$ is ``less precise'' than $\psi_t$, the advantage of working with $\overline{\psi}_t$ is that it is easier to use it to give upper bounds $u(t,v)$ for different $v$. It is less of a natural lower bound for $u(t,v)$ since we did not incorporate $\rho$ as in \eqref{eqn:psit def}, but in order to show that $\exp \{t \hat{\psi}_t(\hat{Z}_t^{(1)})\}$ is almost a lower bound for $u(t,\hat{Z}_t^{(1)})$ one only has to condition on certain high probability events at the single site $\hat{Z}_t^{(1)}$. It would be harder to show that $\exp \{t \psi_t(v)\}$ is an asymptotic upper bound for any $v \neq \hat{Z}_t^{(1)}$, since this involves conditioning on certain events occurring for \textit{all} other $v \neq \hat{Z}_t^{(1)}$, which is much harder (and already we saw that the upper bounds in Propositions \ref{prop:U2 zero} and \ref{prop:U4 zero} do not carry through in the high probability setting).

\begin{prop}\label{prop:total mass LB whp}
Take any $\epsilon > 0$. Under Assumption \ref{assn:whp}, it holds with high $P$-probability as $t\to \infty$ that
\begin{align*}
    \log U(t) \geq t \overline{\psi}_t (\hat{Z}_t^{(1)}) + o\left(\frac{ta(t)}{(\log t)^{1-\epsilon}}\right).
\end{align*}
\end{prop}
\begin{proof}
By Proposition \ref{prop:total mass LB}$(ii)$, it is enough to show that
\[
\overline{\psi}_t (\hat{Z}_t^{(1)}) - \psi_t (\hat{Z}_t^{(1)}) = \deg (\hat{Z}_t^{(1)}) + \frac{|\hat{Z}_t^{(1)}|}{t}\log \left( \frac{(\xi(\hat{Z}_t^{(1)}) - \deg(\hat{Z}_t^{(1)}))te}{|\hat{Z}_t^{(1)}|} \right)  = o\left(\frac{a(t)}{(\log t)^{1-\epsilon}} \right)
\]
whp as $t \to \infty$. This is immediate from Proposition \ref{prop:Zt bounds as}$(iii')$ and Lemma \ref{cor:whppotbo} since whp we have:
\begin{align*}
\frac{|\hat{Z}_t^{(1)}|}{t}\log \left( \frac{(\xi(\hat{Z}_t^{(1)}) - \deg(\hat{Z}_t^{(1)}))te}{|\hat{Z}_t^{(1)}|} \right) \leq \frac{r(t)(\log t)^{\epsilon/2}}{t} \log \left(\frac{a(t) t \log t}{r(t)} \right) &= \frac{a(t)}{(\log t)^{1-\epsilon/2}} \log ((\log t)^2).
%&\leq \frac{2}{q+1} \frac{ta(t)}{(\log t)^{1-\epsilon/2}} \log \log t.
\end{align*}
A similar result holds for $\deg (\hat{Z}_t^{(1)})$ by Proposition \ref{prop:max degree Ft Et}$(ii)$ and since $\hat{Z}_t^{(1)} \in F_{t}$ whp by Proposition \ref{prop:Zt bounds as}$(ii')$.
\end{proof}

Additionally, exactly the same proof as in Lemma \ref{lem:probgap} gives the following (note that Lemma \ref{lem:probgap} was written only under Assumption \ref{assn:whp}, and since $\overline{\psi}_t(v) \geq \psi_t(v)$ for all $v$ the same almost sure lower bounds hold for $\overline{\psi}_t(v)$).

\begin{lem} \label{lem:probgap whp}
Take any $D>0$. Under Assumption \ref{assn:whp}, for any $0<K<D$, we have for all $t$ large enough that
\begin{align}
    \p{\overline{\psi}_t(\hat Z_t^{(1)})-\sup_{v \neq \hat Z_t^{(1)}}\overline{\psi}_t(v)<a(t) (\log t)^{-D}} &\leq 4\alpha (\log t)^{-K}.
\end{align}
\end{lem}

We will need the following result to complete the localisation proof for $\tilde{u}_1$.

\begin{lem}\label{lem:ZtequalZBR}
Under Assumption \ref{assn:whp},
\begin{enumerate}[(i)]
\item $\lim_{t\to \infty} P(\hat Z_t^{(1)}=Z_{B_{\overline{R}_t}}=\tilde{Z}_{B_{\overline{R}_t}})=1$.
\item $\p{\hat{Z}_t^{(1)} \text{ is also the maximiser of } \overline{\psi}_t} \to 1$.
\end{enumerate}
\end{lem}
\begin{proof} 
\begin{enumerate}[(i)]
\item By Lemma \ref{lem:equalZonB}, it is sufficient just to establish the first equality. Since $\psi_t(\hat{Z}_t^{(1)}) \geq \psi_t(Z_{B_{\overline{R}_t}})$ by definition, we can rearrange to deduce that 
\begin{align*}
\xi(Z_{B_{\overline{R}_t}}) - \xi (\hat{Z}_t^{(1)}) &\leq \deg (Z_{B_{\overline{R}_t}}) - \deg (\hat{Z}_t^{(1)}) + \frac{|Z_{B_{\overline{R}_t}}|}{t}\log (\xi(Z_{B_{\overline{R}_t}}) - \deg (Z_{B_{\overline{R}_t}})) - \frac{|\hat{Z}_t^{(1)}|}{t}\log (\xi(\hat{Z}_t^{(1)}) - \deg (\hat{Z}_t^{(1)})) \\
&\leq \deg (Z_{B_{\overline{R}_t}})  + \left(\frac{|Z_{B_{\overline{R}_t}}|}{t} - \frac{|\hat{Z}_t^{(1)}|}{t}\right)\log (\xi(Z_{B_{\overline{R}_t}}) - \deg (Z_{B_{\overline{R}_t}})) + \frac{|\hat{Z}_t^{(1)}|}{t}\log \left(\frac{\xi(Z_{B_{\overline{R}_t}}) - \deg (Z_{B_{\overline{R}_t}})}{\xi(\hat{Z}_t^{(1)}) - \deg (\hat{Z}_t^{(1)})}\right).% \\
%&\leq \deg (Z_{B_{\overline{R}_t}})  + \left(\frac{|Z^{(1)}_{t}|}{t(\log t)^{\frac{1}{2}(1+\frac{1}{\alpha})}}\right)\log (\xi(Z_{B_{\overline{R}_t}}) - \deg (Z_{B_{\overline{R}_t}})) + \frac{|\hat{Z}_t^{(1)}|}{t}\log \left(\frac{\xi(Z_{B_{\overline{R}_t}}) - \deg (Z_{B_{\overline{R}_t}})}{\xi(\hat{Z}_t^{(1)}) - \deg (\hat{Z}_t^{(1)})}\right).
\end{align*}
Now take any $\epsilon \in \left(0, \frac{1}{9}(1+\frac{1}{\alpha})\right)$ and note that:
\begin{itemize}
\item By Lemma \ref{cor:whppotbo} and Lemma \ref{prop:max degree Ft Et}$(ii)$, $\deg (Z_{B_{\overline{R}_t}}) \leq (\log t)^B$ whp as $t \to \infty$.
\item By Proposition \ref{prop:Zt bounds as}$(iii')$, $\hat{Z}_t^{(1)} \leq r(t)(\log t)^{\epsilon} =  ta(t)(\log t)^{\epsilon-1}$ whp as $t \to \infty$.
\item By Lemma \ref{cor:whppotbo} $\xi(Z_{B_{\overline{R}_t}}) - \deg (Z_{B_{\overline{R}_t}}) \leq R(t)(\log t)^{\epsilon} = r(t)(\log t)^{p+\epsilon}$ whp as $t \to \infty$.
\item By Proposition \ref{prop:Zt bounds as}$(i')$ and $(ii')$, $\xi(\hat{Z}_t^{(1)}) - \deg (\hat{Z}_t^{(1)}) \geq a(t)(\log t)^{-\epsilon}$ whp as $t \to \infty$.
\end{itemize}
Substituting these bounds above, and using the fact that we chose $\epsilon < \frac{1}{9}(1+\frac{1}{\alpha})$, we deduce that, whp as $t \to \infty$,
\begin{align*}
\xi(Z_{B_{\overline{R}_t}}) - \xi (\hat{Z}_t^{(1)}) 
\leq (\log t)^B  + a(t)(\log t)^{\epsilon-\frac{1}{2}(1+\frac{1}{\alpha})} + (p+2\epsilon)a(t)(\log t)^{\epsilon - 1}\log \log t &\leq a(t)(\log t)^{\epsilon-\frac{1}{3}(1+\frac{1}{\alpha})} \\
&< a(t) (\log t)^{-2\epsilon}.
\end{align*}
By Lemma \ref{lem:gap to infinity} and Proposition \ref{prop:Zt bounds as}$(i')$, we also have that $g_{B_{\overline{R}_t}} \geq (\overline{R}_t)^{\frac{d}{\alpha}} (\log (\overline{R}_t))^{-\epsilon} \geq a(t) (\log t)^{-2\epsilon}$ whp as $t \to \infty$. 

Combining these two calculations, we deduce that $\xi(Z_{B_{\overline{R}_t}}) - \xi (\hat{Z}_t^{(1)}) < g_{B_{\overline{R}_t}}$ whp, which in turn implies that $B_{\overline{R}_t} = \hat{Z}_t^{(1)}$ whp.
\item By part $(i)$, it follows that $\overline{\psi}_t(\hat{Z}_t) \geq \overline{\psi}_t(v)$ for all $v \in B_{\overline{R}_t} \setminus B_{|\hat{Z}_t|}$ whp. Additionally, by Proposition \ref{prop:Zt bounds as}$(ii')$ and Lemma \ref{cor:whppotbo}, it follows that $\overline{\psi}_t(\hat{Z}_t) \geq \overline{\psi}_t(v)$ for all $v \in B_{r(t)(\log t)^{-\epsilon}}$. Finally, by the same proof used to prove the upper bound in Theorem \ref{thm:Zt asymp intro}, it follows that the maximiser of $\overline{\psi}_t$ is in $B_{r(t)(\log t)^{\epsilon}}$.

Now suppose that both $v$ and $\hat{Z}_t^{(1)}$ are in $B_{r(t)(\log t)^{\epsilon}} \setminus B_{r(t)(\log t)^{-\epsilon}}$ and $v \neq \hat{Z}_t$. Then, since in particular $|v| \leq |\hat{Z}_t^{(1)}|(\log t)^{2\epsilon}$, we have that
\begin{align*}
\overline{\psi}_t(v) - \overline{\psi}_t(\hat{Z}_t^{(1)}) &\leq {\psi}_t(v) - {\psi}_t(\hat{Z}_t^{(1)}) +\frac{|v|}{t}\log \left( \frac{t(\xi
(v) - \deg v)}{t}\right) - \frac{|\hat{Z}_t^{(1)}|}{t}\log \left( \frac{t(\xi
(\hat{Z}_t^{(1)}) - \deg (\hat{Z}_t^{(1)}))}{t}\right) \\
&\leq {\psi}_t(v) - {\psi}_t(\hat{Z}_t^{(1)}) +\frac{|\hat{Z}_t^{(1)}|(\log t)^{2\epsilon}}{t}\log \left( \frac{|\hat{Z}_t^{(1)}|(\xi(v) - \deg v)}{|v|(\xi(\hat{Z}_t^{(1)})) - \deg (\hat{Z}_t^{(1)})}\right).
\end{align*}
Take $\epsilon \in (0, \frac{1}{6})$. By Lemma \ref{cor:whppotbo} and Proposition \ref{prop:Zt bounds as}$(ii)$, we can assume that $\xi(\hat{Z}_t^{(1)}) - \deg (\hat{Z}_t^{(1)}) \geq (\xi(v) - \deg v)(\log t)^{3\epsilon}$, and since $v \in B_{r(t)(\log t)^{\epsilon}} \setminus B_{r(t)(\log t)^{-\epsilon}}$, we can also assume that $\frac{|\hat{Z}_t^{(1)}|}{|v|} \leq (\log t)^{2 \epsilon}$ and $|v| \leq r(t)(\log t)^{\epsilon}$ by Proposition \ref{prop:Zt bounds as}$(iii')$. Moreover, by Lemma \ref{lem:probgap}, we can assume that ${\psi}_t(v) - {\psi}_t(\hat{Z}_t^{(1)}) \leq -a(t)(\log t)^{-\epsilon}$. Assuming these, we therefore have that
\begin{align*}
\overline{\psi}_t(v) - \overline{\psi}_t(\hat{Z}_t^{(1)}) 
&\leq {\psi}_t(v) - {\psi}_t(\hat{Z}_t^{(1)}) +\frac{r(t)(\log t)^{3\epsilon}}{t}\log \log t \leq -a(t)(\log t)^{-\epsilon} + a(t)(\log t)^{4\epsilon-1} <0,
\end{align*}
as required.
\end{enumerate}
\end{proof}

\begin{prop}
Under Assumption \ref{assn:whp}, 
\begin{equation}\label{eqn:u1 whp}
\frac{\sum_{v\in \Ti\setminus\{\hat Z_t^{(1)}\}}\tilde u_1(t,v)}{U(t)}\to 0
\end{equation}
in $P$-probability as $t \to \infty$.
\end{prop}
\begin{proof}
This follows from Lemmas \ref{lem:equalZonB}, \ref{lem:u efunction comp}, \ref{lem:loceigenfunction whp} and \ref{lem:ZtequalZBR}.
\end{proof}

\begin{prop}\label{prop:U2 whp 0}
Under Assumption \ref{assn:whp}, $\frac{\tilde{U}_2(t)}{U(t)} \to 0$ whp.
\end{prop}
\begin{proof}
We will show that
\begin{align}\label{eqn:u_1 log bound}
    \log \left(\tilde{U}_2(t) \right) &\leq \max \left\{ t\sup_{v \neq \hat Z_t^{(1)}}\overline{\psi}_t(v), t\xi(\hat{Z}_t^{(1)}) - \overline{R}_t \log \left( \frac{\overline{R}_t}{et}\right) \right\} + o\left(\frac{ta(t)}{\log t}\right)
\end{align}
whp as $t \to \infty$. To do this, first note that by the definition of $\tilde{U}_2$, Proposition \ref{prop:Zt bounds as}$(i')$ and Lemma \ref{lem:RW ball exit time restricted radii}, we have with high probability that
\begin{align*}
    \log \left(\tilde{U}_2(t) \right) &\leq \log \left( \sum_{r \in [\overline{R}_t, r(t) (\log t)^p]} \exp \left\{t \sup_{v \in B_r} \xi(v)\right\} \pr{\tau_{B_{r-1}} \leq t} \right) \\
%    &\leq \max_{r \in [\overline{R}_t, r(t) (\log t)^p]} \left\{ t\sup_{v \in B_r} \xi (v) + \log \left(\pr{\tau_{B_{r-1}} \leq t}\right) \right\} + \log \left( r(t) (\log t)^p \right) \\
    &\leq t\max_{r \in [\overline{R}_t, r(t) (\log t)^p]} \left\{ \sup_{v \in B_r} \xi (v) - \frac{r}{t} \log \left(\frac{r}{et}\right)\right\} + o\left(\frac{ta(t)}{\log t} \right).
\end{align*}
Let $\hat{r}(t)$ denote the $r$ for which the maximum is attained, and $\hat{v}(t) = \arg \max \{\sup \xi(v): v \in B_{\hat{r}(t)}\}$. If $\hat{v}(t) = \hat{Z}_t^{(1)}$, then
\begin{align*}
t\max_{r \in [\overline{R}_t, r(t) (\log t)^p]} \left\{ \sup_{v \in B_r} \xi (v) - \frac{r}{t} \log \left(\frac{r}{et}\right)\right\} &\leq t\xi(\hat{Z}_t^{(1)}) - \overline{R}_t \log \left(\frac{\overline{R}_t}{et}\right).
\end{align*}
If $\hat{v}(t) \neq \hat{Z}_t^{(1)}$, then
\begin{align*}
t\max_{r \in [\overline{R}_t, r(t) (\log t)^p]} \left\{ \sup_{v \in B_r} \xi (v) - \frac{r}{t} \log \left(\frac{r}{et}\right)\right\} = t\xi(\hat{v}(t)) - \hat{r}(t) \log \left(\frac{\hat{r}(t)}{et}\right) &\leq t\xi(\hat{v}(t)) - |\hat{v}(t)| \log \left(\frac{|\hat{v}(t)|}{et}\right) \\
&\leq t\sup_{v \neq \hat Z_t^{(1)}}\overline{\psi}_t(v).
\end{align*}
This establishes \eqref{eqn:u_1 log bound}. To complete the proof, we invoke Proposition \ref{prop:total mass LB whp} which implies that, for any $\epsilon > 0$, we have with high probability as $t\to\infty$,
\begin{align*}
&\frac{\tilde{U}_2(t)}{U(t)} \leq \exp\left\{\max \left\{ t\left(\sup_{v \neq \hat Z_t^{(1)}}\overline{\psi}_t(v)-\overline{\psi}_t(\hat{Z}_t^{(1)}) \right), |\hat{Z}_t^{(1)}|\log \left(\frac{|\hat{Z}_t^{(1)}|}{et}\right)- \overline{R}_t \log \left( \frac{\overline{R}_t}{et}\right) \right\} + o\left(\frac{ta(t)}{(\log t)^{1 - \epsilon}}\right) \right\}.
\end{align*}
\begin{itemize}
\item To deal with the first term in the maximum above, take $\epsilon = \frac{1}{4}$ and note from Lemma \ref{lem:probgap whp} that for $\eta_t  = (\log t)^{-\frac{1}{2}}$ we have that $P \left (\overline{\psi}_t(\hat Z_t^{(1)})-\sup_{v \neq \hat Z_t^{(1)}}\overline{\psi}_t(v)<\eta_t a(t)\right) \to 0$ as $t\to \infty$. Therefore,
\[
t\left( \sup_{v \neq \hat Z_t^{(1)}}\overline{\psi}_t(v) - \overline{\psi}_t(\hat{Z}_t^{(1)})\right) + o\left(\frac{ta(t)}{(\log t)^{3/4}}\right) \to -\infty
\]
in probability as $t \to \infty$.
\item To deal with the second term in the maximum above, recall that $\overline{R}_t = |\hat{Z}^{(1)}_t| (1+h_t)$ where $h_t = \frac{1}{(\log t)^{\frac{1}{2}(1 + \frac{1}{\alpha})}}$, and choose $\varepsilon$ small enough such that $h_t \geq (\log t)^{-(1-2\varepsilon)}$. We then have by Proposition \ref{prop:Zt bounds as}$(i')$ that, with high probability as $t \to \infty$
\begin{align*}
\overline{R}_t \log \left( \frac{\overline{R}_t}{et}\right) - |\hat{Z}^{(1)}_t| \log \left( \frac{|\hat{Z}^{(1)}_t|}{et}\right) + o\left(\frac{ta(t)}{(\log t)^{1-\epsilon}}\right) &\geq |\hat{Z}^{(1)}_t| h_t \log \left( \frac{|\hat{Z}^{(1)}_t|}{et}\right) + o\left(\frac{ta(t)}{(\log t)^{1-\epsilon}}\right) \\
&\geq C|\hat{Z}^{(1)}_t| h_t  \log t + o\left(\frac{ta(t)}{(\log t)^{1-\epsilon}}\right)\\
&\geq Cta(t) \frac{(\log t)^{-\epsilon}}{ (\log t)^{1-2\varepsilon}} + o\left(\frac{ta(t)}{(\log t)^{1-\epsilon}}\right),
\end{align*}
which diverges as $t \to \infty$.
\end{itemize}
This completes the proof. 

\end{proof}

\begin{prop}
Under Assumption \ref{assn:whp}, $\frac{\tilde{U}_3(t)}{U(t)} \to 0$ whp.
\end{prop}
\begin{proof}
This holds by Proposition \ref{prop:U3 0} (note that we proved it under the weaker assumption), since we chose the same value of $p$ when defining $R(t)$, and in particular $p>q+1$.
\end{proof}

\begin{prop}
Under Assumption \ref{assn:whp}, $\frac{\tilde{U}_4(t)}{U(t)} \to 0$ whp.
\end{prop}
\begin{proof}
We split the sum into two parts by writing,
\begin{align}\label{eqn:sumu2}
\tilde{U}_4(t) &\leq\mathbb E\left[\exp\left\{\int_0^t\xi(X_s)d s\right\}\mathbbm 1\{H_{\hat Z_t^{(1)}}>t\}\mathbbm 1\{\tau_{B_{\overline{R}_t}}>t\}\right]\nonumber\\
    &\leq \sum_{r\leq \overline{R}_t}\mathbb E\left[\exp\left\{\int_0^t\xi(X_s)d s\right\}\mathbbm 1\{H_{\hat Z_t^{(1)}}>t\}\mathbbm 1\{\sup_{s\leq t}|X_s|=r\}\right]\nonumber\\
    &=\sum_{r\leq r(t)(\log t)^{-1}}\mathbb E\left[\exp\left\{\int_0^t\xi(X_s)d s\right\}\mathbbm 1\{H_{\hat Z_t^{(1)}}>t\}\mathbbm 1\{\sup_{s\leq t}|X_s|=r\}\right]\nonumber\\
    &\qquad +\sum_{r(t)(\log t)^{-1}\leq r\leq \overline{R}_t}\mathbb E\left[\exp\left\{\int_0^t\xi(X_s)d s\right\}\mathbbm 1\{H_{\hat Z_t^{(1)}}>t\}\mathbbm 1\{\sup_{s\leq t}|X_s|=r\}\right].
\end{align}
We now bound the two sums appearing here.
\begin{itemize}
\item By Lemma \ref{cor:whppotbo} it holds with high $P$-probability as $t\rightarrow \infty$ that the logarithm of the first sum in \eqref{eqn:sumu2} is upper bounded by
\begin{align*}
%    &\log\left(\sum_{r\leq r(t)(\log t)^{-1}}\mathbb E\left[\exp\left\{\int_0^t\xi(X_s)d s\right\}\mathbbm 1\{H_{\hat Z_t^{(1)}}>t\}\mathbbm 1\{\sup_{s\leq t}|X_s|=r\}\right]\right)\\
%    \log\left(\sum_{r\leq r(t)(\log t)^{-1}}\exp\left\{t\sup_{v\in B_r\setminus\{\hat Z_t^{(1)}\}}\xi(v)\right\}\right) \leq  
    t\sup_{v\in B_{r(t)(\log t)^{-1}}}\xi(v)+\log(r(t)(\log \log t)^{-1}) =o\left(\frac{ta(t)}{(\log t)^{d/2\alpha}}\right).
\end{align*}
\item We know by Proposition \ref{prop:Zt bounds as} that for any $\epsilon > 0$, it holds with high $P$-probability as $t\to \infty$ that
$$\overline{R}_t \leq r(t)(\log t)^{\epsilon}.$$
Thus by Lemma \ref{lem:RW ball exit time restricted radii}, it holds whp that the logarithm of the second sum in \eqref{eqn:sumu2} is upper bounded by
\begin{align*}
%    &\log\left(\sum_{r(t)(\log t)^{-1}\leq r\leq \overline{R}_t}\mathbb E\left[\exp\left\{t\sup_{v\in B_r\setminus \{\hat Z_t^{(1)}\}}\xi(v)\right\}\mathbbm 1\{\sup_{s\leq t}|X_s|=r\}\right]\right)\\
% &\log\left(\sum_{r(t)(\log t)^{-1}\leq r\leq r(t)(\log t)^p}\exp\left\{t\sup_{v\in B_r\setminus \{\hat Z_t^{(1)}\}}\xi(v)\right\} \mathbb P(\tau_{B_{r-1}}\leq t)\right)\\
    &\max_{\substack{r\in [r(t)(\log t)^{-1},\\r(t)(\log t)^p]}}\left\{t\sup_{v\in B_{r}\setminus \{\hat Z_t^{(1)}\}}\xi(v)+\log(\mathbb P(\tau_{B_{r-1}}\leq t))\right\}+\log(r(t)(\log(t))^p)\\
        &\leq \max_{\substack{r\in [r(t)(\log t)^{-1},\\r(t)(\log t)^p]}}\left\{t\sup_{v\in B_{r}\setminus \{\hat Z_t^{(1)}\}}\xi(v)-r\log\left(\frac{r}{et}\right)+o(r(t))\right\}+\log(r(t)(\log(t))^p) \leq t\sup_{v \neq \hat{Z}_t^{(1)}}\overline{\psi}_t(v) + o(r(t)).
    \end{align*}
\end{itemize}
Combining these two calculations and applying Proposition \ref{prop:total mass LB whp} (with $\epsilon=1/2$) and then Lemmas \ref{lem:probgap whp} and \ref{lem:ZtequalZBR}$(ii)$ shows that whp as $t\rightarrow \infty$,
\begin{align*}\label{eqn:logsumu2}
    \log\left(\frac{\tilde{U}_2(t)}{U(t)}\right)\leq t\sup_{v \neq \hat{Z}_t^{(1)}}\{\overline{\psi}_t(v) - \overline{\psi}_t(\hat{Z}_t^{(1)})\}+o\left(\frac{ta(t)}{(\log t)^{d/2\alpha}} + r(t) + \frac{ta(t)}{(\log t)^{1/2}}\right) &\leq -\frac{ta(t)}{(\log t)^{d/4\alpha}} + o\left(\frac{ta(t)}{(\log t)^{d/2\alpha}}\right) \\
    &\to -\infty,
\end{align*}
as required.
\end{proof}

\bibliographystyle{alpha}
\bibliography{biblio}

\begin{thebibliography}{KLMS09}

\bibitem[AD15]{AbDelGWIntro}
R.~{Abraham} and J.-F. {Delmas}.
\newblock {An introduction to Galton-Watson trees and their local limits}.
\newblock {\em ArXiv e-prints}, 1506.05571, June 2015.

\bibitem[AGH20]{AvenaGunHesse2020parabolic}
L.~Avena, O.~G{\"u}n, and M.~Hesse.
\newblock The parabolic {A}nderson model on the hypercube.
\newblock {\em Stochastic Processes and their Applications}, 130(6):3369--3393,
  2020.

\bibitem[Ald91]{AldousCRTII}
D.~Aldous.
\newblock The {C}ontinuum {R}andom {T}ree {II}: an overview.
\newblock {\em Stochastic analysis}, 167:23--70, 1991.

\bibitem[Ald97]{Aldous1997brownian}
D.~Aldous.
\newblock Brownian excursions, critical random graphs and the multiplicative
  coalescent.
\newblock {\em The Annals of Probability}, pages 812--854, 1997.

\bibitem[And58]{Anderson1958Absence}
P.~Anderson.
\newblock Absence of diffusion in certain random lattices.
\newblock {\em Physical review}, 109(5):1492, 1958.

\bibitem[Arc20]{ArcherRWDecGWTrees}
E.~Archer.
\newblock Random walks on decorated {G}alton-{W}atson trees.
\newblock 2020.

\bibitem[Arc21]{archer2021brownian}
E.~Archer.
\newblock Brownian motion on stable looptrees.
\newblock In {\em Annales de l'Institut Henri Poincar{\'e}, Probabilit{\'e}s et
  Statistiques}, volume~57, pages 940--979. Institut Henri Poincar{\'e}, 2021.

\bibitem[CK08]{CroydonKumagai08}
D.~Croydon and T.~Kumagai.
\newblock Random walks on {G}alton-{W}atson trees with infinite variance
  offspring distribution conditioned to survive.
\newblock {\em Electron. J. Probab.}, 13:no. 51, 1419--1441, 2008.

\bibitem[dHW21]{HollanderWang}
F.~den Hollander and D.~Wang.
\newblock The {P}arabolic {A}nderson {M}odel on a {G}alton-{W}atson tree
  revisited, 2021.

\bibitem[dKd20]{HollanderKoenigSantosPAMOnTrees}
F.~{den Hollander}, W.~{K{\"o}nig}, and R.~{dos Santos}.
\newblock {The parabolic Anderson model on a Galton-Watson tree}.
\newblock {\em arXiv e-prints}, page arXiv:2001.05106, January 2020.

\bibitem[Duq09]{DuqSinTree}
T.~Duquesne.
\newblock Continuum random trees and branching processes with immigration.
\newblock {\em Stochastic Process. Appl.}, 119(1):99--129, 2009.

\bibitem[Dwa69]{Dwass1969}
M.~Dwass.
\newblock The total progeny in a branching process and a related random walk.
\newblock {\em Journal of Applied Probability}, 6(3):682--686, 1969.

\bibitem[FM90]{FleischmannMolchanov1990exact}
K.~Fleischmann and S.~Molchanov.
\newblock Exact asymptotics in a mean field model with random potential.
\newblock {\em Probability theory and related fields}, 86(2):239--251, 1990.

\bibitem[GK98]{GeigerKersting1998galton}
J.~Geiger and G.~Kersting.
\newblock The {G}alton--{W}atson tree conditioned on its height.
\newblock {\em Probability Theory and Mathematical Statistics (Vilnius, 1998)},
  pages 277--286, 1998.

\bibitem[GK05]{GaertnerKoenig2005survey}
J.~G{\"a}rtner and W.~K{\"o}nig.
\newblock The parabolic {A}nderson model.
\newblock In {\em Interacting stochastic systems}, pages 153--179. Springer,
  2005.

\bibitem[GM90]{GaertnerMolchanov}
J.~G{\"a}rtner and S.~Molchanov.
\newblock Parabolic problems for the {A}nderson model.
\newblock {\em Communications in mathematical physics}, 132(3):613--655, 1990.

\bibitem[Kes86]{KestenIICtree}
H.~Kesten.
\newblock Subdiffusive behavior of random walk on a random cluster.
\newblock {\em Ann. Inst. H. Poincar\'e Probab. Statist.}, 22(4):425--487,
  1986.

\bibitem[KLMS09]{KLMStwocities09}
W.~K{\"o}nig, H.~Lacoin, P.~M\"{o}rters, and N.~Sidorova.
\newblock A two cities theorem for the parabolic {A}nderson model.
\newblock {\em Ann. Probab.}, 37(1):347--392, 2009.

\bibitem[KMS06]{CompletePreprint}
W.~K{\"o}nig, P.~M{\"o}rters, and N.~Sidorova.
\newblock Complete localisation in the parabolic {A}nderson model with
  {P}areto-distributed potential.
\newblock {\em arXiv preprint math/0608544}, 2006.

\bibitem[K{\"o}n16]{KoenigBook16}
W.~K{\"o}nig.
\newblock {\em The parabolic {A}nderson model}.
\newblock Pathways in Mathematics. Birkh\"{a}user/Springer, [Cham], 2016.
\newblock Random walk in random potential.

\bibitem[Kor12]{Kort}
I.~Kortchemski.
\newblock Invariance principles for {G}alton--{W}atson trees conditioned on the
  number of leaves.
\newblock {\em Stochastic Processes and Their Applications}, 122(9):3126--3172,
  2012.

\bibitem[Kor17]{KortSubexp}
I.~Kortchemski.
\newblock Sub-exponential tail bounds for conditioned stable
  {B}ienaym\'{e}-{G}alton-{W}atson trees.
\newblock {\em Probab. Theory Related Fields}, 168(1-2):1--40, 2017.

\bibitem[M{\"o}r09]{Moerters2009survey}
P.~M{\"o}rters.
\newblock The parabolic {A}nderson model with heavy-tailed potential.
\newblock In {\em Surveys in Stochastic Processes, Proceedings of the 33rd SPA
  Conference in Berlin}, 2009.

\bibitem[MP16]{MuirheadPymarBAMLocalisation}
S.~Muirhead and R.~Pymar.
\newblock Localisation in the bouchaud--anderson model.
\newblock {\em Stochastic Processes and their Applications},
  126(11):3402--3462, 2016.

\bibitem[NP20]{NachmiasPeres2020local}
A.~Nachmias and Y.~Peres.
\newblock The local limit of uniform spanning trees.
\newblock 2020.

\bibitem[OR16]{OrtgieseRobertsIntermittency}
M.~Ortgiese and M.~Roberts.
\newblock Intermittency for branching random walk in {P}areto environment.
\newblock {\em The Annals of Probability}, 44(3):2198--2263, 2016.

\bibitem[Sla68]{SlackInfVar}
R.~Slack.
\newblock A branching process with mean one and possibly infinite variance.
\newblock {\em Z. Wahrscheinlichkeitstheorie und Verw. Gebiete}, 9:139--145,
  1968.

\end{thebibliography}
%\printbibliography

\end{document}